\documentclass[11pt, reqno]{amsart}

\usepackage{amsmath,amssymb,amsthm,epsfig}
\usepackage{hyperref}
\usepackage{amsmath}
\usepackage{amssymb}
\usepackage{color}
\usepackage{ulem}
\usepackage{epsfig}
\usepackage[mathscr]{eucal}

\usepackage[utf8]{inputenc}

\usepackage{stackrel}

\usepackage{epsfig}

\newtheorem{theorem}{Theorem}
\newtheorem{definition}{Definition}

\newtheorem{lemma}{Lemma}

\newtheorem{corollary}{Corollary}
\newtheorem{remark}{Remark}
\newtheorem{example}{Example}
\date{}
\numberwithin{equation}{section}
\numberwithin{theorem}{section}
\numberwithin{lemma}{section}
\numberwithin{corollary}{section}
\numberwithin{remark}{section}
\numberwithin{proposition}{section}
\numberwithin{definition}{section}

\def \Div {\mathrm{div}}

\def \R {\mathbb{R}}

\def \diam {\mathrm{diam}}

\def \dist {\mathrm{dist}}

\newcommand{\tr}{\mathrm{Tr}}
\newcommand{\defeq}{\mathrel{\mathop:}=}

\begin{document}

\title[Global regularity for unbalanced degenerate models]{Global regularity for a class of fully nonlinear PDEs with unbalanced variable degeneracy}

\author[J.V. da Silva]{Jo\~{a}o Vitor da Silva}
\address{Departamento de Matem\'atica \hfill\break \indent Instituto de Matem\'{a}tica, Estast\'{i}stica e Ci\^{e}ncia da Computa\c{c}\~{a}o \hfill\break \indent Universidade Estadual de Campinas - UNICAMP.
\hfill\break \indent Cidade Universit\'{a}ria Zeferino Vaz, CEP: 13083-591, Campinas - SP - Brazil.}
\email{jdasilva@unicamp.br}

\author[E.C.B. J\'{u}nior]{Elzon C.B. J\'{u}nior}
\address{Departamento de Matem\'atica \hfill\break \indent Universidade Federal do Cear\'{a}\hfill\break \indent Campus do Pici, CEP: 60455-760 Fortaleza-CE, Brazil}
\email{bezerraelzon@gmail.com}

\author[G.C. Rampasso]{Giane C. Rampasso}
\address{Departamento de Matem\'atica \hfill\break \indent Instituto de Matem\'{a}tica, Estast\'{i}stica e Ci\^{e}ncia da Computa\c{c}\~{a}o \hfill\break \indent Universidade Estadual de Campinas - UNICAMP.
\hfill\break \indent Cidade Universit\'{a}ria Zeferino Vaz, CEP: 13083-591, Campinas - SP - Brazil.}
\email{gianecr@unicamp.br}

\author[G.C. Ricarte]{Gleydson C. Ricarte}
\address{Departamento de Matem\'atica \hfill\break \indent Universidade Federal do Cear\'{a}\hfill\break \indent Campus do Pici, CEP: 60455-760 Fortaleza-CE, Brazil}
\email{ricarte@mat.ufc.br}

\begin{abstract}

We establish the existence and sharp global regularity results ($C^{0, \gamma}$, $C^{0, 1}$ and $C^{1, \alpha}$ estimates) for a class of fully nonlinear elliptic PDEs with unbalanced variable degeneracy. In a precise way, the degeneracy law of the model switches between two different kinds of degenerate elliptic operators of variable order, according to the null set of a modulating function $\mathfrak{a}(\cdot)\ge 0$. The model case in question is given by
$$
\left\{
\begin{array}{rcrcl}
    \left[|Du|^{p(x)}+\mathfrak{a}(x)|Du|^{q(x)}\right]\mathscr{M}_{\lambda, \Lambda}^{+}(D^2 u)& = & f(x) & \text{in} & \Omega \\
    u(x) & = & g(x) & \text{on} & \partial \Omega.
\end{array}
\right.
$$
for a bounded, regular and open set $\Omega \subset \R^n$, and appropriate continuous data $p(\cdot), q(\cdot)$, $f(\cdot)$ and $g(\cdot)$. Such sharp regularity estimates generalize and improve, to some extent, earlier ones via geometric treatments. Our results are consequences of geometric tangential methods and make use of compactness, localized oscillating and scaling techniques. In the end, our findings are applied in the study of a wide class of nonlinear models and free boundary problems.

\bigskip

\bigskip

\tableofcontents
\end{abstract}
\maketitle
\newpage

\section{Introduction}\label{s1}

Investigations on existence/boundedness/regularity issues involving non-uniformly elliptic operators as follows
$$
\mathscr{L} w \defeq -\Div(\partial_{\xi} F(x, \nabla w)) = f(x) \quad \text{in}\quad \Omega \subset \R^n,
$$
whose ``bulk energy density'' does not fulfil a standard $p$-power growth condition, i.e., a sort of double-sided bound of the type
$$
|\xi|^p\lesssim F(x, \xi)\lesssim |\xi|^p+1, \qquad \text{for} \,\,\,p>1, \,\,\,\text{and}\,\,\,|\xi| \,\,\,\text{large}
$$
are currently very much a classical topic, and they have attracted increasing attention of many researchers in the last decades (cf. \cite{LU68} for a Ladyzhenskaya-Ural'tseva's long-established treatise). Particularly, several results were developed to functionals with nonstandard growth conditions of $(p, q)$-type
$$
|\xi|^p\lesssim F(x, \xi)\lesssim |\xi|^q+1, \qquad \text{for} \,\,\, 1<p\le q, \,\,\,\text{and}\,\,\,|\xi| \,\,\,\text{large},
$$
which are nowadays known as \textit{Autonomous Functional} with $(p, q)-$growth.

An archetypical example of such functionals is given by multi-phase variational integrals
{\scriptsize{
\begin{equation}\label{Eq-Functional}
 W_{\text{loc}}^{1, p}(\Omega)\ni w \mapsto \mathscr{F}_0(w, \Omega) \defeq \int_{\Omega} \left(|Dw|^p+ \sum_{i=1}^{n} |Dw|^{q_i}-fw\right)dx, \quad 1<p<q_1\le \cdots\le q_n,
\end{equation}}}which were firstly examined in the Uraltseva-Urdaletova's groundbreaking work \cite{UU83}, whose special case of bounded minimizers was studied. Additionally, for a regularity treatment, we must quote Marcellini's pioneering work \cite{Marc89}, which proved that minimizers of \eqref{Eq-Functional} (for $f \equiv 0$) belong to $H_{\text{loc}}^{1, \infty}(\Omega)$.
\bigskip

Recently, some authors have introduced a new class of functionals, whose simplest model case is given by
{\scriptsize{
\begin{equation}\label{DP-Eq}
W_{\text{loc}}^{1, p}(\Omega)\ni w \mapsto \mathscr{F}(w, \Omega) \defeq \int_{\Omega} \mathscr{H}_0(x, \nabla w) dx, \quad \text{for} \quad \mathscr{H}_0(x, \xi) \defeq |\xi|^p+ \mathfrak{a}(x)|\xi|^{q},
\end{equation}}}where $0 \le \mathfrak{a} \in C^{0, \alpha}(\Omega)$ and $p$ and $q$ are positive constants  such that
\begin{equation}\label{Eq_Conditions}
  1<p\le q \quad \text{and}\quad \frac{q}{p}< 1 + \frac{\alpha}{n} \quad \text{with} \quad \alpha \in (0, 1].
\end{equation}
It worth stress that \eqref{DP-Eq} changes its ellipticity rate according to the geometry of the zero level set of the modulating function $\mathcal{Z}(\mathfrak{a}, \Omega) \defeq \{\mathfrak{a}(x) = 0\}$. For this reason, \eqref{DP-Eq} is so-termed as \textit{Functional of Double Phase} type.

As a matter of fact, by considering a wide class of double phase functionals
\begin{equation}\label{Funct-Col-Ming}
  w \mapsto \mathscr{G}(w, \Omega) \defeq \int_{\Omega} \mathscr{H}(x, w, \nabla w) dx,
\end{equation}
where $\mathscr{H}: \R^n \times \R \times \R^n \to \R$ is a Carath\'{e}odory function fulfilling a growth condition (and under natural hypothesis \eqref{Eq_Conditions})
$$
L_1\mathscr{H}_0(x, \xi) \le \mathscr{H}(x, w, \xi) \le L_2\mathscr{H}_0(x, \xi) \quad \text{for some constants} \quad 0<L_1\le L_2,
$$
in a series of seminal works, Colombo-Mingione in \cite{CM15} and \cite{CM15II} and Baroni-Colombo-Mingione in \cite{BCM15III} proved many boundedness and local sharp regularity results ($D w \in C_{\text{loc}}^{0, \beta}(\Omega)$ for a universal exponent) for local minimizers of \eqref{Funct-Col-Ming}. We also recommend De Filippis-Mingione's works \cite{DeFM19II} and \cite{DeFM20} for optimal Lipschitz bounds in the context of nonautonomous integrals. Furthermore, existence/multiplicity issues of solutions to double phase problems like \eqref{DP-Eq} were completely addressed by Liu-Dai's recent work \cite{LD18}.

\bigskip
Historically, the study of double phase problems as \eqref{DP-Eq} dates back to Zhikov's fundamental works \cite{Zhi93} and \cite{Zhi95}, which describe the behavior of certain strongly anisotropic materials, whose hardening estates, connected to the exponents of the gradient's growth, change in a point-wise fashion. In such a scenario, a mixture of two heterogeneous materials, with hardening $(p, q)$-exponents, can be performed according to the intrinsic geometry of the null set of the modeling function:
$$
\mathscr{L} w(x_0) = \left\{
\begin{array}{lcl}
  -\Delta_p w(x_0) & \text{if} & x_0 \in \mathcal{Z}(\mathfrak{a}, \Omega) \\
  -\Delta_p w(x_0)-\Div (\mathfrak{a}(x_0)|\nabla w(x_0)|^{q-2}\nabla w(x_0)) & \text{if} & x_0 \notin \mathcal{Z}(\mathfrak{a}, \Omega).
\end{array}
\right.
$$
Moreover, Zhikov's works also give new instances of the occurrence of the so-called \textit{Lavrentiev's phenomenon}, i.e., for a suitable boundary datum $u_0$:
$$
\displaystyle \inf_{w \in u_0+ W^{1, p}(\Omega)} \int_{\Omega} F(x, \nabla w)dx< \inf_{w \in u_0+ W^{1, q}(\Omega)} \int_{\Omega} F(x, \nabla w)dx,
$$
which states that is not possible to attain the minimum of the functional \eqref{DP-Eq} in a more regular fashion - a natural obstruction to regularity (cf. \cite{BCM15III}, \cite{CM15}, \cite{CM15II} and \cite{DeFM19II} for enlightening works). Particularly, Zhikov shown in \cite{Zhi97} that functionals with $p(x)-$growth manifest the Lavrentiev phenomenon if and only if the critical continuity condition is infringed (cf. \cite[Section 4.1]{DHHR17} and \cite[Part 1]{RaRe15}):
$$
   \displaystyle \limsup_{s \to 0} \,\hat{\omega}(s)\ln(s^{-1})\leq \mathrm{C}_1,
$$
for a universal modulus of continuity $\hat{\omega}:[0, \infty) \to [0, \infty)$ such that
$$
   |p(x)-p(y)| \leq \mathrm{C}_2\cdot\hat{\omega}(|x-y|).
$$
On the other hand, such functionals \eqref{DP-Eq} also appear in a variety of physical models; see \cite{Zhi11} for applications in the elasticity theory, \cite{BRR19} for transonic flows, \cite{BD'AFP00} for model of static solutions for elementary particle, among others.

\bigskip
Along recent directions, different generalized functionals have been considered in the literature. For instance, Ragusa and Tachikawa in \cite{RaTach20} analysed minimizers to a class of integral functionals
of double phase type with variable exponents as follows
{\scriptsize{
\begin{equation}\label{EqFuncVarExp}
  W_{\text{loc}}^{1, 1}(\Omega)\ni w \mapsto \int_{\Omega} \left(|\nabla w|^{p(x)} + \mathfrak{a}(x)|\nabla w|^{q(x)}\right)dx, \,\,\, \text{for} \,\,\, q(x)\ge p(x)\ge p_0>1,\,\,\mathfrak{a}(x)\ge 0,
\end{equation}}}where $p(\cdot)$, $q(\cdot)$ and $\mathfrak{a}(\cdot)$ are assumed to be H\"{o}lder continuous functions. In such a context, the authors prove (under appropriate assumptions on data) that local minimizers are $C_{\text{loc}}^{1, \gamma}(\Omega)$ for some $\gamma \in (0, 1)$.

Regarding functionals \eqref{Eq-Functional}, \eqref{Funct-Col-Ming} and \eqref{EqFuncVarExp}, we suggest reading the Mingione-R\v{a}dulescu's interesting survey \cite{MingRad21} on recent developments in such class of problems with nonstandard growth and nonuniform ellipticity regime.

\bigskip
In turn, in the light of recent non-divergence researches (i.e., a non-variational counterpart of certain variational integrals of the calculus of variations with non-standard growth \eqref{Funct-Col-Ming}), De Filippis' manuscript \cite{DeF20} was the pioneering work in considering fully nonlinear problems with non-homogeneous degeneracy. Precisely, De Filippis proved  $C_{\text{loc}}^{1, \gamma}$ regularity estimates for viscosity solutions of
{\scriptsize{
\begin{equation}\label{EqDeFilip}
    \left[|Du|^p+\mathfrak{a}(x)|Du|^q\right]F(D^2 u) = f \in C^{0}(\overline{\Omega}), \quad \text{for} \quad q\ge p>0,\,\,\,\mathfrak{a}(x)\ge 0,
\end{equation}}}for some $\gamma \in (0, 1)$ depending on universal parameters. Aftermath, da Silva-Ricarte in \cite{daSR20} improved De Filippis' result and addressed a variety of applications in nonlinear elliptic models and related free boundary problems.

\bigskip
Concerning fully nonlinear models of (single) degenerate type (i.e. $\mathfrak{a} \equiv 0$)
\begin{equation}\label{SingDegEq}
  \mathcal{G}_p[u] \defeq |D u|^pF(D^2 u) = f(x) \quad \text{in} \quad \Omega
\end{equation}
the list of contributions is fairly diverse, including aspects such as existence/uniqueness issues, Harnack inequality, ABP estimates \cite{daSRRV21} and \cite{daSJR21}, Liouville type results \cite{daSLR20}, local H\"{o}lder and Lipschitz estimates, local  gradient estimates \cite{ART15}, \cite{BD2}, \cite{BD3}, \cite{BD16},  and \cite{IS}, as well as connections with a variety of free boundary problems of Bernoulli type \cite{daSRRV21}, obstacle type \cite{DSVI} and \cite{DSVII}, singular perturbation type \cite{ART17}, \cite{daSJR21} and \cite{RT11}, and dead-core type \cite{daSLR20}, just to name a few.

At this point, we must quote the series of Berindelli-Demengel's key works \cite{BD2}, \cite{BD3} and \cite{BD16}, where local gradient regularity to \eqref{SingDegEq} have been extended to up-to-the-boundary $C^{1, \gamma}$ estimates in the presence of sufficiently regular domain and boundary datum. In addition, the following estimate holds
$$
       \displaystyle \|u\|_{C^{1, \gamma}(\overline{\Omega})}\leq \mathrm{C}(n, \lambda, \Lambda, \gamma)\cdot\left(\|u\|_{L^{\infty}(\Omega)} + \|g\|_{C^{1, \beta_g}(\partial \Omega)} +\|f\|_{L^{\infty}(\Omega)}^{\frac{1}{p+1}}\right).
$$
Aftermath, in the recent work \cite{AS21}, Ara\'{u}jo-Sirakov proved optimal boundary and global gradient estimates for solutions of \eqref{SingDegEq} in the spirit of Ara\'{u}jo \textit{et al}'s work \cite{ART15} and Silvestre-Sirakov's work \cite{SS14}.

We must also place emphasis on Bronzi \text{et al}'s result \cite[Theorem 2.1]{BPRT20}, where the authors show that viscosity solutions for the class of variable-exponent fully nonlinear models
$$
|D u|^{\theta(x)}F(D^2 u)=f(x) \quad \text{in} \quad B_1
$$
are of class $C^{1, \gamma}(B_{1/2})$, where $f\in C^{0}(B_1)$, and $\theta \in C^{0}(B_1)$ enjoys minimal assumptions such that $\displaystyle \inf_{B_1} \theta(x)>-1$ (structural law of singular/degenerate type). Moreover, the following estimate holds
$$
       \displaystyle \|u\|_{C^{1, \gamma}\left( B_{\frac{1}{2}}\right)}\leq \mathrm{C}(n, \lambda, \Lambda, \gamma)\cdot\left(\|u\|_{L^{\infty}(B_1)}  +\|f\|_{L^{\infty}(B_1)}^{\frac{1}{\|\theta^{+}\|_{\infty}+\|\theta^{-}\|_{\infty}+1}}+1\right).
$$

\bigskip
Finally, completing this mathematical state-of-the-art, we must quote the recent work due to Fang-R\v{a}dulescu-Zhang, which by combining the approaches from \cite{BPRT20} and \cite{DeF20} established local $C^{1, \alpha}$ regularity estimates to degenerate fully nonlinear elliptic PDEs with variable exponent
{\scriptsize{
\begin{equation}\label{EqFRZ}
    \left[|Du|^{p(x)}+\mathfrak{a}(x)|Du|^{q(x)}\right]F(D^2 u) = f \in C^{0}(\overline{\Omega}), \quad \text{for} \quad q(x)\ge p(x)\ge p_0>0,\,\,\,\mathfrak{a}(x)\ge 0.
\end{equation}}}

Consequently, after these breakthroughs and taking into account the previous highlights, according to our scientific knowledge, there is no investigation concerning existence/uniqueness, sharp and global regularity for such class of problems \eqref{EqFRZ} in a general scenario with unbalanced degeneracy and variable order. For this reason, such lack of investigations was one of our main impetus in considering existence/regularity issues for a class of strongly degenerate models and deliver some relevant applications.

Therefore, in this work we shall derive existence issues and global geometric regularity estimates for solutions of a class of nonlinear elliptic equations having a non-homogeneous degeneracy and variable exponent, whose general mathematical model is given by
\begin{equation}\label{1.1}
\left\{
\begin{array}{rcrcl}
  \mathcal{H}(x, Du)F(x, D^2 u) & = & f(x) & \text{in} & \Omega \\
  u(x) & = & g(x) & \text{on} & \partial \Omega,
\end{array}
\right.
\end{equation}
 for a bounded and regular open set $\Omega \subset \R^n$, suitable data $f$ and $g$. Throughout this work,
 $\mathcal{H}$ enjoys an appropriated degeneracy law and $F$ is assumed to be a second order, fully nonlinear (uniformly elliptic) operator, i.e., it is nonlinear in its highest order derivatives (to be specified soon).

The primary model case to \eqref{1.1} we have in mind is
{\scriptsize{
\begin{equation}\label{ModelEq}
  \left[|Du|^{p(x)}+\mathfrak{a}(x)|Du|^{q(x)}\right] \tr(\mathbb{A}(x)D^2 u) =  f(x) \quad \text{in} \quad B_1 \,\,\, \text{with} \,\,\, \lambda|\xi|^2\le \xi^{t}\mathbb{A}(x)\cdot \xi \le \Lambda|\xi|^2,
\end{equation}}}for a symmetric matrix $\mathbb{A} \in \text{Sym}(n)$ and constants $0<\lambda \le \Lambda < \infty$. In a mathematical perspective,  \eqref{ModelEq} enjoys distinct types of degeneracy laws under a variable exponent regime, depending on the values of the modulating function $\mathfrak{a}(\cdot) \ge 0$, as well as of magnitude of vectorial component
{\scriptsize{
$$
\text{If} \quad x \notin \mathcal{Z}(\mathfrak{a}, \Omega) \quad \Rightarrow \quad |\xi|^{p(x)} \le \mathcal{H}(x, \xi) \le \left\{
\begin{array}{rcl}
  \left(1+ \|\mathfrak{a}\|_{L^{\infty}(\Omega)}\right)|\xi|^{p(x)} & \text{if} & |\xi|\le 1 \\
  \left(1+\|\mathfrak{a}\|_{L^{\infty}(\Omega)}\right)|\xi|^{q(x)} & \text{if} & |\xi|>1.
\end{array}
\right.
$$}}
$$
\text{If} \quad x \in \mathcal{Z}(\mathfrak{a}, \Omega) \quad \Rightarrow \quad \mathcal{H}(x, \xi) = |\xi|^{p(x)}.
$$
As a consequence, diffusion properties of the model exhibit a sort of non-uniformly elliptic and doubly (non-homogeneous) degenerate signature, which combines two different type operators of variable order (cf. \cite{ART15}, \cite{BD2}, \cite{BD3}, \cite{BD16}, \cite{BPRT20} and \cite{IS} for single degeneracy scenarios).

In conclusion, in this research framework, our main contributions are:

\begin{enumerate}
  \item Existence/uniqueness of solutions (Theorem \ref{existence});
  \item Global H\"{o}lder or Lipschitz estimates (Theorems \ref{Xizao} and \ref{pgrande});
  \item Higher regularity: gradient estimates (improved local and sharp up-to-the-boundary estimates) (Theorems \ref{main1*} and \ref{main1});
  \item Geometric non-degeneracy properties of solutions (Theorem \ref{main3});
  \item Finer estimates and some applications in nonlinear models and free boundary problems (Section \ref{sec.Applic}).
\end{enumerate}

\subsection{Statement of the main results}\label{ssec.def}

In the sequel, we will present some useful definitions. Firstly, let us introduce the notion of viscosity solution for our operators.

\begin{definition}
  [{\bf Viscosity solutions}] A function $u \in C^{0}(\Omega)$ is a viscosity super-solution (resp. sub-solution) to \eqref{1.1} if whenever $\varphi \in C^2(\Omega)$ and $x_0 \in \Omega$ such that $u-\varphi$ has a local minimum (resp. a local maximum) at $x_0$, then
$$
   \mathcal{H}(x_0, D\varphi(x_0))F(x_0, D^2 \varphi(x_0)) \leq f(x_0) \qquad \text{resp.} \,\,\,(\cdots \geq f(x_0)).
$$
Finally, $u$ is said to be a viscosity solution if it is simultaneously a viscosity super-solution and
a viscosity sub-solution.
\end{definition}

In order to measure the smoothness of solutions in suitable spaces, we are going to use the following norms and semi-norms (see, \cite[Section 1]{Kov99}):

\begin{definition}
  [{\bf $C^{1, \alpha}$ norm and semi-norm}] For $\alpha \in (0,1]$, $C^{1, \alpha}(\Omega)$ denotes the space of all functions $u$ whose gradient $Du(x)$ there exists in the classical sense for every $x\in \Omega$ such that the norm
$$
\|u\|_{C^{1, \alpha}(\Omega)} \defeq \|u\|_{L^{\infty}(\Omega)} + \|Du\|_{L^{\infty}(\Omega)} + [u]_{C^{1, \alpha}(\Omega)}
$$
is finite. Moreover, we have the semi-norm
\begin{equation}\label{Semi-Norm}
  [u]_{C^{1, \alpha}(\Omega)} \defeq \sup_{\substack{x_0\in \Omega \\0< r\leq \diam(\Omega)}}  \inf_{\mathfrak{l \in \mathcal{P}_1}} \frac{\|u-\mathfrak{l}\|_{L^{\infty}(B_r(x_0)\cap \Omega)}}{r^{1+\alpha}},
\end{equation}
where $\mathcal{P}_1$ denotes the spaces of polynomial functions of degree at most 1. As a result, $u \in C^{1, \alpha}(\Omega)$ implies every component of $Du$ is $C^{0, \alpha}(\Omega)$ (see, \cite[Main Theorem]{Kov99}).
\end{definition}

Now we are in a position to state our main results. The first one establishes an optimal geometric estimate. In effect, it reads that if the source term is bounded and (A0)-(A2), \eqref{1.2} and \eqref{1.3} are in force, then any bounded viscosity solution of \eqref{1.1} belongs to $C^{1, \alpha}$ at interior points, where
 \begin{equation}\label{SharpExp}
   \alpha \defeq \min\left\{\alpha^{-}_{\mathrm{F}},  \frac{1}{p_{\text{max}}+1}, \beta_g\right\},
 \end{equation}
where $\alpha_{\mathrm{F}} = \alpha_{\mathrm{F}}(n, \lambda, \Lambda) \in (0, 1]$ is the optimal exponent to (local) H\"{o}lder continuity of gradient of solutions to homogeneous problem with ``frozen coefficients'' $F(D^2 \mathfrak{h}) = 0$ (see, \cite{C89}, \cite[Ch.5 \S 3]{CC95} and \cite{Tru88}).

Under the structural assumptions in subsection \ref{SectionSA} we are in a position to present the following results:

As our first result, we prove the existence of viscosity solutions to \eqref{1.1}.

\begin{theorem}[{\bf Existence of viscosity solutions}]\label{existence} Suppose assumptions (A0$^{\prime}$)-(A5) and \eqref{Cont-H} are in force for a continuous datum $f$. Then, there exists a unique viscosity solution $u \in C^0(\overline{\Omega})$ to \eqref{1.1}.
\end{theorem}

In the sequel, we present our first (basic) regularity result.

\begin{theorem}[{\bf H\"{o}lder regularity}]\label{Xizao}
		Let $g$ be a Lipschitz continuous datum. Suppose that $u$ satisfies
\begin{equation}\label{dirichlet}
  		\left\{
		\begin{array}{rcl}
			\mathcal{H}(y, \nabla u)F(y,D^{2}u)=f(y) \ &\text{in}& \ B\cap\{y_{y}>\phi(y^{\prime})\} \\
			u(y) = g(y) 			& \text{on}& \ B\cap\{y_{n}=\phi(y^{\prime})\}
		\end{array}
		\right.
\end{equation}

Then, for every $\gamma, r \in (0, 1)$, $u \in C^{0, \gamma}(B_{r}\cap\{y_{n}>\phi(y^{\prime})\})$. Furthermore, the following estimate holds
{\scriptsize{
$$
\|u\|_{C^{0, \gamma}(B_{r}\cap\{y_{n}>\phi(y^{\prime})\})} \leq \mathrm{C}(\lambda,\Lambda, n, q_{\text{max}},p_{\text{min}},r,\|\mathfrak{a}\|_{L^{\infty}(\Omega)}, \| f\|_{L^{\infty}(\Omega)}, \mathrm{Lip}_{g}(\partial \Omega)).
$$}}
\end{theorem}

\begin{remark}
  By combining the above estimate, a covering argument and the local H\"{o}lder regularity in Theorem \ref{HoldEstThm}, we can conclude a global $\gamma_0-$H\"{o}lder regularity. Moreover, the following estimate holds
$$
\|u\|_{C^{0, \gamma_0}(\overline{\Omega})} \leq \mathrm{C}(\lambda,\Lambda,n, q_{\text{max}},p_{\text{min}}, \|u\|_{L^{\infty}(\overline{\Omega})}, \|\mathfrak{a}\|_{L^{\infty}(\Omega)}, \| f\|_{L^{\infty}(\Omega)}, \mathrm{Lip}_{g}(\partial \Omega)).
$$
\end{remark}

In the sequel, we present one of our main regularity result.

\begin{theorem}[{\bf Improved regularity at interior points}]\label{main1*} Assume that assumptions (A0)-(A5) there hold. Let $u$ be a bounded viscosity solution to \eqref{1.1}. Then, $u$ is $C^{1, \alpha}$, at interior points, for an $\alpha \in (0, \alpha_{\mathrm{F}}) \cap \left(0, \frac{1}{p_{\text{max}}+1}\right]$. More precisely, for any point $x_0 \in \Omega$ such that $B_r(x_0)\Subset \Omega$ there holds
$$
       \displaystyle [u]_{C^{1, \alpha}(B_r(x_0))}\leq \mathrm{C}\cdot\left(\|u\|_{L^{\infty}(\Omega)}+ \|f\|_{L^{\infty}(\Omega)}^{\frac{1}{p_{\text{min}}+1}}+1\right),
$$
for $r \in \left(0, \frac{1}{2}\right)$ where $\mathrm{C}>0$ is a universal constant\footnote{A constant is said to be universal if it depends only on dimension, degeneracy and ellipticity constants, $\alpha_{\mathrm{F}}$, $\beta_g$, $L_1, L_2$ and $\|F\|_{C^{\omega}(\Omega)}$.}.
\end{theorem}

We must stress that such a Theorem is an improved version of the one addressed in \cite[Theorem 1.1]{FRZ21} (see also \cite[Theorem 1]{DeF20}).

As a result of Theorem \ref{main1*} we have the Corollary (cf. \cite[Corollary 3.2]{ART15}):

\begin{corollary}\label{Cormain1} Assume that assumptions of Theorem \ref{main1*} are in force. Suppose $F$ to be a concave/convex operator. Then, $u$ is $C_{\text{loc}}^{1, \frac{1}{p_{\text{max}}+1}}(\Omega)$. Moreover, there holds
$$
       \displaystyle  \displaystyle \|u\|_{C^{1, \frac{1}{p_{\text{max}}+1}}(\Omega^{\prime})}\leq \mathrm{C}(\verb"universal")\cdot\left(\|u\|_{L^{\infty}(\Omega)}  +\|f\|_{L^{\infty}(\Omega)}^{\frac{1}{p_{\text{min}}+1}}+1\right).
$$
\end{corollary}

One of the main novelties of our approach consists of removing the restriction of analyzing $C_{\text{loc}}^{1, \alpha}$ regularity estimates just along the \textit{a priori} unknown set of singular points of solutions $\mathcal{S}_{0}(u, \Omega)$, where the ``ellipticity of the operator'' degenerates (see for example, \cite{daSRS19I}, \cite{daSS18} and \cite{T15}, where improved regularity estimates were addressed along certain sets of degenerate points of existing solutions).

From now on, we will label the singular zone of existing solutions as
$$
  \mathcal{S}_{r, \theta}(u, \Omega^{\prime}) \defeq \left\{x_0 \in \Omega^{\prime} \Subset \Omega: |Du(x_0)| \leq r^{\theta}, \,\,\,\mbox{for}\,\,\,0\leq r\ll1\right\}.
$$

A geometric inspection to Theorem \ref{main1*} ensures that if $u$ satisfies \eqref{1.1} and $x_0 \in \mathcal{S}_{r, \alpha}(u, \Omega^{\prime})$, then near $x_0$ we get
$$
    \displaystyle \sup_{B_r(x_0)} |u(x)|\leq |u(x_0)|+\mathrm{C}\cdot r^{1+\alpha},
$$

Nonetheless, from a geometric perspective, it plays an essential qualitative role to obtain the (counterpart) sharp lower bound estimate for such PDEs with non-homogeneous degeneracy law. This feature is designated \textit{Non-degeneracy property} of solutions. Therefore, under a natural, non-degeneracy hypothesis on $f$, we establish the precise behavior of solutions at certain (interior) ``singular zones''.

\begin{theorem}[{\bf Non-degeneracy property}]\label{main3} Suppose that assumptions of Theorem \ref{main1*} are in force. Let $u$ be a bounded viscosity solution to \eqref{1.1} with $f(x) \geq \mathfrak{m}>0$ in $\Omega$. Given $x_0 \in \mathcal{S}_{r, \alpha}(u, \Omega^{\prime})$, there exists a constant $\mathrm{c}_0 = \mathrm{c}_0(\mathfrak{m},\|\mathfrak{a}\|_{L^{\infty}(\Omega)}, \mathrm{L}_1, n, \lambda, \Lambda, p_{\text{min}}, q_{\text{min}}, \Omega)>0$, such that
$$
  \displaystyle \sup_{\partial B_r(x_0)} \frac{u(x)-u(x_0)}{r^{1+\frac{1}{p_{\text{min}}+1}}} \geq  \mathrm{c}_0  \quad \text{for all} \quad  r \in \left(0, \frac{1}{2}\right).
$$
\end{theorem}

Such a quantitative information plays an essential role in the development of several analytic and geometric problems, such as in the study of blow-up procedures and related weak geometric and free boundary analysis (cf. \cite{daSRRV21}, \cite{daSRS19I} and \cite{daSS18} for related topics). This is the first non-degeneracy result for nonlinear degenerate models with variable exponent in the current literature.

Finally, completing our analysis, we obtain our last regularity result.

\begin{theorem}[{\bf Sharp regularity - up to the boundary estimates}]\label{main1} Suppose that assumptions (A0)-(A5) are in force. Let $u$ be a bounded viscosity solution to \eqref{1.1}. Then, $u$ is $C^{1, \alpha}$, up to the boundary, for $\alpha$ satisfying \eqref{SharpExp}. More precisely, the following estimate there holds
$$
       \displaystyle \|u\|_{C^{1, \alpha}(\overline{\Omega})}\leq \mathrm{C}(\verb"universal")\cdot\left(\|u\|_{L^{\infty}(\Omega)} + \|g\|_{C^{1, \beta_g}(\partial \Omega)} +\|f\|_{L^{\infty}(\Omega)}^{\frac{1}{p_{\text{min}}+1}}+1\right).
$$
\end{theorem}

As an immediate consequence of our global estimates, we obtain the following optimal regularity.

\begin{corollary}\label{Cor2} Assume that assumptions of Theorem \ref{main1} are in force. Suppose further that $F$ and $g$ are as in \cite{SS14}. Then, $u$ is $C^{1, \frac{1}{p_{\text{min}}+1}}(\overline{\Omega})$. Moreover, there holds
$$
       \displaystyle  \displaystyle \|u\|_{C^{1, \frac{1}{p_{\text{max}}+1}}(\overline{\Omega})}\leq \mathrm{C}(\verb"universal")\cdot\left(\|u\|_{L^{\infty}(\Omega)} + \|g\|_{C^{1,1}(\partial \Omega)} +\|f\|_{L^{\infty}(\Omega)}^{\frac{1}{p_{\text{min}}+1}}+1\right).
$$
\end{corollary}

Our findings extend/generalize regarding non-variational scenario, former results (H\"{o}lder gradient estimates) from \cite[Theorem 3.1 and Corollary 3.2]{ART15}, \cite{BD2}, \cite{BD3} and \cite{IS}, and to some extent, those from \cite[Theorems 1.1 and 1.2]{AS21}, \cite[Theorem 1.1]{BD16}, \cite[Theorem 2.1]{BPRT20}, \cite[Theorem 1]{DeF20} and \cite[Theorem 1.1]{FRZ21} by making using of different approaches and techniques adapted to the general framework of the fully nonlinear non-homogeneous degeneracy models. Finally, they are even striking for the  very special toy model
$$
 \mathcal{G}_{p, q, \mathfrak{a}}[u] = \left[|Du|^{p(x)}+\mathfrak{a}(x)|Du|^{q(x)}\right] \mathscr{M}_{\lambda, \Lambda}^{\pm}(D^2 u) = 1 \quad \text{with} \quad u_{|\partial \Omega} = 1.
$$

We should stress that an extension of our results also holds to Multi-Phase equations with variable exponents, which are a generalization of \eqref{EqFRZ}
$$
   \left(|Du|^{p(x)} + \sum_{i=1}^{n} \mathfrak{a}_i(x)|Du|^{q_i(x)}\right)F(x, D^2 u) = f(x, u) \quad \text{in} \quad \Omega,
$$
where $0\le\mathfrak{a}_i \in C^0(\Omega)$, $p, q_i \in C^0(\Omega)$ for $1\le i \le n$, and $0<p(x)\leq q_1(x)\leq \cdots\leq q_n(x)< \infty$ (cf. \cite{DeFO19} and \cite{UU83} for multi-phase variational problems).

Finally, we believe our results can be applied to others variational/non-variational elliptic problems (see Section \ref{sec.Applic}):

\begin{enumerate}
  \item Improved point-wise estimates (cf. \cite{BPRT20});
  \item Sharp regularity for Strong $p(x)-$Laplacian equation (connections with $C^{1, \frac{1}{3}}$-conjecture for $\infty-$harmonic functions - see \cite{ES08});
  \item Sharp estimates for $(p(x)\&q(x))-$harmonic functions (connections with $C^{p^{\prime}}$-regularity conjecture for inhomogeneous $p-$harmonic functions - see \cite{ATU17});
  \item Applications in some geometric free boundary problems:
  \begin{itemize}
    \item[\checkmark] Dead-core type problems (cf. \cite{daSLR20}, \cite{daSRS19I}, \cite{daSS18} and \cite{Tei16});
    \item[\checkmark] Obstacle type problems (cf. \cite{ALS15}, \cite{DSVI} and \cite{DSVII});
    \item[\checkmark] Bernoulli type problems (cf. \cite{daSRRV21});
    \item[\checkmark] Singularly perturbed problems (cf. \cite{ART17}, \cite{daSJR21} and \cite{RT11}).
  \end{itemize}
\end{enumerate}

\subsection{Structural assumptions of the model}\label{SectionSA}

Throughout this work, we will suppose the following structural conditions:
\begin{itemize}
  \item[(A0$^{\prime}$)]({{\bf Regularity of the domain}}) In this work, we will assume that $\Omega \subset \R^n$ is a bounded open domain, which satisfies a uniform exterior sphere condition. This is, for each $z\in\partial\Omega$, we can choose $x_z \in \Omega$ such that
      $$
         \overline{B_r(x_z)}\cap\overline \Omega=\{z\} \quad  \text{with} \quad r=|z-x_z|.
      $$

Particularly, it is well-known that $C^2$ domains (i.e. with $C^2$ boundary) satisfy a uniform exterior sphere condition. We will use this fact in the future.

  \item[(A0)]({{\bf Continuity and normalization condition}})
  $$
  \text{Fixed}\,\, \Omega \ni x \mapsto F(x, \cdot) \in C^{0}(\text{Sym}(n)) \quad  \text{and} \quad F(\cdot, \text{O}_{n\times n}) = 0.
  $$
  \item[(A1)]({{\bf Uniform ellipticity}}) For any pair of matrices $\mathrm{X}, \mathrm{Y}\in Sym(n)$
$$
    \mathscr{M}_{\lambda, \Lambda}^-(\mathrm{X}-\mathrm{Y})\leq F(x, \mathrm{X})-F(x, \mathrm{Y})\leq \mathscr{M}_{\lambda, \Lambda}^+(\mathrm{X}-\mathrm{Y})
$$
where $\mathscr{M}_{\lambda, \Lambda}^{\pm}$ are the \textit{Pucci's extremal operators} given by
\[
   \mathscr{M}_{\lambda, \Lambda}^-(\mathrm{X})\defeq \lambda\sum_{e_i>0}e_i+\Lambda\sum_{e_i<0}e_i\quad\textrm{ and }\quad \mathscr{M}_{\lambda, \Lambda}^+(X)\defeq \Lambda\sum_{e_i>0}e_i+\lambda\sum_{e_i<0}e_i
\]
for \textit{ellipticity constants} $0<\lambda\leq \Lambda< \infty$, where $\{e_i(\mathrm{X})\}_{i=1}^{n}$ are the eigenvalues of $\mathrm{X}$.

Additionally, for our Theorem \ref{main1} (resp. Corollary \ref{Cormain1}), we must require some sort of continuity assumption on coefficients:

  \item[(A2)]({{\bf $\omega-$continuity of coefficients}}) There exist a uniform modulus of continuity $\omega: [0, \infty) \to [0, \infty)$ and a constant $\mathrm{C}_{\mathrm{F}}>0$ such that
$$
\Omega \ni x, x_0 \mapsto \Theta_{\mathrm{F}}(x, x_0) \defeq \sup_{\substack{\mathrm{X} \in  Sym(n) \\ \mathrm{X} \neq 0}}\frac{|F(x, \mathrm{X})-F(x_0, \mathrm{X})|}{\|\mathrm{X}\|}\leq \mathrm{C}_{\mathrm{F}}\omega(|x-x_0|),
$$
which measures the oscillation of coefficients of $F$ around $x_0$. For simplicity purposes, we shall often write $\Theta_{\mathrm{F}}(x, 0) = \Theta_{\mathrm{F}}(x)$.

Finally, for notation purposes we define
$$
\|F\|_{C^{\omega}(\Omega)} \defeq \inf\left\{\mathrm{C}_{\mathrm{F}}>0: \frac{\Theta_{\mathrm{F}}(x, x_0)}{\omega(|x-x_0|)} \leq \mathrm{C}_{\mathrm{F}}, \,\,\, \forall \,\,x, x_0 \in \Omega, \,\,x \neq x_0\right\}.
$$

\item[(A3)]({{\bf Assumptions on data}}) Throughout this manuscript we will assume the following assumptions on data: $f \in L^{\infty}(\overline{\Omega})$ and $g \in C^{1, \beta_g}(\partial \Omega)$ for some $\beta_g \in (0, 1]$. Moreover, when necessary, we invoke continuity assumption on $f$. (see Theorem \ref{existence}).

\item[(A4)]({{\bf Non-homogeneous degeneracy}}) In our studies, we enforce that the diffusion properties of the model \eqref{1.1} degenerate along an \textit{a priori} unknown set of singular points of existing solutions:
$$
   \mathcal{S}_0(u, \Omega) \defeq \{x \in \Omega: |Du(x)| = 0\}.
$$
Consequently, we will impose that $\mathcal{H}:\Omega \times \R^n \to [0, \infty)$  -- non-homogeneous degeneracy law -- one behaves as
\begin{equation}\label{1.2}
     \mathrm{L}_1 \cdot \mathcal{K}_{p, q, \mathfrak{a}}(x, |\xi|) \leq \mathcal{H}(x, \xi)\leq \mathrm{L}_2 \cdot \mathcal{K}_{p, q, \mathfrak{a}}(x, |\xi|)
\end{equation}
for constants $0<\mathrm{L}_1\le \mathrm{L}_2< \infty$, where
\begin{equation}\label{N-HDeg}\tag{\bf N-HDeg}
   \mathcal{K}_{p, q, \mathfrak{a}}(x, |\xi|) \defeq |\xi|^{p(x)}+\mathfrak{a}(x)|\xi|^{q(x)}, \,\,\,\text{for}\,\,\, (x, \xi) \in \Omega \times \R^n.
\end{equation}

\item[(A5)]({{\bf Assumptions on variable exponents}}) In turn, for the degeneracy law in \eqref{N-HDeg}, we suppose that the functions $p, q \in L^{\infty}(\Omega)$ and the modulating function $\mathfrak{a}(\cdot)$ fulfil
\begin{equation}\label{1.3}
   0< p_{\text{min}}\leq p(x) \le p_{\text{max}} \leq q(x)\le q_{\text{max}}<\infty \quad \text{and} \quad 0\le \mathfrak{a} \in C^0(\overline{\Omega}).
\end{equation}

\end{itemize}

In conclusion, our manuscript is organized as follows: in Section \ref{sec.ALemmas} we show some preliminary results, namely the existence of solutions (Theorem \ref{existence}) and H\"{o}lder and Lipschitz regularity estimates. Sections \ref{Sec3} and \ref{Sec4} are concerned to prove the main results of this paper (Theorems \ref{main1*}, \ref{main1}, \ref{main3} and Corollary \ref{Cormain1}) by obtaining sharp regularity estimates at interior and boundary points. Finally, Section \ref{sec.Applic} is devoted to present some relevant application of our findings with relevant nonlinear elliptic problems, as well as we establish connections with some nonlinear geometric free boundary problems and beyond.

\section{Preliminary results}\label{sec.ALemmas}

We start off by presenting some key results that will be used later on.

As treated in \cite{BD2}, in order to prove the Theorem \ref{main1}, we may suppose, without loss of generality, that $0 \in \partial \Omega$, and the interior normal is $e_n$. For this reason, by the Implicit Function Theorem, there exist a ball $\mathrm{B}=\mathrm{B}_{\mathrm{R}}(0) \subset \mathbb{R}^n$, and $\mathrm{D}^{\prime} \subset \mathrm{B}^{\prime}_{\mathrm{R}}(0)$ ball of $\mathbb{R}^{n-1}$ and $\phi \in C^2(\mathrm{D}^{\prime})$ such that
\begin{equation}\label{Condphi}
    \phi(0)=|D\phi(0)|=0 \quad  \text{for} \quad y=(y^{\prime}, y_n).
\end{equation}
Moreover,
$$
	\Omega \cap \mathrm{B} \subset \{y_n > \phi(y^{\prime}), \,\, y^{\prime}\in \mathrm{D}^{\prime}\} \,\,\, \textrm{and} \,\,\, \partial \Omega \cap \mathrm{B} = \{y_n = \phi(y^{\prime}), \,\, y^{\prime}\in \mathrm{D}^{\prime}\}.
$$

As a result, we will consider viscosity solutions of
\begin{equation} \label{Eq2.1}
\left\{
\begin{array}{rclcl}
\mathcal{H}(y, Du)F(y, D^2 u) & = & f(y) & \text{in} & \mathrm{B} \cap \{y_n >\phi(y^{\prime})\}  \\
  u& = & g & \text{on} & \mathrm{B} \cap \{y_n = \phi(y^{\prime})\},
\end{array}
\right.
\end{equation}

Therefore, it will be sufficient to prove a localized version for Theorem \ref{main1}. More precisely, for any $x \in \mathrm{B} \cap \{y_n > \phi(y^{\prime})\}$ with $B_1(x) \subset \mathrm{B}$ and for any $r \in (0, 1)$, we have that solutions of \eqref{Eq2.1} are $C^{1, \alpha}(B_r(x) \cap \{y_n \ge \phi(y^{\prime})\})$.

In the sequel, we will present an ABP estimate adapted to our context of fully nonlinear models with unbalanced variable degeneracy. Such an estimate is pivotal in order to obtain universal bounds for viscosity solutions in terms of data of the problem. The proof follows the same lines as \cite[Theorem 2.5]{daSJR21}. For this reason, we will omit it here.

\begin{theorem}[{\bf Alexandroff-Bakelman-Pucci estimate}]\label{ABPthm}
  Suppose that assumptions (A0)-(A2) there hold. Then, there exists a constant $\mathrm{C} = \mathrm{C}(n, \lambda, p_{min}, q_{max}, \diam(\Omega))>0$ such that for any $u \in C^0(\overline{\Omega})$ viscosity sub-solution (resp. super-solution) of \eqref{1.1} in $\{x \in \Omega:u(x)>0\}$ (resp. $\{x \in \Omega:u(x)<0\}$, satisfies
{\scriptsize{
$$
\displaystyle \sup_{\Omega} u(x) \leq \sup_{\partial \Omega} g^{+}(x) +\mathrm{C}\cdot\diam(\Omega)\max\left\{\left\|\frac{f^{-}}{1+\mathfrak{a}}\right\|^{\frac{1}{p_{min}+1}}_{L^n(\Gamma^{+}(u^{+}))}, \left\|\frac{f^{-}}{1+\mathfrak{a}}\right\|^{\frac{1}{q_{max}+1}}_{L^n(\Gamma^{+}(u^{+}))}\right\},
$$}}
respectively,
{\scriptsize{
$$
\left(\displaystyle \sup_{\Omega} u^{-}(x) \leq \sup_{\partial \Omega} g^{-}(x) +\mathrm{C}\cdot\diam(\Omega)\max\left\{\left\|\frac{f^{+}}{1+\mathfrak{a}}\right\|^{\frac{1}{p_{min}+1}}_{L^n(\Gamma^{+}(u^{-}))}, \left\|\frac{f^{+}}{1+\mathfrak{a}}\right\|^{\frac{1}{q_{max}+1}}_{L^n(\Gamma^{+}(u^{-}))}
\right\}\right)
$$}}
where
$$
\Gamma^{+}(u) \defeq \left\{x \in \Omega: \exists \,\, \xi \in \R^n\,\,\,\text{such that}\,\,u(y)\le u(x)+\langle \xi, y-x\rangle \,\, \forall\, y \in \Omega\right\}.
$$

Particularly, we conclude that
{\scriptsize{
$$
\displaystyle \|u\|_{L^{\infty}(\Omega)} \leq \|g\|_{L^{\infty}(\partial \Omega)} +\mathrm{C}\cdot\diam(\Omega)\max\left\{\left\|\frac{f}{1+\mathfrak{a}}\right\|^{\frac{1}{p_{min}+1}}_{L^n(\Omega)}, \left\|\frac{f}{1+\mathfrak{a}}\right\|^{\frac{1}{q_{max}+1}}_{L^n(\Omega)}\right\}
.$$}}
\end{theorem}

We also present the following interior H\"{o}lder regularity result. The proof follows the same lines as \cite[Theorem 8.5]{daSJR21}.

\begin{theorem}[{\bf Local H\"{o}lder estimates}]\label{HoldEstThm}
Let $u$ be a viscosity solution to
$$
   \mathcal{H}(x, \nabla u)F(x, D^2u) = f(x) \quad \text{in} \quad B_1.
$$
where $f$ is a continuous and bounded function. Then, $u \in C_{\text{loc}}^{0, \alpha^{\prime}}(B_1)$ for some universal $\alpha^{\prime} \in (0, 1)$. Moreover,
{\scriptsize{
$$
\|u\|_{C^{0, \alpha^{\prime}}(\Omega^{\prime})} \leq \mathrm{C}(\verb"universal", \dist(\Omega^{\prime}, \partial \Omega))\left(\|u \|_{L^{\infty}(\Omega)} + \max\left\{ \left\|\frac{f}{1+\mathfrak{a}}\right\|^{\frac{1}{{p_{\text{min}}+1}}}_{L^n(\Omega)}, \left\|\frac{f}{1+\mathfrak{a}}\right\|^{\frac{1}{{q_{\text{max}}}+1}}_{L^n(\Omega)}\right\}\right).
$$}}
\end{theorem}

Before presenting our main Key Lemma, the following simple convergence result will be instrumental in the proof of Lemma \ref{L0}:
	
	\begin{lemma}\label{lem.stab} Let $(F_k(x,\mathrm{X}))_{k\in \mathbb{N}}$ be a sequence of operators satisfying (A0)-(A2) with the same ellipticity constants and same modulus of continuity in $B_1(0)$. Then, there exists an elliptic operator $F_0$ which still satisfies (A0)-(A2) such that
		\[
		F_k\longrightarrow F_0 \quad \mbox{uniformly on compact subsets of} \quad \text{Sym}(n)\times B_1(0).
		\]
	\end{lemma}

Another important piece of information we need in our article concerns to the notion of stability of viscosity solutions, i.e., the limit of a sequence of viscosity solutions turns out to be a viscosity solution of the corresponding limiting equation. The following Lemma will be instrumental in the proof of Lemma \ref{L0}, whose proof can be found in \cite{BD16} and \cite[Lemma 3.2]{FRZ21}.

\begin{lemma}[{\bf Stability}]\label{P1}
Let $(g_j)_j$ be a sequence of Lipschitz continuous functions such that $g_j \to g_{\infty}$. Suppose that $(u_j)$ is a sequence of uniformly bounded continuous viscosity solutions of
$$
\left\{
\begin{array}{rcrcl}
\mathcal{H}_j (y,Du_j + \zeta_j)F_j(x, D^2 u_j) & = & f_j & \text{in} & \mathrm{B} \cap \{y_n > \phi(y')\}  \\
  u_j  & = & g_j & \text{on} & \mathrm{B} \cap \{y_n = \phi(y')\},
\end{array}
\right.
$$
where $(\zeta_j)_j \subset \mathbb{R}^n$. Suppose further that $f_j \to 0$, $\zeta_j \to \zeta_{\infty}$ and $F_j (\cdot) \to F_{\infty}(\cdot)$. 
Then, one can extract from $(u_j)_j$ a subsequence which converges uniformly to $u_{\infty}$ on $\overline{\mathrm{B} \cap \{y_n > \phi(y^{\prime})\} }$. Moreover, such a limit $u_{\infty}$ satisfies in the viscosity sense
$$
\left\{
\begin{array}{rclcl}
F_{\infty}( D^2 u_{\infty}) & = & 0 & \text{in} & \mathrm{B} \cap \{y_n > \phi(y^{\prime})\}  \\
  u_{\infty} & = & g_{\infty} & \text{on} & \mathrm{B} \cap \{y_n = \phi(y^{\prime})\}.
\end{array}
\right.
$$
\end{lemma}

\begin{remark}\label{RemarkSeqTransl} It is noteworthy to stress that since $(\zeta_j)_j \subset \mathbb{R}^n$ converges, then it is bounded. Then, taking into account the Global H\"{o}lder estimates and ABP (Theorems \ref{Xizao} and \ref{ABPthm}) imply that $(u_j+ x \cdot \zeta_j)_j$ is a convergent sequence. In other words, the sequence of solutions in Lemma \ref{P1} is pre-compact in the $C^{0, \gamma}-$topology.
\end{remark}

\subsection{Existence of solutions: Proof of Theorem \ref{existence}}

In this subsection, we prove the classical result known as Comparison Principle for unbalanced degenerate scenario with variable exponent. For that end, we will consider the following approximating problem
\begin{equation}\label{appro_problem}
\mathcal{G}_{\varepsilon}[u] \defeq	\mathcal{H}_{\varepsilon}(x, \nabla u)[\varepsilon u + F(x, D^2 u)]  =  f(x) \quad \text{ in } \quad \Omega
\end{equation}
where $\mathcal{H}_{\varepsilon}:\Omega\times\mathbb{R}^n\to [0,\infty)$ satisfies
\begin{equation}\label{1.2-epsilon}
	\mathrm{L}_1 \cdot \mathcal{K}_{p, q, \mathfrak{a}}^\varepsilon(x, |\xi|) \leq \mathcal{H}_{\varepsilon}(x, \xi)\leq \mathrm{L}_2 \cdot \mathcal{K}_{p, q, \mathfrak{a}}^{\varepsilon}(x, |\xi|)
\end{equation}
for constants $0<\mathrm{L}_1\le \mathrm{L}_2< \infty$, with
\begin{equation}\label{N-HDeg-epsilon}
	\mathcal{K}_{p, q, \mathfrak{a}}^{\varepsilon}(x, |\xi|) \defeq (\varepsilon+|\xi|)^{p(x)}+\mathfrak{a}(x)(\varepsilon+|\xi|)^{q(x)}, \,\,\,\text{for}\,\,\, (x, \xi) \in \Omega \times \R^n
\end{equation}
for $\varepsilon \in (0, 1)$ and $u\in C^0(\overline{\Omega})$. The desired comparison result will hold true by letting $\varepsilon \to 0^{+}$. The proof follows essentially the idea from \cite{CIL92} adapted to our scenario.

Additionally to assumptions (A0)-(A2) and (A4)-(A5), we will assume the following condition: $p, q, \mathfrak{a} \in C^{0, 1}(\Omega)$ and there exist a constant
$$
  \mathrm{C}_{p, q, \mathfrak{a}} = \mathrm{C}_{p, q, \mathfrak{a}}\left(\|\nabla p\|_{L^{\infty}(\Omega)}, \|\nabla q\|_{L^{\infty}(\Omega)}, \|\nabla \mathfrak{a}\|_{L^{\infty}(\Omega)}, \|\mathfrak{a}\|_{L^{\infty}(\Omega)}, L_1, L_2\right) >0
$$
and a modulus of continuity $\omega_{p, q, \mathfrak{a}}: [0, \infty) \to [0, \infty)$ such that
{\small{\begin{equation}\label{Cont-H}
   |\mathcal{H}_{\varepsilon}(x, \xi)-\mathcal{H}_{\varepsilon}(y, \xi)| \le \mathrm{C}_{p, q, \mathfrak{a}} (\varepsilon+|\xi|)^{p_{\text{min}}}\left(|\ln(\varepsilon+|\xi|)|+1\right)\omega_{p, q, \mathfrak{a}}(|x-y|).
\end{equation}}}
for all $(x, y, \xi) \in \Omega \times \Omega\times \left(B_1 \setminus \{0\}\right)$.

\begin{remark} \label{Obs3.1}
By way of illustration, when
$$\mathcal{H}_{\varepsilon}(x, \xi) = (\varepsilon+|\xi|)^{p(x)} + \mathfrak{a}(x)(\varepsilon+|\xi|)^{q(x)}$$
 then \eqref{Cont-H} becomes
$$
   |\mathcal{H}_{\varepsilon}(x, \xi)-\mathcal{H}_{\varepsilon}(y, \xi)| \le  \mathcal{A}_1^{\varepsilon}(x, y , \xi) + \mathcal{A}_2^{\varepsilon}(x, y , \xi),
$$
where
$$
\mathcal{A}_1^{\varepsilon}(x, y , \xi) \defeq \|\nabla p\|_{L^{\infty}(\Omega)}(\varepsilon+|\xi|)^{p_{\text{min}}}|\ln(\varepsilon+|\xi|)||x-y|
$$
and
{\small{$$
\mathcal{A}_2^{\varepsilon}(x, y , \xi) \defeq \left(\|\nabla q\|_{L^{\infty}(\Omega)}\|\mathfrak{a}\|_{L^{\infty}(\Omega)}|\ln(\varepsilon+|\xi|)| + \|\nabla \mathfrak{a}\|_{L^{\infty}(\Omega)}\right)(\varepsilon+|\xi|)^{q_{\text{min}}}|x-y|.
$$}}

Finally, we must stress that assumption \eqref{Cont-H} takes into account just in order to prove the Comparison Principle. In the whole paper, we make use only assumptions (A0)-(A2) and (A4)-(A5).
\end{remark}

Now, we are in a position to deliver the proof of Comparison Principle.

\begin{theorem}[{\bf Comparison Principle}]\label{comparison principle}
	Assume that assumptions (A0)-(A2), (A4)-(A5) and \eqref{Cont-H} are in force. Let $f \in C^0(\overline{\Omega})$. Suppose $u$ and $v$ are respectively a viscosity supersolution and 	subsolution of \eqref{appro_problem}. If $u \leq v$ on $\partial \Omega$, then $u \leq v$ in $\Omega$.
\end{theorem}
\begin{proof}
	We shall prove this result by contradiction. For this end, suppose that such a statement is false. Then, we can assume that
	$$
	\mathrm{M}_0 \defeq \max_{x\in\overline{\Omega}}(u-v)(x)>0.
	$$
	
	For $\delta>0$, let us define
	$$
	\mathrm{M}_{\delta}\defeq \max_{x,y\in\overline{\Omega}}\left[u(x)-v(y)-\frac{|x-y|^2}{2\delta}\right].
	$$
	Assume that the maximum $\mathrm{M}_\delta$ is attained in a point $(x_\delta, y_\delta)\in \overline{\Omega}\times\overline{\Omega}$ and notice that $\mathrm{M}_\delta\geq \mathrm{M}_0$.
	
	From \cite[Lemma 3.1]{CIL92}, we have that
	\begin{equation}\label{limit-CP}
	\lim_{\delta\to0}\frac{|x_\delta-y_\delta|}{\delta}=0.
	\end{equation}
	This implies that $x_\delta,y_\delta\in\Omega$ for $\delta$ sufficiently small. Moreover, \cite[Theorem 3.2 and Proposition 3 ]{CIL92} assures the existence of a limiting super-jet $\left(\frac{x_\delta-y_\delta}{\delta}, \mathrm{X}\right)$ of $u$ at $x_\delta$ and a limiting sub-jet $\left(\frac{x_\delta-y_\delta}{\delta}, \mathrm{Y}\right)$ of $v$ at $y_\delta$, where the matrices $\mathrm{X}$ and $\mathrm{Y}$ satisfies the inequality
	\begin{equation}\label{matrices-ine}
		-\frac{3}{\delta}
		\begin{pmatrix}
			\mathrm{Id}_n& 0 \\
			0& \mathrm{Id}_n
		\end{pmatrix}
		\leq
		\begin{pmatrix}
			\mathrm{X}&0\\
			0&-\mathrm{Y}
		\end{pmatrix}
		\leq
		\frac{3}{\delta}
		\begin{pmatrix}
			\mathrm{Id}_n & -\mathrm{Id}_n \\
			-\mathrm{Id}_n& \mathrm{Id}_n
		\end{pmatrix},	
	\end{equation}
where $\mathrm{Id}_n$ is the identity matrix.

Observe that the assumption (A1) implies that the operator $F$ is, in particular, degenerate elliptic. It implies that,
\begin{equation}\label{MonotF}
  F(x, \mathrm{Y}) \le F(x, \mathrm{X})
\end{equation}
for every $x\in\Omega$ fixed, since $\mathrm{X}\leq \mathrm{Y}$ from \eqref{matrices-ine}.

Therefore, as a result of the two viscosity inequalities (and sentence \eqref{MonotF})
$$
	\mathcal{H}_{\varepsilon}\left(x_\delta, \frac{(x_\delta-y_\delta)}{\delta}\right)[\varepsilon u(x_\delta) + F(x_\delta, \mathrm{X})]  \leq  f(x_\delta)
$$
and
$$
\mathcal{H}_{\varepsilon}\left(y_\delta, \frac{(x_\delta-y_\delta)}{\delta}\right)[\varepsilon v(y_\delta) + F(y_\delta, \mathrm{Y})]  \geq  f(y_\delta)
$$
we get
\begin{equation}\label{estl-CP}
  \begin{array}{rcl}
    \frac{\varepsilon \mathrm{M}_0}{2} & \le & \varepsilon (u(x_\delta)-v(y_\delta)) \\
     & \le & [F(x_\delta, \mathrm{Y})-F(y_\delta, \mathrm{Y})] + \frac{f(x_\delta)}{\mathcal{H}_{\varepsilon}\left(x_\delta, \frac{(x_\delta-y_\delta)}{\delta}\right)}-\frac{f(y_\delta)}{\mathcal{H}_{\varepsilon}\left(y_\delta, \frac{(x_\delta-y_\delta)}{\delta}\right)}
  \end{array}
\end{equation}

Now, observe that, from assumption (A2), we can estimate
\begin{equation}\label{est2-CP}
|F(x_\delta, \mathrm{Y})-F(y_\delta, \mathrm{Y})|\leq \mathrm{C}_{\mathrm{F}}\omega(|x_\delta-y_\delta|)\|\mathrm{Y}\|.
\end{equation}
Moreover, if $\omega_f$ is a modulus of continuity of $f$ on $\overline{\Omega}$, from assumptions \eqref{1.2} and \eqref{Cont-H} we obtain

{\scriptsize{
\begin{equation}\label{est3-CP}
	\begin{aligned}
		&\frac{f(x_\delta)}{\mathcal{H}_{\varepsilon}\left(x_\delta, \frac{(x_\delta-y_\delta)}{\delta}\right)}-\frac{f(y_\delta)}{\mathcal{H}_{\varepsilon}\left(y_\delta, \frac{(x_\delta-y_\delta)}{\delta}\right)}\\
		\leq&\frac{f(x_\delta)-f(y_\delta)}{\mathcal{H}_{\varepsilon}\left(x_\delta, \frac{(x_\delta-y_\delta)}{\delta}\right)}+f(y_\delta)\left[\frac{1}{\mathcal{H}_{\varepsilon}\left(x_\delta, \frac{(x_\delta-y_\delta)}{\delta}\right)}-\frac{1}{\mathcal{H}_{\varepsilon}\left(y_\delta, \frac{(x_\delta-y_\delta)}{\delta}\right)}\right]\\
		\leq&\frac{\omega_f(|x_\delta-y_\delta|)}{L_1 \cdot \mathcal{K}_{p, q, \mathfrak{a}}^{\varepsilon}\left(x_\delta, \frac{|x_\delta-y_{\delta}|}{\delta}\right)}\\
		+&\|f\|_{L^{\infty}(\Omega)}\frac{\left|\mathcal{H}_{\varepsilon}\left(y_\delta, \frac{(x_\delta-y_\delta)}{\delta}\right)-\mathcal{H}_{\varepsilon}\left(x_\delta, \frac{(x_\delta-y_\delta)}{\delta}\right)\right|}{L^2_1 \cdot \mathcal{K}_{p, q, \mathfrak{a}}^{\varepsilon}\left(x_\delta, \frac{|x_\delta-y_{\delta}|}{\delta}\right)\mathcal{K}_{p, q, \mathfrak{a}}^{\varepsilon}\left(y_\delta, \frac{|x_\delta-y_{\delta}|}{\delta}\right)}\\
		\leq& \mathrm{C}_{p, q, \mathfrak{a}}\|f\|_{L^{\infty}(\Omega)}\frac{ \left(\varepsilon+\frac{|x_\delta-y_{\delta}|}{\delta}\right)^{p_{\text{min}}}\left(\left|\ln\left(\varepsilon+\frac{|x_\delta-y_{\delta}|}{\delta}\right)\right|+1\right)\omega_{p, q, \mathfrak{a}}(|x_{\delta}-y_{\delta}|)}{L^2_1 \left(\varepsilon+\frac{|x_\delta-y_{\delta}|}{\delta}\right)^{2q_{\text{max}}}}\\
 + & \frac{\omega_f(|x_\delta-y_\delta|)}{L_1 \left(\varepsilon+\frac{|x_\delta-y_{\delta}|}{\delta}\right)^{q_{\text{max}}}}
	\end{aligned}
\end{equation}}}

Therefore, combining the sentences \eqref{estl-CP}, \eqref{est2-CP} and \eqref{est3-CP} we conclude
$$
\begin{array}{rcl}
  \frac{\varepsilon \mathrm{M}_0}{2} & \le & \mathrm{C}_{\mathrm{F}}\omega(|x_\delta-y_\delta|)\|\mathrm{Y}\| + \frac{\omega_f(|x_\delta-y_\delta|)}{L_1 \left(\varepsilon+\frac{|x_\delta-y_{\delta}|}{\delta}\right)^{q_{\text{max}}}}\\
   & + & \mathrm{C}_{p, q, \mathfrak{a}}\|f\|_{L^{\infty}(\Omega)}\frac{ \left(\varepsilon+\frac{|x_\delta-y_{\delta}|}{\delta}\right)^{p_{\text{min}}}\left(\left|\ln\left(\varepsilon+\frac{|x_\delta-y_{\delta}|}{\delta}\right)\right|+1\right)\omega_{p, q, \mathfrak{a}}(|x_{\delta}-y_{\delta}|)}{L^2_1 \left(\varepsilon+\frac{|x_\delta-y_{\delta}|}{\delta}\right)^{2q_{\text{max}}}} \\
   & = & \text{o}(1) \quad \text{as}\quad \delta \to 0^{+},
\end{array}
$$
which yields a contradiction, thereby proving the desired result.
\end{proof}

In this point, we are ready to prove the existence of viscosity solution to \eqref{appro_problem}. The proof can be obtained employing the ideas of the Perron's method.

\begin{lemma}[{\bf Existence of sub/supersolutions}]\label{existence_sub_super}
	Assume that assumptions (A0)-(A5) are in force. Let $\Omega$ a bounded domain satisfying a uniform exterior sphere condition and $f \in C^0(\overline{\Omega})$. Then, for every $\varepsilon \in (0, 1)$, there exist a viscosity subsolution $u_1\in C^0(\overline{\Omega})$ and a viscosity supersolution $u_2\in C^0(\overline{\Omega})$ to \eqref{appro_problem} satisfying $u_1=u_2=g$ on $\partial \Omega$.
\end{lemma}

\begin{proof}
	For every $\varepsilon \in (0, 1)$, let us prove the existence of a continuous viscosity subsolution $u_1$ of \eqref{appro_problem} agreeing with $g$ on $\partial\Omega$. The supersolution case follows from an analogous argument.	
	
	 Firstly, we will construct a global supersolution to \eqref{appro_problem}. For this end, set $	 \mathfrak{C}_0 \defeq \frac{\|f\|_{L^\infty(\Omega)}}{\mathrm{L}_1}$ and choose a point $x_0$ such that $\dist(x_0,\partial\Omega)\geq 1$. Define
	 \[
	 \mathrm{M}_{\ast} \defeq \max\{\mathfrak{C}_0,\lambda n\}
	 \]
	 and a function
	\[
	v_1(x) \defeq \overline{\mathrm{M}}-\frac{\mathrm{M}_{\ast}}{2\lambda n}|x-x_0|^2,
	\]
	where $\overline{\mathrm{M}}$ is chosen such that $v_1>\|g\|_{L^\infty(\partial\Omega)}$ on $\partial\Omega$.
	
	Observe that, for every $x\in\Omega$
	{\scriptsize{
	\begin{equation*}
		\begin{aligned}
			&\mathcal{H}(x, \varepsilon+\nabla v_1(x))[\varepsilon v_1(x) + F(x, D^2 v_1(x))]\\
			&\geq L_1 \cdot \mathcal{K}_{p, q, \mathfrak{a}}(x, \varepsilon+|\nabla v_1(x)|)[\varepsilon v_1(x) + F(x, D^2 v_1(x))]\\
			&=\mathrm{L}_1(\varepsilon+|\nabla v_1(x)|)^{p(x)}[\varepsilon v_1(x) + F(x, D^2 v_1(x))]+\mathfrak{a}(x)\mathrm{L}_1(\varepsilon+|\nabla v_1(x)|)^{q(x)}[\varepsilon v_1(x) + F(x, D^2 v_1(x))]\\
			&\geq \mathrm{L}_1\mathrm{M}_{\ast}(1+\mathfrak{a}(x))\\
			&\geq \mathrm{L}_1\mathrm{M}_{\ast}\\
			&\geq f(x).
		\end{aligned}
	\end{equation*}}}
	Now, let $z\in\partial\Omega$ be a fixed point. Since $\Omega$ satisfies a uniform exterior sphere condition, we can choose $x_z$ such that $\overline{B_r(x_z)}\cap\overline \Omega=\{z\}$ with $r=|z-x_z|$. Next, set $\mathrm{R} \defeq r+\diam(\Omega)$. In what follows, we define the function
	\[
	v_z(x)\defeq \mathrm{K}(r^{-\alpha_0}-|x-x_z|^{-\alpha_0})
	\]
	where $\alpha_0>2, \mathrm{K}>0$.
	Notice that $v_z(z)=0$, $v_z(x)>0$ in $\Omega$ and
	\[
	Dv_z(x)=\mathrm{K}\alpha_0 \frac{x-x_z}{|x-x_z|^{\alpha_0+2}}
	\]
	and
	\[
	D^2v_z(x)=\mathrm{K}\alpha_0 \frac{\mathrm{Id}_n}{|x-x_z|^{\alpha_0+2}}-\mathrm{K}\alpha_0 (\alpha_0+2)\frac{(x-x_z)\otimes(x-x_z)}{|x-x_z|^{\alpha_0+4}}.
	\]
	Hence,
	\[
	|Dv_z|\geq \mathrm{K}\alpha_0\frac{1}{\mathrm{R}^{1+\alpha_0}}\quad \mbox{in}\quad \Omega.
	\]
	Now, let us observe that, if $\lambda(\alpha_0+2)-n\Lambda\geq 1$, we obtain
	\begin{equation*}
		\begin{aligned}
			F(D^2v_z(x))&\geq \mathrm{K}\alpha_0 \frac{1}{|x-x_z|^{\alpha_0+2}}(\lambda(\alpha_0+2)-n\Lambda)\\
			&\geq \mathrm{K}\alpha_0 \frac{1}{|x-x_z|^{\alpha_0+2}}.
		\end{aligned}
	\end{equation*}
	In the sequel, set $\mathrm{K}>0$ satisfying
$$
\mathrm{K} \geq \frac{1}{\alpha_0}\max\left\{\mathrm{R}^{1+\alpha_0}, \,\mathrm{R}^{1+\alpha_0}\left(\mathfrak{C}_0+\|g\|_{L^\infty(\partial\Omega)}\right)\right\}.
$$	

	Now, let $\delta \in (0, 1)$ be fixed. For constants $\mathrm{C}_\delta\geq 1$, define the functions
	\[
	v_{z,\delta}(x) \defeq g(z)+\delta+\mathrm{C}_\delta v_z(x),
	\]
	such that $v_{z,\delta}\geq g$ on $\partial\Omega$. We stress that the constants $C_\delta$ dependent only on the modulus of continuity of $g$ and are independent of $z$.
	
	Therefore,
{\scriptsize{
$$
\begin{array}{rl}
  \mathcal{H}(x, \varepsilon+\nabla v_{z,\delta}(x))[\varepsilon v_{z,\delta}(x) + F(x, D^2 v_{z,\delta}(x))] & \geq \mathrm{L}_1(\varepsilon+|\nabla v_{z,\delta}(x)|)^{p(x)}[\varepsilon v_{z,\delta}(x) + F(x, D^2 v_{z,\delta}(x))] \\
   & \geq \mathrm{L}_1\left(-\|g\|_{L^\infty(\partial\Omega)}+\mathrm{C}_\delta \mathrm{K}\alpha_0\frac{1}{\mathrm{R}^{2+\alpha_0}}\right) \\
   & \geq \mathrm{L}_1\mathrm{K} \\
   & \ge f(x) \quad \text{for every} \quad \varepsilon,\delta \in (0, 1)\quad \text{and} \quad x\in\Omega.
\end{array}
$$}}

 This means that, for every $\varepsilon,\delta \in (0, 1)$ and $z\in\partial\Omega$, the functions $v_{z,\delta}$ are subsolutions of \eqref{appro_problem}. Hence,
	\[
	\hat{v}_{z,\delta}(x)\defeq \min\{v_{z,\delta}(x),v_1(x)\}
	\]
	are viscosity subsolutions of \eqref{appro_problem}. Observe that, if we define
	\[
	u_1(x)\defeq \inf\{\hat{v}_{z,\delta}(x):\,z\in\partial\Omega \quad\text{and} \quad \delta \in (0, 1)\},
	\]
	we can conclude that $u_1$ is a viscosity subsolution to \eqref{appro_problem} such that $u_1=g$ on $\partial\Omega$.
\end{proof}

Finally, by combining Theorem \ref{comparison principle}, Lemma \ref{existence_sub_super} and the Perron's method we can derive the existence of a viscosity solution to the approximating equations \eqref{appro_problem}. Finally, the desired existence of a viscosity solution to \eqref{1.1} is obtained.

\begin{proof}[{\bf Proof of Theorem \ref{existence}}]
	Let $\varepsilon \in (0, 1)$ fixed. From the above discussion, there exists a viscosity solution $u_\varepsilon$ to \eqref{appro_problem} such that $u_1\leq u_{\varepsilon}\leq u_2$ and $u_\varepsilon=g$ on $\partial\Omega$.

Since the sequence $(u_\varepsilon)_{0<\varepsilon<1}$ is uniformly bounded in $\mathcal{C}^{0, \gamma}_{loc}(\Omega)$ (see \cite{daSRRV21} and \cite{DSVI}), through a subsequence if necessary, we have that $u_\varepsilon$ converges to a function $u_{\infty}$ as $\varepsilon\to0$. From standard stability results (see Lemma \ref{P1}) we conclude that $u_\infty$ solves
	\[
	\mathcal{H}(x, Du_{\infty})F(x,D^2u_\infty)=f(x)\quad \mbox{in}\quad\Omega, \quad \text{with}\quad  u_\infty=g\,\,\, \text{on}\,\,\, \partial\Omega.
	\]
\end{proof}


	\subsection{Global equi-continuity of solutions: Proof of Theorem \ref{Xizao}}
	
	In this intermediate section we establish global H\"{o}lder and Lipschitz estimates (equi-continuity) for viscosity solutions of
		\begin{equation} \label{3.1}
			\left\{
			\begin{array}{rcrcl}
				\mathcal{H}(y, Du + \zeta)F( D^2 u) & = & f(y) & \text{in} & B_1(x) \cap \{y_n >\phi(y^{\prime})\}\\
				u(y) & = & g(y) & \text{on} & B_1(x) \cap \{y_n = \phi(y^{\prime})\},
			\end{array}
			\right.
		\end{equation}

	As a matter of fact, such equi-continuity is a direct consequence of a series of auxiliary Lemmas, which yield compactness with respect to topology of uniform convergence.

Differently from \cite{Wint09} (see also \cite{daSN21}), we perform a simplest approach, which allows us to carry out the proof without using a change of variable scheme of viscosity solutions (a diffeomorphism for flat boundaries), which is a perfect fit for uniformly elliptic operators, but more technical for degenerate ones (cf. \cite{BV21}). In fact, the strategy of proofs used in this session, follows with corresponding adjustments to our unbalanced degeneracy setting, the ones from Berindelli-Demengel's works \cite{BD2} and \cite{BD16}.

    The first step towards establishing a Global H\"{o}lder regularity is obtaining a uniform estimate in terms of distance map to the hyper-surface points $\left\{y_{n}=\phi(y^{\prime}) \right\}$.

	\begin{lemma}\label{comparelemma} Let $g \in C^{1,\beta_g}$ be a boundary datum and $\phi$ a function as above. If $d_{\Omega}: \overline{\Omega} \to [0, +\infty)$ is the distance to the hyper-surface $\{y_{n}=\phi(y^{\prime})\}$, then, for every $\gamma, r \in (0, 1)$, there exist $\delta_{0}  = \delta_0(\lvert \lvert f \rvert\rvert_{\infty}, \lambda, \Lambda, p_{\text{min}}, q_{\text{max}}, r, \Omega, \mathrm{Lip}_{g}(\partial \Omega))$, such that for every $\delta \in (0, \delta_{0})$, if $u$ is a normalized viscosity solution, i.e. $\|u\|_{L^{\infty}(\Omega)}\leq 1$, of
		
\begin{equation}\label{dirichlet}
  		\left\{
		\begin{array}{rcl}
			\mathcal{H}(y, \nabla u)F(y,D^{2}u)=f(y) \ &\text{in}&  \mathrm{B}\cap\{y_{y}>\phi(y^{\prime})\} \\
			u(y) = g(y) 			& \text{on}&  \mathrm{B}\cap\{y_{n}=\phi(y^{\prime})\}
		\end{array}
		\right.
\end{equation}
 then the following estimate holds
		\begin{equation*}
			\lvert u(y^{\prime},y_{n})-g \rvert \leq \frac{2}{\delta}\frac{d_{\Omega}(y)}{1+d^{\gamma}_{\Omega}(y)} \quad \text{in} \quad B_{r}(0)\cap\{y_{n}>\phi(y^{\prime})\}.
		\end{equation*}
	\end{lemma}

	\begin{proof}
		In the sequel, we will split the proof in two cases: when $g\equiv 0$ and the case where $g$ is not identically null.

In the first case, from ABP estimate (Theorem \ref{ABPthm}) we may suppose that $\|u \|_{L^{\infty}(\Omega)}\leq 1$, since by hypothesis the boundary data is null and we may consider that $f$ has small norm (via smallest regime).
		
As to the remainder of the proof, it is enough to analyze the set $\Omega_{\delta} \defeq \{y \in \Omega: d(y)<\delta\}$, since the complementary setting follows from the fact that $\|u \|_{L^{\infty}(\Omega)}\leq 1$ combined with local H\"{o}lder estimates (\ref{HoldEstThm})
{\scriptsize{
$$
\|u\|_{C^{0, \alpha^{\prime}}(\Omega^{\prime})} \leq \mathrm{C}(\verb"universal", \dist(\Omega^{\prime}, \partial \Omega))\left(\|u \|_{L^{\infty}(\Omega)} + \max\left\{\left\|\frac{f}{1+\mathfrak{a}}\right\|^{\frac{1}{{p_{\text{min}}+1}}}_{L^n(\Omega)}, \left\|\frac{f}{1+\mathfrak{a}}\right\|^{\frac{1}{{q_{\text{max}}}+1}}_{L^n(\Omega)}\right\}\right).
$$}}

		At first, let us consider $\delta <\eta$ such that $d_{\Omega}(y)<\eta$. It is well-known that $d_{\Omega} \in C^{2}$. Additionally,  $\| D^{2} d_{\Omega} \|_{L^{\infty}(\Omega)} \le \mathrm{K}$. In the following lines, we must build up a viscosity super-solution such that
		\begin{equation}\label{Gzao}
			\mathcal{H}(y, \nabla v)\mathscr{M}^{+}_{\lambda,\Lambda}(D^{2} v)<-\|\lvert f \|_{L^{\infty}(\Omega)} \quad \text{in} \quad  B\cap\{y_{n}>\phi(y^{\prime})\} \cap \Omega_{\delta}
		\end{equation}
with the boundary condition
\begin{equation*}
		v \geq u \quad \text{on} \quad \partial\left(\mathrm{B}\cap\{y_{n}>\phi(y^{\prime})\} \cap \Omega_{\delta}\right)
	\end{equation*}

		We aim that the following profile is the desired function
		\begin{equation}
			v(y)=\left\{
			\begin{array}{lcl}
				\frac{2}{\delta}\frac{d_{\Omega}(y)}{1+d^{\gamma}_{\Omega}(y)} & \text{for} & \lvert y \rvert <r  \\
				\frac{2}{\delta}\frac{d_{\Omega}(y)}{1+d^{\gamma}_{\Omega}(y)}+\frac{1}{(1-r)^{3}}(\lvert y\rvert -r)^{3} & \text{for} &  \lvert y \rvert\geq r,
			\end{array}
			\right.
		\end{equation}
		
		As verified in \cite[Lemma 2.2]{BD2}, the boundary condition is fulfilled.
				
		Now, in order to check that $v$ is a viscosity super-solution of \eqref{Gzao}, let us observe that
		
		$$
		Dv=\left\{
		\begin{array}{lcl}
			\frac{2}{\delta}\frac{1+(1-\gamma)d^{\gamma}_{\Omega}}{(1+d^{\gamma}_{\Omega})^{2}}Dd_{\Omega} & \text{for} & \lvert y \rvert <r \\
			\frac{2}{\delta}\frac{1+(1-\gamma)d^{\gamma}_{\Omega}}{(1+d^{\gamma}_{\Omega})^{2}}Dd_{\Omega}+\frac{y}{\lvert y\rvert}\frac{1}{(1-r)^{2}}(\lvert y\rvert -r)^{2} & \text{for} & \ \lvert y \rvert\geq r
		\end{array}
		\right.
		$$
is non-degenerate, this is
		
$$
   \vert Dv \rvert\geq \frac{1}{4\delta} \quad \text{provided}  \quad \delta\leq \frac{1-r}{12}.
$$
In effect, we have that $v \in C^{2}$. Furthermore,
		
		$$
		\begin{array}{lcl} \displaystyle
			D^{2}v \displaystyle &=& -\Big(\frac{2\gamma d^{\gamma -1}_{\Omega}}{\delta}\Big)\frac{(1+\gamma)+(1-\gamma)d^{\gamma}_{\Omega}}{(1+d^{\gamma}_{\Omega})^{3}}Dd_{\Omega}\otimes Dd_{\Omega} + \frac{2}{\delta}\frac{1+(1-\gamma)d^{\gamma}_{\Omega}}{(1+d^{\gamma}_{\Omega})^{2}}D^{2}d_{\Omega} \\
			\displaystyle &+& \frac{3(\lvert y\rvert -r)}{(1-r)^{3}}\Big\{ \Big(\mathrm{Id}_n- \frac{y\otimes y}{\lvert y \rvert^{2}}\Big)\frac{(\lvert y \rvert -r)}{\lvert y\rvert}+2\frac{ y\otimes y}{\lvert y\rvert^{2}}\Big\}.
		\end{array}
		$$
		
		Hence,
		$$
		\begin{array}{lcl}
			\mathscr{M}^{+}_{\lambda,\Lambda}(D^{2}v) &\leq& \Big\{\frac{-2\gamma \lambda d^{\gamma-1}_{\Omega}}{\delta}\frac{(1+\gamma)+(1-\gamma)d^{\gamma}_{\Omega}}{(1+d^{\gamma}_{\Omega})^{3}}+\frac{2n\Lambda \mathrm{K}}{\delta}\frac{1+(1-\gamma)d^{\gamma}_{\Omega}}{(1+d^{\gamma}_{\Omega})^{2}}+ \frac{6\Lambda n}{(1-r)^{2}}\Big\}\\
			\displaystyle &\leq & \Big\{-(2\gamma \delta^{\gamma-2}\lambda)\frac{(1+\gamma)}{(1+\delta^{\gamma})^{3}}+\frac{2}{\delta}n\mathrm{K}\Lambda+\frac{6n\Lambda}{(1-r)^{2}} \Big\}
		\end{array}
		$$
		
		Now, in order to obtain \eqref{Gzao}, we must impose
		
		\begin{equation*}
		\mathscr{M}^{+}_{\lambda,\Lambda}(D^{2}v) <-\| f\|_{L^{\infty}(\Omega)}(\mathrm{L}_{1}\mathcal{K}_{p,q,\mathfrak{a}}(y,\lvert Dv\rvert)))^{-1},
		\end{equation*}		
		which it is satisfied precisely when (for $\delta<\min\left\{\eta, \frac{1-r}{12}\right\}$):
		$$
\frac{2\gamma\lambda(1+\gamma)\delta^{\gamma}}{\delta^{2}(1+\delta^{\gamma})^{3}}>\frac{2}{\delta}n\mathrm{K}\Lambda+\frac{6n\Lambda}{(1-r)^{2}}+\frac{\| f\|_{L^{\infty}(\Omega)}}{\mathrm{L}_{1}\min\Big\{\Big(\frac{1}{4\delta}\Big)^{q_{max}},\Big(\frac{1}{4\delta}\Big)^{p_{min}}\Big\}}
		$$

		From Comparison Principle (Theorem \ref{comparison principle}) $u \leq v$ in $B\cap\{y_{n}>\phi(y^{\prime})\} \cap \Omega_{\delta}$. To finish this case, we replace
		$v$ by $-v$ in the previous computations and we restricted the analysis to the set $B_{r}\cap\{y_{n}>\phi(y^{\prime})\}$.

Now, for the case $g \ne 0$, let us consider $\Psi$ the unique solution of
		
		$$
		\left\{
		\begin{array}{rcccl}
			\mathscr{M}^{+}_{\lambda,\Lambda}(D^{2} \Psi) & = &0  &\text{in} & \mathrm{B}\cap\{y_{n}>\phi(y^{\prime})\} \\
			\Psi & = & g & \text{on} & \mathrm{B}\cap \{y_{n}=\phi(y^{\prime})\}.
		\end{array}
		\right.
		$$
		
		Remember that $\Psi \in C^{1,\beta}(\mathrm{B}\cap\{y_{n}\geq \phi(y^{\prime})\})\cap C^{2}(\mathrm{B}\cap \{y_{n}> \phi(y^{\prime})\})$. Moreover, the function $\Psi$ can be chosen so that
$$
\left\{
\begin{array}{l}
  \|\Psi \|_{L^{\infty}(\mathrm{B}\cap\{y_{n}>\phi(y^{\prime})\})} \leq \| g\|_{L^{\infty}(\mathrm{B}\cap \{y_{n}=\phi(y^{\prime})\})} \leq 1  \\
  \|D \Psi \|_{L^{\infty}(\mathrm{B}\cap\{y_{n}>\phi(y^{\prime})\})} \leq \mathrm{C}(\lambda,\Lambda,n,\Omega)\| Dg\|_{L^{\infty}(\mathrm{B}\cap \{y_{n}=\phi(y^{\prime})\})}.
\end{array}
\right.
$$

Now, let us consider the function		
		$$
		w(y)=\left\{
		\begin{array}{lcl}
			\frac{2}{\delta}\frac{d_{\Omega}(y)}{1+d^{\gamma}_{\Omega}(y)} + \Psi(y) & \text{for} & \lvert y \rvert <r \\
			\frac{2}{\delta}\frac{d_{\Omega}(y)}{1+d^{\gamma}_{\Omega}(y)}+\frac{1}{(1-r)^{3}}(\lvert y\rvert -r)^{3}+\Psi(y)\ &\text{for}& \ \lvert y \rvert\geq r
		\end{array}
		\right.
		$$
		
		As in the previous case, we are able to show that
		
$$
		v \geq u \quad \text{on} \quad \partial (\mathrm{B}\cap\{y_{n}>\phi(y^{\prime})\}\cap \Omega_{\delta}).
$$

Finally, for $\delta>0$ sufficiently small, we have that $w$ fulfills \eqref{Gzao} as well. From this point, the proof follows as in \cite[Lemma 2.2]{BD2}, thereby finishing the Lemma.
	\end{proof}

	In the sequel, by combining the estimate from previous Lemma \ref{comparelemma} and the Ishii-Lions Lemma, we obtain H\"{o}lder regularity of the Dirichlet problem \eqref{dirichlet} (up to the boundary).

	\begin{proof}[{\bf Proof of Theorem \ref{Xizao}}] Let $r_{1} \in \left(r, \frac{1}{2}\right)$ be fixed. Without loss of generality, we suppose that $\| u \|_{\infty}\leq 1$. Let $x_{0} \in B_{r} \cap \{y_{n}>\phi(y^{\prime})\}$ and $\Xi$ the function, defined as
		\begin{equation}
			\Xi(x,y)=u(x)-u(y)-\mathrm{M}\lvert x-y \rvert^{\gamma}-\mathrm{L}\Big(\lvert x-x_{0}\rvert^{2}+\lvert y-x_{0} \rvert^{2} \Big)
		\end{equation}
		
	We are going to show that for $\mathrm{L}, \mathrm{M} \gg 1$ large enough (independent of $x_{0}$), we get
		
		\begin{equation}\label{nonpositive}
			\Xi(x,y) \leq 0 \quad \text{in} \quad (B_{r_{1}}\cap\{y_{n}>\phi(y^{\prime})\})^{2}
		\end{equation}
which will imply that $u \in C^{0, \gamma}(B_{r}\cap\{y_{n}>\phi(y^{\prime})\})$, when $x=x_{0}$ and allowing this vary. For to show \eqref{nonpositive} at those points where $y_{n}=\phi(y^{\prime})$, we use the Lemma \ref{comparelemma}, which ensures that there exists $\mathrm{K}_{0}>0$ such that
		
		\begin{equation*}
			\lvert u(x) - g(x^{\prime})\rvert \leq \mathrm{K}_{0}d(x,\partial \Omega) \quad \text{for} \quad x \in B_{r_{1}}\cap\{y_{n}>\phi(y^{\prime})\}.
		\end{equation*}

		In effect, since $\lvert x^{\prime}-y^{\prime} \rvert \leq \lvert x-y \rvert$, we have
{\scriptsize{		
		$$
		\begin{array}{lcl}
			\lvert u(x^{\prime},x_{n})-u(y^{\prime},\phi(y^{\prime})) \rvert \displaystyle & \leq & \lvert u(x^{\prime},x_{n})-u(x^{\prime},\phi(x^{\prime})) \rvert + \lvert u(x^{\prime},\phi(x^{\prime}))-u(y^{\prime},\phi(y^{\prime})) \rvert
			\\
			\displaystyle &\leq& \mathrm{K}_{0}d(x,\partial\Omega)+\mathrm{Lip}_{g}(\partial \Omega)\lvert x^{\prime}-y^{\prime}\rvert
			\\
			\displaystyle & \leq & \mathrm{K}_{0}\lvert x-(y^{\prime},\phi(y^{\prime})) \rvert + \mathrm{Lip}_{g}(\partial \Omega)\lvert x-(y^{\prime},\phi(y^{\prime})) \rvert
		\end{array}
		$$}}

As a result, if we select $\mathrm{M} > \mathrm{K}_{0}+\mathrm{Lip}_{g}(\partial \Omega)$ (large enough), then

		$$
		\begin{array}{lcl}\label{Gamado}
			\displaystyle \Xi(x,y)
			\displaystyle & \leq & \lvert u(x)-u(y)\rvert-\mathrm{M}\lvert x-y \rvert^{\gamma}-\mathrm{L}\Big(\lvert x-x_{0}\rvert^{2}+\lvert y-x_{0}\rvert^{2} \Big) \\
			\displaystyle & \leq & (\mathrm{K}_{0}+\mathrm{Lip}_{g}(\partial \Omega))\lvert x-y \rvert-\mathrm{M}\lvert x-y \rvert^{\gamma}-\mathrm{L}\Big(\lvert x-x_{0}\rvert^{2}+\lvert y-x_{0}\rvert^{2} \Big)
			\\
			\displaystyle & \leq & \mathrm{M}(\lvert x-y \rvert - \lvert x-y \rvert^{\gamma})-\mathrm{L}\Big(\lvert x-x_{0}\rvert^{2}+\lvert y-x_{0}\rvert^{2} \Big) \leq 0.
		\end{array}
		$$
		where do we get the thesis 
		$$
		\Xi(x,y)\leq 0 \quad \text{in} \quad (B_{r_{1}}\cap\{y_{n}=\phi(y^{\prime})\})^{2}.
        $$
		
		In order to verify the desired estimate on the rest of the boundary, it is enough to select $\mathrm{L}>\frac{4}{(r_{1}-r)^{2}}$ and remember that $\|u\|\leq 1$. Finally, if $\mathrm{L}$ is large enough and $\mathrm{M}$ is such that

		$$
		\mathrm{M} > \mathrm{C}\lvert\lvert f \rvert\rvert_{\infty}\lvert \hat{a}-\hat{b} \rvert^{2-\gamma}\Gamma^{-1},
		$$
where $\Gamma>0$ will be defined soon in \eqref{GAMMA}.

		In this point, let us suppose for sake of contradiction that $\Xi(x,y)>0$ for some $(x,y) \in (B_{r_{1}}\cap\{y_{n}>\phi(y^{\prime}) \})^{2}$, then there is $(\hat{a},\hat{b})$ such that
		
		\begin{equation}
			\Xi(\hat{a},\hat{b})= \sup\limits_{\overline{B_{r_{1}}}} \,\Xi(x,y)>0.
		\end{equation}
		
		It is worth to stress that $\hat{a} \ne \hat{b}$, otherwise the result holds true. Moreover, the choice of $\mathrm{L}$ guarantees that $\hat{a},\hat{b} \in B_{\frac{r_{1}+r}{2}} \cap\{y_{n}>\phi(y^{\prime})\}$. Thus, from the Ishii-Lions Lemma (see \cite[Theorem 3.2]{CIL92}) for every $\varepsilon >0$ depending on the norm of $\mathcal{Q}\defeq D^{2}(\mathrm{M}\lvert x-y\rvert^{\gamma})$, there exist $\mathrm{X} ,\mathrm{Y} \in \text{Sym}(n)$ such that
		
		\begin{equation}\label{superasati}
			(\gamma \mathrm{M}(\bar{a}-\bar{b})\lvert \hat{a}-\hat{b}\rvert^{\gamma -2}+2L(\hat{a}-x_{0}), \mathrm{X}) \in \overline{\mathcal{J}^{2,+}}u( \hat{a})
		\end{equation}
		
		\begin{equation}\label{inferaexistencia}
			(\gamma \mathrm{M}(\hat{a}-\hat{b})\lvert \hat{a}-\hat{b}\rvert^{\gamma -2}-2\mathrm{L}(\hat{b}-x_{0}),-\mathrm{Y}) \in \overline{\mathcal{J}^{2,-}}u( \hat{b})
		\end{equation}
		
		with
		
		\begin{equation}\label{matricialidentidade}
			\left(
			\begin{array}{cc}
				\mathrm{X} & 0 \\
				0 & \mathrm{Y} \\
			\end{array}
			\right)
			\leq  \left(
			\begin{array}{cc}
				\mathcal{Q} & -\mathcal{Q} \\
				-\mathcal{Q} & \mathcal{Q} \\
			\end{array}
			\right)+
			(2\mathrm{L}+\varepsilon)
			\left(
			\begin{array}{cc}
				\mathrm{Id}_{n} & 0 \\
				0 & \mathrm{Id}_{n} \\
			\end{array}
			\right),
		\end{equation}
		
		Next, let
$$
a_{x}=\gamma \mathrm{M}(\hat{a}-\hat{b})\lvert \hat{a}-\hat{b}\rvert^{\gamma -2}+2\mathrm{L}(\hat{a}-x_{0}) \quad \text{and} \quad b_{y}=\gamma \mathrm{M}(\hat{a}-\hat{b})\lvert \hat{a}-\hat{b}\lvert^{\gamma -2}-2\mathrm{L}(\hat{b}-x_{0}).
$$
Since $\gamma \in (0, 1)$ as soon as $2\mathrm{L}(r+r_{1})\leq \frac{\gamma}{2}\mathrm{M}r_{1}^{\gamma -1}$, we have that $2\mathrm{L}\lvert \hat{a} -x_{0}\rvert\leq \frac{\gamma \mathrm{M}}{2}\lvert \hat{a}-\hat{b}\rvert^{\gamma -1}$, i.e

				\begin{equation}\label{estimateab}
	\frac{1}{2}\gamma \lvert \hat{a}-\hat{b} \rvert^{\gamma-1} \le \lvert a_{x}\rvert, \lvert b_{y}\rvert \le 2\gamma \mathrm{M}\lvert \hat{a}-\hat{b} \rvert^{\gamma -1}.
\end{equation}
		
		Now, by using \eqref{inferaexistencia} and \eqref{superasati}, the uniform ellipticity, the continuity of the coefficients of $F$, \eqref{estimateab},  the upper estimates for the norms $\lvert\lvert X+Y\rvert\rvert$, $\lvert\lvert X\rvert\rvert$ and $\lvert\lvert Y\rvert\rvert$ which  can be found in \cite[Proposition 2.3]{BD2} and finally \eqref{GAMMA}:

	\begin{equation}\label{GAMMA}
		\lvert \mathcal{H}(\hat{a}, a_{x})-\mathcal{H}(\hat{b},b_{y}) \rvert \displaystyle \leq \Gamma(\gamma,\lvert \hat{a}-\hat{b}\rvert, \mathrm{L}_{2}, \mathrm{L}_{1},\lvert\lvert \mathfrak{a}\rvert\rvert_{\infty},p_{\text{min}},q_{\text{max}}),
		\end{equation}
where
	$$
	\begin{array}{lcl}
		\Gamma \displaystyle &\defeq & c_{p,q, \mathfrak{a}}(2\gamma \mathrm{M}\lvert \hat{a}-\hat{b}\rvert^{\gamma-1})^{p_{\text{min}}+1}\omega_{p,q, \mathfrak{a}}(\lvert \hat{a}-\hat{b}\rvert )\\
		\displaystyle &+&c(\lvert\lvert \mathfrak{a}\rvert\rvert_{\infty},\mathrm{L}_{2},\mathrm{L}_{1})\max\Big\{ (2\gamma \mathrm{M}\lvert \hat{a}-\hat{b}\rvert^{\gamma-1})^{p_{\text{min}}},(2\gamma \mathrm{M}\lvert \hat{a}-\hat{b}\rvert^{\gamma-1})^{q_{\text{max}}} \Big\}
	\end{array}
	$$

		Finally, the upper estimates for the norms $\lvert\lvert \mathrm{X}+\mathrm{Y}\rvert\rvert$, $\lvert\lvert \mathrm{X}\rvert\rvert$ and $\lvert\lvert \mathrm{Y}\rvert\rvert$  can be found in \cite[Proposition 2.3]{BD2}.
		
		\small{$$
			\begin{array}{lcl}
				f(\hat{b}) & \leq & \displaystyle \mathcal{H}(\hat{b},b_{y})F(\bar{b},-\mathrm{Y}) \\
				& \leq & \displaystyle (\mathcal{H}(\hat{a},a_{x})+\Gamma)F(\hat{b},-\mathrm{Y}) \\
				& \leq & \mathcal{H}(\hat{a},a_{x})F(\hat{b},-\mathrm{Y})-\Gamma\mathscr{M}^{-}_{\lambda,\Lambda}(\mathrm{Y}) \\
				& \leq & \displaystyle \Big(F(\hat{a}, \mathrm{X})+\lvert\lvert \mathrm{X}\rvert\rvert \mathrm{C}_{\mathrm{F}}\omega(\lvert \hat{a}-\hat{b}\rvert)-\mathscr{M}^{-}_{\lambda,\Lambda}(\mathrm{X}+\mathrm{Y})\Big)\mathcal{H}(\hat{a},a_{x})-\Gamma\mathscr{M}^{-}_{\lambda,\Lambda}(\mathrm{Y}) \\
				& \leq & \displaystyle f(\hat{a})+\mathcal{H}(\hat{a},a_{x})\Big(\mathrm{C}_{\mathrm{F}}\lvert\lvert \mathrm{X}\rvert\rvert\omega(\lvert \hat{a}-\hat{b} \rvert) -\mathscr{M}^{-}_{\lambda,\Lambda}(\mathrm{X}+\mathrm{Y})\Big)-\Gamma\mathscr{M}^{-}_{\lambda,\Lambda}(\mathrm{Y})\\
				&\leq& \displaystyle f(\hat{a})+\mathcal{H}(\hat{a}, a_{x})\Big(\mathrm{C}_{\mathrm{F}}\lvert\lvert \mathrm{X}\rvert\rvert\omega(\lvert \hat{a}-\hat{b} \rvert)-\lambda \mathrm{C}\mathrm{M}\lvert \hat{a}-\hat{b}\rvert^{\gamma-2}\Big)-\mathrm{C}\mathrm{M}\Gamma\lvert \hat{a}-\hat{b}\rvert^{\gamma-2}\\
				&\leq& \displaystyle f(\hat{a})-\mathrm{C}\mathrm{M}\Gamma\lvert \hat{a}-\hat{b}\rvert^{\gamma-2}
			\end{array}
			$$
		}
		which yields a contradiction, provided $\mathrm{L}$ and $\mathrm{M}$ are large enough.

	\end{proof}
	
	When we invoke, in the Approximation Lemma (boundary case) \ref{L0}, we shall invoke the following Lipschitz estimate's near the boundary for a switched equation (in the vector input).
	
	\begin{theorem}({\bf Lipschitz estimates for large deviations})\label{pgrande} Let $g$ be a Lipschitz continuous datum. Assume that $u$ satisfies
		$$
		\left\{
		\begin{array}{rcl}
			\mathcal{H}(x,pe_{n}+\nabla u)F(x,D^{2}u)=f \ &\text{in}& \ B_{1}(x)\cap\{y_{n}>0\} \\
			u = g \
			& \text{on}& \ B_{1}(x)\cap\{y_{n}=0\}
		\end{array}
		\right.
		$$
	with $\|u \|_{L^{\infty}(B_{1}(x)\cap\{y_{n}>0\})}\leq 1$ and $\lvert \lvert f\rvert\rvert_{L^{\infty}(B_{1}(x)\cap\{y_{n}>0\}\})}\leq \varepsilon<1$. Then, for all $r \in (0, 1)$, there exists $s_{0}=s_{0}(\lambda,\Lambda,n,q_{\text{max}},p_{\text{min}},r,\varepsilon_{0},\mathrm{Lip}_{g}(\partial \Omega))$, such that if $\lvert p \rvert > \frac{1}{s_{0}}$, then $u \in C^{0, 1}(B_{r}(x)\cap\{y_{n}>0\})$. Moreover,
$$
\|u\|_{C^{0, 1}(B_{r}(x)\cap\{y_{n}>0\})} \le   \mathrm{C}(\lambda,\Lambda,n,q_{\text{max}},p_{\text{min}},r,\varepsilon_{0},\mathrm{Lip}_{g}(\partial \Omega)).
$$
	\end{theorem}
	
	Before presenting the proof, we will need the following auxiliary Lemma.
	
	\begin{lemma}
		
		Suppose $g$ is Lipschitz continuous datum. For every $\gamma, r \in (0, 1)$, there exists $\delta=\delta(\lvert\lvert f\rvert\rvert, \lambda, \Lambda, q_{\text{max}},r , \mathrm{Lip}_{g}(\partial \Omega))$, such that for $\mathfrak{b}<\frac{\delta}{4}$, any normalized solution of
		
		$$
		\left\{
		\begin{array}{rcl}
			\mathcal{H}(x,e_{n}+\mathfrak{b}\nabla u)F(x,D^{2}u)=f \ &\text{in}& \ B_{1}(x)\cap\{y_{n}>0\} \\
			u = g \
			& \text{on}& \ B_{1}(x)\cap\{y_{n}=0\}
		\end{array}
		\right.
		$$
	satisfies
$$
 \lvert u(y^{\prime},y_{n})-g(y^{\prime}) \rvert \leq \frac{2}{\delta}\frac{y_{n}}{1+y_{n}^{\gamma}} \quad  \text{in} \quad  B_{r}(x)\cap\{y_{n}>0\}.
$$		
\end{lemma}
	
\begin{proof}
Firstly, we are going to suppose by simplicity $g\equiv 0$. If $\mathfrak{b}=0$ the result holds true. Hence, we will assume $\mathfrak{b} \ne 0$. Now replacing the distance of $y$ to the boundary by $y_{n}$; so we consider
		
		\begin{equation}
			v(y)=\left\{
			\begin{array}{lcl}
				\frac{2}{\delta}\frac{y_{n}}{1+y_{n}^{\gamma}(y)} & \text{for} & y_{n}<\delta ,\,\,\lvert y^{\prime} \rvert <r  \\
				\frac{2}{\delta}\frac{y_{n}}{1+y_{n}^{\gamma}(y)}+\frac{1}{(1-r)^{3}}(\lvert y^{\prime}\rvert -r)^{3} & \text{for} &  y_{n}<\delta,\,\,\lvert y^{\prime} \rvert\geq r,
			\end{array}
			\right.
		\end{equation}
		
		As in the Lemma \ref{comparelemma}, it is sufficient consider the set $\{y_{n}<\delta\}$, since the assumption $osc u\leq 1$ implies $\lvert\lvert u \rvert\rvert_{\infty}\leq 1$, the result holds elsewhere.

One more time, we intend to use the Comparison Principle (Theorem \ref{comparison principle}) and prove that $u\leq v$. The desired lower bound follows, considering $-v$ in place of $v$.

Now, notice that for $v$ to fulfil
		\begin{equation}
			\mathcal{H}(x,e_{n}+\mathfrak{b}\nabla u) \mathscr{M}^{+}_{\lambda,\Lambda}(D^{2}u)<-\lvert\lvert f\rvert\rvert_{\infty} \quad \text{in} \quad \mathrm{B}. 	
		\end{equation}
		it is sufficient to select $\delta$ such that
		$$
		(2\gamma \delta^{\gamma-2}\lambda)\frac{(1+\gamma)}{(1+\delta^{\gamma})^{3}}>\frac{2}{\delta}n\mathrm{C}_{1}\Lambda+\frac{6n\Lambda}{(1-r)^{2}}+\frac{2^{q_{max}}\lvert\lvert f\rvert\rvert_{\infty}}{\mathrm{L}_{1}}
		$$
for $\mathfrak{b}<\frac{\delta}{4}$, and recall that $\lvert \nabla v\rvert \leq \frac{2}{\delta}$. In addition, $v\geq u$ on $\partial(\mathrm{B}\cap\{0<y_{N}<\delta\})$. Consequently, by Comparison Principle the desired estimate is derived in $\{\lvert y^{\prime}\rvert<r,\,\,y_{n}>0\}$. In the case $g\ne 0$, we take the function $v$ as in the proof of Lemma \ref{comparelemma}, with $d_{\Omega}$ replaced by $y_{n}$.
	\end{proof}

We are in a position to prove the Theorem \ref{pgrande}.
	
	\begin{proof}[{\bf Proof of Theorem \ref{pgrande}}]
		Recall that $u$ is a viscosity solution of
		
		\begin{equation}\label{Hbarra}
			\mathcal{\hat{H}}(x,e_{n}+\mathfrak{b}\nabla u)F(x,D^{2}u)=\tilde{f}(x)
			\end{equation}
where $\mathfrak{\mathfrak{b}}=\frac{1}{p}$, $\tilde{f}(x)=\lvert p\rvert^{-p(x)} f(x)$, and $\mathcal{\hat{H}}(x,\xi)=\lvert p\rvert^{-p(x)}\mathcal{H}(x,p\xi)$ satisfies
		
		\begin{equation*}
			\mathrm{L}_{1}\mathcal{K}_{p,q, \hat{\mathfrak{a}}}(x,\lvert \xi \rvert) \leq \mathcal{\hat{H}}(x,\xi) \leq \mathrm{L}_{2}\mathcal{K}_{p,q, \hat{\mathfrak{a}}}(x,\lvert \xi \rvert)
		\end{equation*}
where $\hat{\mathfrak{a}}(x) \defeq \lvert p \rvert^{q(x)-p(x)}\mathfrak{a}(x)$. Indeed, if $\phi \in C^{2}(B_{1}(x)\cap\{y_{n}>0\})$, and $x_{0} \in B_{1}(x)\cap\{y_{n}>0\}$ such that $u-\phi$ has a local minimum at $x_{0}$, then by hypothesis
		
	$$
		\begin{array}{lcl}
			\mathcal{\hat{H}}(x_{0},e_{n}+\mathfrak{b}\nabla \phi)F(x_{0},D^{2}\phi) \displaystyle &=& 	\lvert p \rvert^{-p(x_{0})}\mathcal{H}(x_{0},p(e_{n}+\mathfrak{b}\nabla \phi)) F(x_0,D^{2}\phi) \\
			\displaystyle &=& \lvert p \rvert^{-p(x_{0})}\mathcal{H}(x_{0},pe_{n}+\nabla \phi)F(x_0,D^{2}\phi) \\
			\displaystyle &\leq& \lvert p\rvert^{-p(x_{0})}f(x_{0}),
		\end{array}
		$$
	i.e, $u$ is a viscosity super-solution of \eqref{Hbarra}. Analogously, $u$ is a viscosity sub-solution of \eqref{Hbarra}.

		 Now, let $r_{1} \in (r, 1)$ and let $x_{0} \in B_{r_{1}}(x)\cap\{y_{n}>0\}$, and $\mathrm{L}=\frac{4}{(r_{1}-r)^{2}}$, thus
		
		$$
\Xi(x,y)=u(x)-u(y)- \mathrm{M}\omega(\lvert x-y \rvert)-\mathrm{L}(\lvert x-x_{0} \rvert^{2}+\lvert y-x_{0} \rvert^{2}),
$$
where $\omega(s)=s-\omega_{0}s^{\frac{3}{2}}$ if $s\leq s_{0}=\Big(\frac{2}{3\omega_{0}}\Big)^{2}$ and $\omega(s)=\omega(s_{0})$ if $s\geq s_{0}$.	

Next, if we prove $\Xi(x,y)\leq 0$ in $B_{r_{1}}(x)$, since $\mathrm{M}$ is independent of $x_{0}$ , by choosing $x=x_{0}$, we get
		
$$
u(x_{0})-u(y)\leq \mathrm{M}\lvert x_{0}-y\rvert+\mathrm{L}\rvert x_{0}-y\rvert^{2}
$$

Now, if $y=x_{0}$ for all $x \in B_{r_{1}}$, we get
		
$$
  u(x)-u(x_{0})\leq \mathrm{M}\lvert x_{0}-x\rvert+ \mathrm{L}\rvert x_{0}-x\rvert^{2}
$$

		Thus, for $(x,y) \in B_{r}$, we have the desired result. Now, we begin to observe that if the supremum of $u$ is achieved in $(\hat{a},\hat{b}) \in \overline{B_{r}(x)}$, with our choice of $\mathrm{M}$, we ensure that neither $\hat{a}$ nor $\hat{b}$ can belong to the part $\{y_{n}=0\}$. The rest of the proof holds as in the proof of Theorem \ref{Xizao} as long as we choose $\delta$ small enough  such that
		$$
		(2\gamma \delta^{\gamma-2}\lambda)\frac{(1+\gamma)}{(1+\delta^{\gamma})^{3}}>\frac{2}{\delta}n\mathrm{K}\Lambda+\frac{6n\Lambda}{(1-r)^{2}}+\frac{2^{q_{max}}\lvert\lvert f\rvert\rvert_{\infty}}{L_{1}\max\{\lvert p\rvert^{p_{min}},\lvert p \rvert^{q_{max}}\}}
		$$
and $\mathfrak{b}\leq \frac{\delta}{4}$, thereby finishing the result.	
	\end{proof}

\section{ Proof of Theorems \ref{main1*} and \ref{main3}}\label{Sec3}

In this intermediate section, we will establish a sharp $C_{\text{loc}}^{1,\alpha}$ regularity estimates for viscosity solutions of
\begin{equation}\label{EqRegloc}
    \mathcal{H}(x, Du)F(x, D^2 u)  =  f(x) \quad  \text{in}  \quad \Omega
\end{equation}

Summarizing (for the sharp interior estimates) our approach is a byproduct of an oscillation-type mechanism (such a strategy relies on original ideas from \cite[Theorem 1]{ALS15} and \cite[Theorem 10]{LL17}, and afterwards considered in \cite{ATU17} and \cite{APR}, see also \cite{AdaSRT19}, \cite{daSR20} and \cite{PRS20}) combined with a localized argument, whose proof is conducted by analyzing two cases:

\begin{enumerate}
  \item[(A)] If $|D u| \ll 1$ with a controlled magnitude, then a perturbation of the $\mathfrak{F}-$harmonic profile leads to the inhomogeneous problem at the limit via a stability argument in a $C^1-$fashion.
  \item[(B)] On the other hand, if $|Du|\geq L_0>0$ (with a uniform lower bound), then classical estimates (see \cite{C89}, \cite{CC95} and \cite{Tru88}) can be enforced, since the problem becomes uniformly elliptic, i.e.
  $$
  \mathscr{M}_{\lambda, \Lambda}^-(D^2 u)\leq \mathrm{C}_0(L^{-1}_0, \|f\|_{L^{\infty}(\Omega)}) \,\, \text{and} \,\, \mathscr{M}_{\lambda, \Lambda}^+(D^2 u)\geq -\mathrm{C}_0(L^{-1}_0, \|f\|_{L^{\infty}(\Omega)}).
  $$

\end{enumerate}

Finally, in contrast with \cite[Theorem 3.1]{ART15}, \cite[Theorem 1.1]{APR}, \cite[Theorem 1]{DeF20} and \cite[Theorem 1]{IS}, our strategy does not make use of a suitable switched problem (a sort of deviation by planes), neither appeals to a blow-up argument as the one addressed in \cite[Lemma 9]{LL17}.

An essential tool we will use is the gradient \textit{a priori} estimate from \cite[Theorem 1.1]{FRZ21}, which we will state below for completeness.

\begin{theorem}[{\bf Gradient estimates}]\label{GradThm} Let $F$ be an operator satisfying (A0)-(A2) and let $u$ be a bounded viscosity solution to
$$
 \left[|Du|^{p(x)} +\mathfrak{a}(x)|Du|^{q(x)}\right] F(D^2u) = f \in L^{\infty}(B_1).
$$
Then,
\begin{equation*}
\|u\|_{C^{1,\gamma}\left(B_{\frac{1}{2}}\right)}\leq \mathrm{C}\cdot \left(\|u\|_{L^{\infty}(B_1)} +1+ \|f\|_{L^{\infty}(B_1)}^{\frac{1}{p_{\text{min}}+1}}\right)
\end{equation*}
for universal constants $\gamma \in (0, 1)$ and $\mathrm{C}>0$.
\end{theorem}

\begin{remark} Let us mention that under continuity assumption on the coefficients,
any viscosity solutions to
$$
 \left[|Du|^{p(x)} +\mathfrak{a}(x)|Du|^{q(x)}\right] F(x, D^2u) = f \in L^{\infty}(B_1).
$$
are locally of class $C^{1, \gamma}$, for some $\gamma \in (0, 1)$, depending on universal parameters.
\end{remark}

Next, as in \cite{daSR20} we define an appropriate class of solutions to our problem:

\begin{definition}
  \label{FineClass} Let $F$ be a fully nonlinear operator satisfying (A0)-(A2). For $\mathcal{H}$ and $f$ satisfying (A3)-(A5), we say that $u \in \mathcal{J}(F, \mathcal{H}, f)(B_1(0))$ provided
\begin{enumerate}
  \item $\mathcal{H}(x, Du)F(x, D^2 u) = f(x)$ in $B_1(0)$ in the viscosity sense.
  \item  $\|u\|_{L^{\infty}(B_1(0))}\leq 1$ in $B_1(0)$.
\end{enumerate}
\end{definition}

The first key step towards the proof of Theorem \ref{main1*} is to show that the solutions to \eqref{1.1}, in inner domains, can be approximated in a suitable manner by $\mathfrak{F}-$harmonic profiles, and that during such a process certain regularity properties of solutions are preserved.

At this point, we are in a position to enunciate the following Key Lemma (cf. \cite[Lemma 5.1]{ART15} and \cite[Lemma 4.1]{DeF20}):

\begin{lemma}[{\bf Approximation Lemma - Local version}]\label{lem.flat}
Let $\mathfrak{h} \in C^{0}\left(\overline{B_{\frac{1}{2}}(0)}\right)$ be the unique viscosity solution to
$$
\left\{
\begin{array}{rcrcl}
  F(D^2 \mathfrak{h}) & = & 0 & \text{in} & B_{\frac{1}{2}}(0) \\
  \mathfrak{h} & = & u & \text{on} & \partial B_{\frac{1}{2}}(0),
\end{array}
\right.
$$
Then, given $\iota \in (0, 1)$, there exists a $\delta= \delta(\iota, n, \lambda, \Lambda, p_{\text{min}}, q_{\text{max}} )>0$ such that if $u \in \mathcal{J}(F, \mathcal{H}, f)(B_{1}(0))$ with
$$
    \max\left\{\Theta_{\mathrm{F}}(x), \left\|f\right\|_{L^{\infty}(B_{1}(0))}\right\} \leq \delta
$$
then
\begin{equation}\label{eq.flat}
     \max\left\{\|u_k-\mathfrak{h}\|_{L^{\infty}\left(B_{\frac{1}{2}}(0)\right)}, \|Du_k-D\mathfrak{h}\|_{L^{\infty}\left(B_{\frac{1}{2}}(0)\right)}\right\} \leq \iota.
\end{equation}
\end{lemma}

\begin{proof}
  The proof follows the same lines as the one in \cite[Lemma 2.3]{daSR20}. For this reason, we omit it here in order to avoid a unnecessary repetition.
\end{proof}

\begin{remark}[{\bf Smallness regime}]\label{SmallRegime} Let us argue on the scaling character of our problem, which enables us to put the proof of Theorem \ref{main1} under the assumptions of Approximation Lemma \ref{lem.flat}. Let $u$ be a viscosity solution of \eqref{1.1}. Fix a point $x_0 \in \Omega^{\prime} \Subset \Omega$, we define $v: B_1(0) \to \R$ as follows
$$
    v(x) = \frac{u(\tau x+ x_0)}{\kappa}
$$
for parameters $\kappa, \tau>0$ to be determined later. It is easy to verify that $v$ fulfills (in the viscosity sense)
$$
\mathcal{H}_{\kappa, \tau}(x, D v) F_{\kappa, \tau}(x, D^2v) =  f_{\kappa, \tau}(x)  \textrm{ in } B_1(0),
$$
where
\[
\left\{
\begin{array}{rcl}
  F_{\kappa, \tau}(x, X) & \defeq & \frac{\tau^{2}}{\kappa}F\left(x_0+\tau x, \frac{\kappa}{\tau^{2}} X\right) \\
  f_{\kappa, \tau}(x) & \defeq & \frac{\tau^{p(x_0+\tau x)+2}}{\kappa^{p(x_0+\tau x)+1}}f(x_0+\tau x)\\
  \mathfrak{a}_{\kappa, \tau}(x) & \defeq & \left(\frac{\tau}{\kappa}\right)^{p(x_0+\tau x)-q(x_0+\tau x)}\mathfrak{a}(x_0+\tau x)\\
  \mathcal{H}_{\kappa, \tau}(x, \xi) &\defeq & \left(\frac{\tau}{\kappa}\right)^{p(x_0+\tau x)}\mathcal{H}\left(x_0+\tau x, \frac{\kappa}{\tau}\xi\right)\\
  \mathcal{K}_{p, q, \mathfrak{a}}^{\kappa, \tau}(x, |\xi|)&\defeq& |\xi|^{p(x_0+\tau x)}+\mathfrak{a}_{\kappa, \tau}(x)|\xi|^{q(x_0+\tau x)}.
\end{array}
\right.
\]
Hence, $F_{\kappa, \tau}$ fulfils the structural assumptions (A0) and (A1). Moreover,
$$
L_1 \cdot \mathcal{K}_{p, q, \mathfrak{a}}^{\kappa, \tau}(x, |\xi|)\leq  \mathcal{H}_{\kappa, \tau}(x, \xi) \leq L_2 \cdot \mathcal{K}_{p, q, \mathfrak{a}}^{\kappa, \tau}(x, |\xi|) \quad \text{for} \quad (x, \xi) \in \Omega \times \R^n.
$$
Now, for given $\iota \in (0, 1)$, which will be sufficiently small but fixed, let $\delta_{\iota}>0$ be the universal constant in the statement of Approximation Lemma \ref{lem.flat}. Then, we choose
$$
 \kappa \defeq \|u\|_{L^{\infty}(\Omega)} + 1 +\delta_{\iota}^{-1}\|f\|^{\frac{1}{p_{\text{min}}+1}}_{L^{\infty}(\Omega)}\\
$$
and
$$
 \tau = \min \left\{\frac{1}{2},\, \frac{1}{4}\dist(\Omega^{\prime},\, \partial \Omega), \left(\frac{\delta_{\iota}}{\|f\|_{L^{\infty}(\Omega)}+1}\right)^{\frac{1}{p_{\text{min}}+2}},\, \omega^{-1}\left(\frac{\delta_{\iota}}{\mathrm{C}_{\mathrm{F}}+1}\right)\right\}.
$$
Therefore, with such choices, $v$,  $F_{\kappa, \tau}$ and $f_{\kappa, \tau}$ fall into the framework of Approximation Lemma \ref{lem.flat}.
\end{remark}

The next Lemma provides the first step of the geometric iteration, under a proper control on growth of the gradient:

\begin{lemma}\label{lem.firststep}
Under the assumptions of Lemma \ref{lem.flat} there exists $\rho\in \left(0,\frac{1}{2}\right)$ such that
\begin{equation}\label{Aproxcond}
\displaystyle \sup_{B_{\rho}(0)} \frac{\left|u(x)-l_{0} u(x)\right|}{\rho^{1+\alpha}}\leq 1.
\end{equation}\label{2.4}
where $\mathfrak{l}_{0} u(x) = u(0)+D u(0)\cdot x$.
\end{lemma}

\begin{proof}
Let $\iota\ll 1$ to be chosen \textit{a posteriori}. From Approximation Lemma \ref{lem.flat} we know that there exists $\delta_\iota>0$, such that whenever
\begin{equation}\label{EqSmallCond}
 \max\left\{\Theta_{\mathrm{F}}(x), \left\|f\right\|_{L^{\infty}(B_{1}(0))}\right\}\leq\delta_{\iota},
\end{equation}
then \eqref{eq.flat} holds. For $\rho\in \left(0,\frac{1}{2}\right)$ to be fixed soon and $x\in B_{\rho}(0)$ we compute
\[
   \left|u(x)-\mathfrak{l}_{0} u(x)\right| \leq |u(x)-\mathfrak{h}(x)| + |\mathfrak{h}(x)-\mathfrak{l}_{0} \mathfrak{h}(x)|+ |\mathfrak{h}(0)-u(0)| + |(D\mathfrak{h}-D u)(0)\cdot x|
\]
so that
\[
\sup_{B_\rho(0)}\left|u(x)-\mathfrak{l}_{0} u(x)\right| \leq \sup_{B_\rho(0)}|\mathfrak{h}(x)-\mathfrak{l}_{0} \mathfrak{h}(x)|+ 3\iota
\]
provided that \eqref{EqSmallCond} there holds.

Now, according to the available regularity theory to homogeneous problem with ``frozen coefficients'' (see, \cite{C89}, \cite{CC95} and \cite{Tru88}) we have
$$
   \displaystyle |\mathfrak{h}(x)-\mathfrak{l}_0 \mathfrak{h}(x)|   \leq  \mathrm{C}(n, \lambda, \Lambda)\cdot|x|^{1+\alpha_{\mathrm{F}}} \quad \forall\,\, x \in B_{\frac{1}{2}}(0),
$$
where $\mathrm{C}>0$ and $\alpha_{\mathrm{F}} =  \alpha_{\mathrm{F}}(n, \lambda, \Lambda)\in (0, 1]$. Finally,
$$
\begin{array}{rcl}
\displaystyle \sup_{B_\rho(0)}\left|u(x)-\mathfrak{l}_{0} u(x)\right|&\le &  \mathrm{C}(n, \lambda, \Lambda)\cdot\rho^{1+\alpha_{\mathrm{F}}}+3\iota\\
&\leq & \rho^{1+\alpha},
\end{array}
$$
as long as we make the following universal choices:
\begin{equation}\label{2.5}
  \rho \in \left(0, \min\left\{\frac{1}{2}, \,\left(\frac{4}{5\mathrm{C}(n, \lambda, \Lambda)}\right)^{\frac{1}{\alpha_{\mathrm{F}}-\alpha}}\right\}\right) \quad \text{and} \quad   \iota \in  \left(0, \frac{1}{15}\rho^{1+\alpha}\right)
\end{equation}
Therefore, we obtain \eqref{Aproxcond}, thereby finishing the proof.
\end{proof}

Different from scenario of second order operators, using the previous lemma we can no longer proceed with an iterative scheme, since the operator's degeneracy law is not translating invariance by affine maps. Thus, the following simple consequence will provide the correct quantitative information.

\begin{lemma}[{\bf $1^{st}$ step of induction}]\label{c3.1}
Suppose that the assumptions of Lemma \ref{lem.firststep} are in force. Then, for $\rho>0$ fulfilling \eqref{2.5} we have
$$
\displaystyle \sup_{B_{\rho}(0)}\left|u(x)-u(0)\right|\leq\rho^{1+\alpha}+\rho|D u(0)|.
$$
\end{lemma}

Next, we will obtain the precise control on the influence of the gradient of $u$, we iterate (in an induction procedure) solutions in suitable dyadic balls. The proof is slightly similar to the one in \cite[Theorem 3.1]{AdaSRT19} and \cite[Lemma 2.5]{daSR20}. For this reason, we will omit it here.

\begin{lemma}[{\bf $k^{th}$ step of induction}]\label{lem.dy} Under the assumptions of Lemma \ref{lem.firststep} one obtain
\begin{equation}\label{3.3}
\displaystyle\sup_{B_{\rho^k}(0)}\left|u(x)-u(0)\right|\leq\rho^{k(1+\alpha)}+|D u (0)|\sum_{j=0}^{k-1}\rho^{k+j\alpha}.
\end{equation}
\end{lemma}

Next result provides a regularity estimate inside the singular zone.

\begin{lemma}\label{l3.3}
Suppose that the assumptions of Lemma \ref{lem.dy} are in force. Then, there exists a constant $\mathrm{M}_0(\verb"universal")>1$ such that
$$
\displaystyle \sup_{B_{r}(0)} \frac{|u(x)-u(0)|}{r^{1+\alpha}}\leq \mathrm{M}_0 \cdot \left(1+|D u(0)|r^{-\alpha}\right),\,\,\forall r\in(0,\rho).
$$
\end{lemma}

\begin{proof}
Firstly, fix any $r\in(0,\rho)$ and choose $k\in\mathbb{N}$ the smallest integer such that $\rho^{k+1}<r\leq\rho^{k}$. By using Lemma \ref{lem.dy}, we estimate
\begin{align*}
\sup_{B_r(0)}\frac{|u(x)-u(0)|}{r^{1+\alpha}} & \leq \frac{1}{\rho^{1+\alpha}} \sup_{B_{\rho^k}(0)}\frac{|u(x)-u(0)|}{\rho^{k(1+\alpha)}} \\
										   & \displaystyle \leq \frac{1}{\rho^{1+\alpha}} \left(1+|Du(0)|\rho^{-k(1+\alpha)}\sum_{j=0}^{k-1}\rho^{k+j\alpha}\right)\\
										   & \leq \frac{1}{\rho^{1+\alpha}} \left(1+|Du(0)|\rho^{-k\alpha}\sum_{j=0}^{k-1}\rho^{j\alpha}\right)\\
										   & \leq \frac{1}{\rho^{1+\alpha}} \left(1+|Du(0)|\rho^{-k\alpha}\frac{1}{1-\rho^\alpha}\right)\\
                                           & \leq  \frac{1}{\rho^{1+\alpha}(1-\rho^\alpha)}\cdot (1+|Du(0)|r^{-\alpha})\\
   										   & = \mathrm{M}_0\cdot (1+|Du(0)|r^{-\alpha}),\\
\end{align*}
thereby concluding the proof.
\end{proof}

Now, we can give the proof of first main result of this manuscript:

\begin{proof}[{\bf Proof of Theorem \ref{main1}}]
Without loss of generality, we may assume that $x_0=0$. Notice that the degenerate ellipticity of the operator naturally leads us to separate the study into two different regimes, depending on whether $|Du(0)|$ is ``sufficiently small'' or not.

\vspace{0.3cm}

\begin{enumerate}
  \item If $0 \in \mathcal{S}_{r, \alpha}(u, \Omega)$
\vspace{0.1cm}

By using Lemma \ref{lem.dy} we estimate
\begin{align*}
\sup_{B_r(0)}\left|u(x)-\mathfrak{l}_{0} u(x)\right| & \leq \sup_{B_r(0)}|u(x)-u(0)|+|D u(0)|r \\
												 &  \leq \mathrm{M}_0 \cdot r^{1+\alpha}\left(1+|D u(0)|r^{-\alpha}\right)+r^{1+ \alpha}\\
												 & \leq  3\mathrm{M}_0 \cdot r^{1+\alpha}\\
\end{align*}
as desired in this case.

\vspace{0.3cm}
  \item If $0 \notin \mathcal{S}_{r, \alpha}(u, \Omega)$ i.e. $r^\alpha< |D u(0)|\leq L$
\vspace{0.1cm}

In this case, let us define $r_0 \defeq |D u(0)|^{\frac{1}{\alpha}}$ and
$$
u_{r_0}(x) \defeq \frac{u(r_0x)-u(0)}{r_0^{1+\alpha}}.
$$
Hence, we are allowed to apply Lemma \ref{l3.3} and conclude that
\begin{equation}\label{EstSmallGrad}
\displaystyle \sup_{B_{r_0}(0)} |u(r_0x)-u(0)| \leq 2\mathrm{M}_0\cdot r_0^{1+\alpha}.
\end{equation}
Now, notice that $u_{r_0}$ fulfills in the viscosity sense
$$
\mathcal{H}_{r_0}(x, D u_{r_0})F_{r_0}(x, D^2 u_{r_0}) = f_{r_0}(x) \quad \text{in} \quad B_1(0),
$$
where
$$
\left\{
\begin{array}{rcl}
F_{r_0}(x, \mathrm{X}) & \defeq & r_0^{1-\alpha}F\left(r_0x, r_0^{-(1-\alpha)}\mathrm{X}\right) \\
f_{r_0}(x)  & \defeq & r_0^{1-\alpha(p(r_0x)+1)}f(r_0x)\\
\mathcal{H}_{r_0}(x, \xi ) & \defeq & r_0^{-p(r_0x)\alpha}\mathcal{H}\left(r_0x, r_0^{\alpha}\xi\right)\\
\mathfrak{a}_{r_0}(x) & \defeq & r_0^{(q(r_0x)-p(r_0x))\alpha}\mathfrak{a}(r_0x)
\end{array}
\right.
$$
and
\begin{equation}\label{Eq3.2}
\displaystyle u_{r_0}(0) = 0, \,\,\,|D u_{r_0}(0)| = 1 \quad \text{and} \quad \|f_{r_0}\|_{L^{\infty}(B_1(0))} \le 1.
\end{equation}
Moreover, \eqref{EstSmallGrad} assures us that $u_{r_0}$ is uniformly bounded in the $L^{\infty}-$topology. From Theorem \ref{GradThm} it follows (using \eqref{Eq3.2}) that
$$
\|u_{r_0}\|_{C^{1,\gamma}(B_{1/2}(0))}\leq \mathrm{C}(\verb"universal").
$$
Such an estimate and one more time \eqref{Eq3.2}, allow us to choose a radius $0<\rho_0(\verb"universal")\ll 1$ such that
$$
\mathrm{c}_0 \leq |D u_{r_0}(x)|\leq \mathrm{c}_0^{-1} \,\,\,\forall \,\,x \in B_{\rho_0}(0) \,\,\,\text{and}\,\,\, \mathrm{c}_0 \in (0, 1) \,\,\,\text{fixed}.
$$

Particularly, we obtain (in the viscosity sense)
\[
 F_{r_0}(x, D^2u_{r_0}) =\tilde{f}_{r_0}(x)\defeq \frac{f_{r_0}(x)}{\mathcal{H}_{r_0}(x, Du_{r_0})} \quad \text{in} \quad B_{\rho_0}(0),
\]

The previous statement says $\tilde{f}_{r_0}$ is (universally) bounded in $B_{\rho_0}(0)$ and we get the result from classical estimates (see, \cite{C89}, \cite[Section 8.2]{CC95} and \cite{Tru88} - see also \cite{daSN21} for sharp estimates) since the equation becomes uniformly elliptic:
$$
  \mathscr{M}_{\lambda, \Lambda}^-(D^2 u_{r_0})\leq \mathrm{C}_0\left(p_{\text{min}}, q_{\text{max}}, \mathfrak{c}_0, \mathrm{L}^{-1}_1, \|f\|_{L^{\infty}(\Omega)}, \|\mathfrak{a}\|_{L^{\infty}(\Omega)}\right)
$$
and
$$
\mathscr{M}_{\lambda, \Lambda}^+(D^2 u_{r_0})\geq -\mathrm{C}_0\left(p_{\text{min}}, q_{\text{max}}, \mathfrak{c}_0, L^{-1}_1, \|f\|_{L^{\infty}(\Omega)}, \|\mathfrak{a}\|_{L^{\infty}(\Omega)}\right).
$$

 Therefore, $u_{r_0}\in C_{\text{loc}}^{1, \alpha_0}(B_{\rho_0}(0))$ for any $\alpha_0 \in (0, \alpha_{\mathrm{F}})$. As a result, we obtain
 \begin{equation}\label{EqEstUnifElliOper}
  \displaystyle \sup_{B_r(0)}\left|u_{r_0}(x)-\mathfrak{l}_{0} u_{r_0}(x)\right|\leq \mathrm{C} \cdot r^{1+\alpha_0}, \,\,\,\forall\,\,r \in \left(0, \frac{\rho_0}{2}\right),
 \end{equation}
 which, one re-writes in terms of $u$ as follows
 $$
 \displaystyle \sup_{B_r(0)}\left|\frac{u(r_0x)-u(0)}{r_0^{1+\alpha_0}}-r_0^{-\alpha_0}Du(0)\cdot x\right|\leq \mathrm{C} \cdot r^{1+\alpha_0}.
 $$
 Next, by doing $\alpha=\alpha_0$ in the above estimate (see, \eqref{SharpExp}), we conclude
  $$
 \displaystyle \sup_{B_r(0)}\left|u(x)-\mathfrak{l}_{0} u(x)\right|\leq \mathrm{C} \cdot r^{1+\alpha}, \,\,\,\forall\,\,r \in \left(0, \frac{\rho_0r_0}{2}\right),
 $$

 In conclusion, for $ r \in \left[\frac{\rho_0r_0}{2}, r_0\right)$, we have
 $$
 \begin{array}{rcl}
 \displaystyle \sup_{B_r(0)}\left|u(x)-\mathfrak{l}_{0} u(x)\right| & \leq &\displaystyle \sup_{B_{r_0}(0)}\left|u(x)-\mathfrak{l}_{0} u(x)\right|\\
 &\leq &\displaystyle \sup_{B_{r_0}(0)}\left|u(x)-u(0)\right| +|Du(0)|r_0\\
 &\leq & (2\mathrm{M}_0+1)\cdot r_0^{1+\alpha}\\
 &\leq & 3\mathrm{M}_0 \left(\frac{2}{\rho_0}\right)^{1+\alpha}\cdot r^{1+\alpha}.
 \end{array}
   $$

Finally, from characterization of Dini-Campanato spaces (see \textit{e.g.} \cite{Kov99}) we conclude that $u$ is $C^{1, \alpha}$ at $x_0 = 0$. Moreover, a standard covering argument yields the corresponding estimate in any compact subset $\Omega^{\prime} \Subset \Omega$, which finishes the proof.

\end{enumerate}
\end{proof}

\begin{proof}[{\bf Proof of Corollary \ref{Cormain1}}]
It follows from Theorem \ref{main1}, since solutions to the homogeneous problem (with ``frozen coefficients'') for such classes of operators are $C_{\text{loc}}^{1,1}(\Omega)$, in other words, $\alpha_{\mathrm{F}}=1$ (see, \cite[Section 6]{daSR20} for more examples). Therefore, we are able to choose $\alpha = \frac{1}{p_{\text{max}}+1} \in (0, 1)$ in the sentences \eqref{2.5} and \eqref{EqEstUnifElliOper}.
\end{proof}


Next, we will deliver the proof of the geometric non-degeneracy.

\begin{proof}[{\bf Proof of Theorem \ref{main3}}]
Firstly, for $x_0 \in \Omega^{\prime} \Subset \Omega$ let us define the scaled function:
 $$
    u_{r, x_0}(x) \defeq \frac{u(x_0+rx)-u(x_0)+ \varepsilon}{r^{\frac{p_{\textrm{min}}+2}{p_{\textrm{min}}+1}}} \quad \text{for} \quad x \in B_1(0).
 $$
Now, observe that $u_{r, x_0}$ fulfills in the viscosity sense
$$
  \mathcal{H}_{r, x_0}(x, D u_{r, x_0})F_{r, x_0}(x, D^2 u_{r, x_0}) = f_{r, x_0}(x) \quad \text{in} \quad B_1(0),
$$
where
$$
\left\{
\begin{array}{rcl}
  F_{r, x_0}(x, \mathrm{X}) & \defeq & r^{\frac{p_{\textrm{min}}}{p_{\textrm{min}}+1}}F\left(x_0+rx, r^{-\frac{p_{\textrm{min}}}{p_{\textrm{min}}+1}}\mathrm{X}\right) \\
  f_{r, x_0}(x)  & \defeq & f(x_0 + r x)\\
  \mathcal{H}_{r, x_0}(x, \xi ) & \defeq & r^{-\frac{p_{\textrm{min}}}{p_{\textrm{min}}+1}}\mathcal{H}\left(x_0+rx, r^{\frac{1}{p_{\textrm{min}}+1}}\xi\right)\\
  \mathfrak{a}_{r, x_0}(x) & \defeq & r^{\frac{q_{\textrm{max}}-p_{\textrm{min}}}{p_{\textrm{max}}+1}}\mathfrak{a}(x_0+rx).
\end{array}
\right.
$$

Firstly, let us introduce the comparison function:
$$
   \Theta(x) \defeq \mathfrak{c}\cdot |x|^{\frac{p_{\textrm{min}}+2}{p_{\textrm{min}}+1}},
$$
where the constant $\mathfrak{c}>0$ will be chosen in such a way that
$$
\mathcal{H}_{x_0,r}(x, D\Theta)F_{x_0,r}(x, D^2 \Theta) < f_{x_0,r}(x)\quad \text{in} \quad B_R(0) \Subset \Omega.
$$
Note that
\begin{eqnarray*}
	D_i \Theta(x) &=&  \mathfrak{c} \left( \frac{p_{\textrm{min}}+2}{p_{\textrm{min}}+1} \right) x_i \cdot |x|^{-\frac{p_{\textrm{min}}}{p_{\textrm{min}}+1}}\\
	D_{ij} \Theta(x)& =&  \mathfrak{c} \left( \frac{p_{\textrm{min}}+2}{p_{\textrm{min}}+1} \right) \left[ \delta_{ij} - \frac{p_{\textrm{min}}}{p_{\textrm{min}}+1} \frac{x_i x_j}{|x|^2}\right] \cdot |x|^{-\frac{p_{\textrm{min}}}{p_{\textrm{min}}+1}}.
\end{eqnarray*}
Let now $\mathfrak{A}=(\mathfrak{A}_{ij})$ be a symmetric matrix whose eigenvalues belong to $[\lambda, \Lambda]$, i.e.,
$$
	\lambda |\xi |^2 \le \sum_{i,j=1}^{N} \mathfrak{A}_{ij} \xi_i \xi_j \le \Lambda |\xi|^2, \quad \forall \, \xi \in \mathbb{R}^n.
$$
We will write in this case that $\mathfrak{A} \in \mathcal{A}_{\lambda, \Lambda}$. Hence, we have
{\scriptsize{
\begin{eqnarray*}
\sum_{i,j=1}^{n} \mathfrak{A}_{ij} D_{ij} \Theta(x) &=& \sum_{i,j=1}^{n} \mathfrak{A}_{ij} \left \{ \mathfrak{c} \left( \frac{p_{\textrm{min}}+2}{p_{\textrm{min}}+1}\right) \left[\delta_{ij} - \frac{p_{\textrm{min}}}{p_{\textrm{min}}+1} \frac{x_ix_j}{|x|^2} \right] \cdot |x|^{-\frac{p_{\textrm{min}}}{p_{\textrm{min}}+1}}\right\}\\
&=&   \mathfrak{c} \left( \frac{p_{\textrm{min}}+2}{p_{\textrm{min}}+1} \right) \left[ \sum_{i,j=1}^{n} \mathfrak{A}_{ij} \delta_{ij} - \frac{p_{\textrm{min}}}{p_{\textrm{min}}+1} \sum_{i,j=1}^{n} \mathfrak{A}_{ij} \frac{x_i x_j}{|x|^2} \right] \cdot |x|^{-\frac{p_{\textrm{min}}}{p_{\textrm{min}}+1}}\\
&=&  \mathfrak{c} \left( \frac{p_{\textrm{min}}+2}{p_{\textrm{min}}+1} \right) \left[ \sum_{i=1}^{n} \mathfrak{A}_{ii}  - \frac{p_{\textrm{min}}}{p_{\textrm{min}}+1} \frac{1}{|x|^2}\sum_{i,j=1}^{n} \mathfrak{A}_{ij} x_i x_j \right] \cdot |x|^{-\frac{p_{\textrm{min}}}{p_{\textrm{min}}+1}}\\
&\le&  \mathfrak{c} \left( \frac{p_{\textrm{min}}+2}{p_{\textrm{min}}+1} \right) \left[ n \Lambda - \frac{p_{\textrm{min}}}{p_{\textrm{min}}+1} \lambda\right] |x|^{-\frac{p_{\textrm{min}}}{p_{\textrm{min}}+1}}.
\end{eqnarray*}
}}
Thus,
\begin{eqnarray*}
	\mathcal{M}^{+}_{\lambda, \Lambda}(D^2 \Theta(x)) &=& \sup_{\mathfrak{A} \in \mathcal{A}_{\lambda, \Lambda}} \left\{ \sum_{i,j=1}^{n} \mathfrak{A}_{ij} D_{ij} \Theta(x)\right\} \\
	&\le&  \mathfrak{c} \left( \frac{p_{\textrm{min}}+2}{p_{\textrm{min}}+1} \right) \left[ n \Lambda - \frac{p_{\textrm{min}}}{p_{\textrm{min}}+1} \lambda\right]  |x|^{-\frac{p_{\textrm{min}}}{p_{\textrm{min}}+1}}.
\end{eqnarray*}

Now, let us note that by  assumptions \eqref{1.2} and \eqref{N-HDeg} we have
{\scriptsize{
$$
\begin{array}{rcl}
  \mathcal{H}_{r,x_0}(x, D\Theta) &=& r^{-\frac{p_{\textrm{min}}}{p_{\textrm{min}}+1}}\mathcal{H}\left(x_0+rx, r^{\frac{1}{p_{\textrm{min}}+1}} D \Theta(x_0+rx)\right)\\
   &\le&  \mathrm{L}_2 \cdot \left( |r^{\frac{1}{p_{\textrm{min}}+1}}D\Theta|^{p(x_0+rx)} + \mathfrak{a}_{r,x_0}(x) |r^{\frac{1}{p_{\textrm{min}}+1}}D\Theta|^{q(x_0+rx)} \right) \cdot r^{-\frac{p_{\textrm{min}}}{p_{\textrm{min}}+1}} \\
  &=& \mathrm{L}_2 . \Big( r^{\frac{p(x_0+rx)}{1+p_{\textrm{min}} }} \mathfrak{c}^{p(x_0+rx)} \left( \frac{p_{\textrm{max}} +2}{p_{\textrm{min}}+1}\right)^{p(x_0+rx)} \cdot |x|^{\frac{p(x_0+rx)}{1+ p_{\textrm{min}}}} +\\
  &&\mathfrak{a}_{r,x_0}(x)  r^{\frac{q(x_0+rx)}{1+p_{\textrm{min}} }} \mathfrak{c}^{q(x_0+rx)} \left( \frac{p_{\textrm{min}} +2}{p_{\textrm{min}}+1}\right)^{q(x_0+rx)} \cdot |x|^{\frac{q(x_0+rx)}{1+ p_{\textrm{min}}}} \Big)  \cdot r^{-\frac{p_{\textrm{min}}}{p_{\textrm{min}}+1}}\\
  &\le&  \mathrm{L}_2 . \Big( r^{\frac{p(x_0+rx)}{1+p_{\textrm{min}} }} \mathfrak{c}^{p(x_0+rx)} \left( \frac{p_{\textrm{min}} +2}{p_{\textrm{min}}+1}\right)^{p_{\textrm{max}}} \cdot |x|^{\frac{p(x_0+rx)}{1+ p_{\textrm{min}}}} +\\
  &&\mathfrak{a}_{r,x_0}(x)  r^{\frac{q(x_0+rx)}{1+p_{\textrm{min}} }} \mathfrak{c}^{q(x_0+rx)} \left( \frac{p_{\textrm{min}} +2}{p_{\textrm{min}}+1}\right)^{q_{\textrm{max}}} \cdot |x|^{\frac{q(x_0+rx)}{1+ p_{\textrm{min}}}} \Big)  \cdot r^{-\frac{p_{\textrm{min}}}{p_{\textrm{min}}+1}}\\
&\le&\mathrm{L}_2 \cdot \Big( \mathfrak{c}^{p(x_0+rx)} \left( \frac{p_{\textrm{min}}+2}{p_{\textrm{min}}+1}\right)^{p_{\textrm{max}}} + r^{\frac{q(x_0+rx) -p(x_0+rx)}{1+p_{\textrm{min}}}} \mathfrak{a}_{r,x_0}(x).\\
&\cdot& \mathfrak{c}^{q(x_0+rx)}  \left( \frac{p_{\textrm{min}}+2}{p_{\textrm{min}}+1}\right)^{q_{\textrm{max}}} |x|^{\frac{q(x_0+rx) - p(x_0+rx)}{1+p_{\textrm{max}}}} \Big) r^{\frac{p(x_0+rx) - p_{\textrm{min}}}{1+p_{\textrm{min}}}} . |x|^{\frac{p(x_0+rx)}{1+p_{\textrm{min}}}}\\
&\le&\mathrm{L}_2 \cdot \Big( \mathfrak{c}^{p(x_0+rx)} \left( \frac{p_{\textrm{min}}+2}{p_{\textrm{min}}+1}\right)^{p_{\textrm{max}}} + r^{\frac{q(x_0+rx) - p_{\textrm{min}} + q_{\textrm{max}}-p(x_0+rx) }{1+p_{\textrm{max}}}} \mathfrak{a}(x_0+rx).\\
&\cdot& \mathfrak{c}^{q(x_0+rx)}  \left( \frac{p_{\textrm{min}}+2}{p_{\textrm{min}}+1}\right)^{q_{\textrm{max}}} |x|^{\frac{q_{\textrm{max}} - p_{\textrm{min}}}{1+p_{\textrm{max}}}} \Big)|x|^{\frac{p_{\textrm{min}}}{1+p_{\textrm{min}}}} \\
&\le&\mathrm{L}_2 \cdot \Big( \mathfrak{c}^{p(x_0+rx)} \left( \frac{p_{\textrm{min}}+2}{p_{\textrm{min}}+1}\right)^{p_{\textrm{max}}} + \| \mathfrak{a}\|_{L^{\infty}}. \mathfrak{c}^{q(x_0+rx)}  \left( \frac{p_{\textrm{min}}+2}{p_{\textrm{min}}+1}\right)^{q_{\textrm{max}}} \Big)\cdot |x|^{\frac{p_{\textrm{min}}}{1+p_{\textrm{min}}}}  .
\end{array}
$$
}}
In the sequel, we will split the analysis in two cases:
\begin{itemize}
\item If $0 < \mathfrak{c} < 1$,
$$
	 \mathcal{H}_{r,x_0}(x, D\Theta) \le \mathrm{L}_2 \left( \mathfrak{c}^{p_{\textrm{min}}}  \left( \frac{p_{\textrm{min}}+2}{p_{\textrm{min}}+1}\right)^{p_{\textrm{max}}} + \|\mathfrak{a}\|_{L^{\infty}} \mathfrak{c}^{p_{\textrm{max}}} \left( \frac{p_{\textrm{min}}+2}{p_{\textrm{min}}+1}\right)^{q_{\textrm{max}}} \right)\cdot |x|^{\frac{p_{\textrm{min}}}{1+p_{\textrm{min}}}} .
$$
Thus,
$$
	 \mathcal{H}_{r,x_0}(x, D\Theta) F_{r,x_0}(x,D^2 \Theta) \le \Xi_1 \cdot \Xi_2
$$
where
$$
\Xi_1 \defeq \left[\left(\frac{p_{\textrm{min}}+2}{p_{\textrm{min}}+1}\right)^{p_{\textrm{max}}+1} \mathfrak{c}^{p_{\textrm{min}}+1} + \|\mathfrak{a}\|_{L^{\infty}(\Omega)} \left(\frac{p_{\textrm{min}}+2}{p_{\textrm{min}}+1}\right)^{q_{\textrm{max}}+1}\mathfrak{c}^{p_{\textrm{max}}+1}\right],
$$
and
$$
\Xi_2 \defeq \mathrm{L}_2\cdot \left(n \Lambda - \frac{p_{\textrm{min}}}{p_{\textrm{min}}+1}  \lambda\right).
$$
At this point, consider the  function $\mathfrak{g}: [0, \infty) \to \R$ given by
$$
   \mathfrak{g}(t) \defeq \Xi_2\cdot t^{p_{\textrm{min}}+1}\left[\left(\frac{p_{\textrm{min}}+2}{p_{\textrm{min}}+1}\right)^{p_{\textrm{max}}+1}+\Xi_3\cdot t^{p_{\textrm{max}}-p_{\textrm{min}}}\right]-\mathfrak{m},
$$
where
$$
    \displaystyle \Xi_3 \defeq \|\mathfrak{a}\|_{L^{\infty}(\Omega)} \left(\frac{p_{\textrm{min}}+2}{p_{\textrm{min}}+1}\right)^{q_{\textrm{max}}+1} \quad \text{and} \quad \mathfrak{m} \defeq \inf_{\Omega} f(x).
$$
Now, let us label $\mathrm{T}_0>0$ its smallest root, which there exists thanks to the assumption $\displaystyle \mathfrak{m} \defeq \inf_{\Omega} f(x)>0$. Therefore, we are able to choose a $\mathfrak{c}=\mathfrak{c}(\mathfrak{m},\|\mathfrak{a}\|_{L^{\infty}(\Omega)}, \mathrm{L}_2, n, \lambda, \Lambda, p_{\textrm{min}},q_{\textrm{max}}, p_{\textrm{max}} , \Omega) \in (0, \mathrm{T}_0)$ such that
$$
 \mathcal{H}_{r,x_0}(x, D\Theta)F_{r,x_0}(x, D^2 \Theta) < f_{r,x_0}(x) \qquad \text{point-wisely}.
$$

\item If $ \mathfrak{c} > 1$,
$$
	 \mathcal{H}_{r,x_0}(x, D\Theta) \le \mathrm{L}_2 \left( \mathfrak{c}^{p_{\textrm{max}}}  \left( \frac{p_{\textrm{min}}+2}{p_{\textrm{min}}+1}\right)^{p_{\textrm{max}}} + \|\mathfrak{a}\|_{L^{\infty}} \mathfrak{c}^{q_{\textrm{max}}} \left( \frac{p_{\textrm{min}}+2}{p_{\textrm{min}}+1}\right)^{q_{\textrm{max}}} \right)\cdot |x|^{\frac{p_{\textrm{min}}}{1+p_{\textrm{min}}}}
$$
Thus,
 $$
	 \mathcal{H}_{r,x_0}(x, D\Theta) F_{r,x_0}(x,D^2 \Theta) \le \Xi_4 \cdot \Xi_2
$$
where
$$
\Xi_4 \defeq \left[\left(\frac{p_{\textrm{min}}+2}{p_{\textrm{min}}+1}\right)^{p_{\textrm{max}}+1} \mathfrak{c}^{p_{\textrm{max}}+1} + \|\mathfrak{a}\|_{L^{\infty}(\Omega)} \left(\frac{p_{\textrm{min}}+2}{p_{\textrm{min}}+1}\right)^{q_{\textrm{max}}+1}\mathfrak{c}^{q_{\textrm{max}}+1}\right].
$$

Now, we consider the function $\hat{\mathfrak{g}}: [0, \infty) \to \R$ given by
$$
   \hat{\mathfrak{g}}(t) \defeq \Xi_2\cdot t^{p_{\textrm{max}}+1}\left[\left(\frac{p_{\textrm{min}}+2}{p_{\textrm{min}}+1}\right)^{p_{\textrm{max}}+1}+\Xi_3\cdot t^{q_{\textrm{max}}-p_{\textrm{max}}}\right]-\mathfrak{m},
$$
where
$$
    \displaystyle \Xi_3 \defeq \|\mathfrak{a}\|_{L^{\infty}(\Omega)} \left(\frac{p_{\textrm{min}}+2}{p_{\textrm{min}}+1}\right)^{q_{\textrm{max}}+1} \quad \text{and} \quad \mathfrak{m} \defeq \inf_{\Omega} f(x).
$$
As in the previous case, we are able to choose $\hat{\mathrm{T}}_0>0$ and $\mathfrak{c}=\mathfrak{c}(\mathfrak{m},\|\mathfrak{a}\|_{L^{\infty}(\Omega)}, \mathrm{L}_2, n, \lambda, \Lambda, p_{\textrm{min}},q_{\textrm{max}}, p_{\textrm{max}} , \Omega) \in (0, \hat{\mathrm{T}}_0)$ such that
$$
 \mathcal{H}_{r,x_0}(x, D\Theta)F_{r,x_0}(x, D^2 \Theta) < f_{r,x_0}(x) \qquad \text{point-wisely}.
$$
\end{itemize}

Finally, if $u_{r, x_0} \leq \Theta$ on the whole boundary of $B_1(0)$, then the Comparison Principle (Lemma \ref{comparison principle}), would imply that
$$
   u_{r, x_0}(x) \leq \Theta(x) \quad \mbox{in} \quad B_1(0),
$$
which contradicts the assumption that $u_{r, x_0}(0)> 0$. Therefore, there exists a point $z \in \partial B_1(0)$ such that
$$
      u_{r, x_0}(z) > \Theta(z) = \mathfrak{c}(\mathfrak{m},\|\mathfrak{a}\|_{L^{\infty}(\Omega)}, \mathrm{L}_1, n, \lambda, \Lambda, p_{\textrm{min}}, p_{\textrm{max}}, q_{\textrm{max}}, \Omega).
$$
By scaling back and letting $\varepsilon \to 0$ we finish the proof of the Theorem.
\end{proof}

\section{Proof of Theorem \ref{main1}}\label{Sec4}

In this section, we will deliver the proof of our sharp gradient estimates, i.e., Theorem \ref{main1}. Before proceeding to such a proof, we need to collect a number of auxiliary Lemmas, which play a decisive role in our strategy.

The first key step towards the proof of Theorem \ref{main1} is to show that for any $\zeta \in \mathbb{R}^n$, solutions to
\begin{equation} \label{T3.1}
\left\{
\begin{array}{rclcl}
\mathcal{H}(y, Du + \zeta)F( D^2 u) & = & f(y) & \text{in} & B_1(x) \cap \{y_n >\phi(y')\}  \\
  u& = & g & \text{on} & B_1(x) \cap \{y_n = \phi(y)\},
\end{array}
\right.
\end{equation}
can be approximated in a suitable manner by $F-$harmonic profiles, and that during such a process certain regularity properties of solutions are preserved.


\begin{lemma}[{\bf Approximation Lemma - Boundary case}]\label{L0}
 For all $x \in \mathrm{B}$ such that $B_1(x) \subset \mathrm{B}$,  let $\zeta \in \mathbb{R}^n$ be an arbitrary vector and $u \in C(B_1(x) \cap \{y_n > \phi(y^{\prime})\})$ a viscosity solution to
 $$
\left\{
\begin{array}{rclcl}
 \mathcal{H}(y,Du+\zeta) F(y, D^2 u) & = & f(y) & \text{in} & B_1(x) \cap \left\{y_n >\phi(y^{\prime})\right\} \\
  u & = & g & \text{on} & B_1(x) \cap \{y_n = \phi(y^{\prime})\},
\end{array}
\right.
$$
satisfying $\|u\|_{L^{\infty}(B_1(x) \cap \{y_n > \phi(y^{\prime})\})}, \|g\|_{_{C^{1, \beta_g}(B_1(x) \cap \{y_n =\phi(y^{\prime})\})}} \le 1$. Then, for all $\delta >0$ there exists an $\epsilon = \epsilon (n, \lambda, \Lambda, \delta, p_{\text{max}}) \in (0,1)$ such that if
$$
	\max\left\{\Theta_F(x), \,\,\,\|f\|_{L^{\infty}(B_1(x) \cap \{y_n > \phi(y^{\prime})\})} \right\}\le \epsilon
$$
then we can find a function $\mathfrak{h}$, solution to a constant coefficient, homogeneous $(\lambda,\Lambda)$-uniform elliptic equation
\begin{equation} \label{BD}
\left\{
\begin{array}{rcrcl}
\mathfrak{F}( D^2 \mathfrak{h}) & = & 0 & \text{in} & B_1(x) \cap \{y_n > \phi(y^{\prime})\}  \\
  \mathfrak{h} & = & g & \text{on} & B_1(x) \cap \{y_n = \phi(y^{\prime})\},
\end{array}
\right.
\end{equation}
such that
$$
	\|u-\mathfrak{h}\|_{L^{\infty}(B_1(x) \cap \{y_n > \phi(y^{\prime})\})} \le \delta.
$$
\end{lemma}

\begin{proof}
The proof is based on a contradiction argument. Then, let us assume that the claim is not satisfied. Thus, there are $\delta_0 >0$ and sequences  $(\mathcal{H}_j)_j, (f_j)_j, (u_j)_j , (F_j)_j, (g_j)_j$ and sequence of vectors $(\zeta_j)_j$ for which there exist viscosity solutions $u_j$ of
$$
\left\{
\begin{array}{rclcl}
\mathcal{H}_j(y, Du_j + \zeta_j)F_j(y, D^2 u_j) & = & f_j(y) & \text{in} & B_1(x) \cap \{y_n >\phi(y^{\prime})\}  \\
  u_j & = & g_j & \text{on} & B_1(x) \cap \{y_n = \phi(y^{\prime})\},
\end{array}
\right.
$$
satisfying:
\begin{enumerate}
\item[(i)] $(F_j(y, \cdot))$ are uniformly $(\lambda,\Lambda)$-elliptic operators;

\item[(ii)] $\max\left\{ \Theta_{F_j}(x), \,\,\|f_j\|_{L^{\infty}(B_1(x) \cap \{y_n > \phi(y^{\prime})\})}\right\} \le \frac{1}{j}  $ and
$$
    f_j \in C^0(B_1(x) \cap \{y_n > \phi(y^{\prime})\})
    $$

\item[(iii)] $\mathcal{H}_j(y,\cdot)$ satisfying \eqref{1.2} and \eqref{N-HDeg} with
$$
	0 < p_{\text{min}} \le p_j(y) \le p_{\text{max}} \le q_j(y) \le q_{\text{max}} < \infty,
$$
and
{\small$$
p_j(\cdot), q_j(\cdot) \in C^0(B_1(x) \cap \{y_n > \phi(y^{\prime})\}) \quad \text{and} \quad 0 \le \mathfrak{a}_j(\cdot) \in C^0(B_1(x) \cap \{y_n > \phi(y^{\prime})\});
$$}

\item[(iv)] $u_j \in C^0(B_1(x) \cap \{y_n > \phi (y^{\prime})\})$ and $\|u_j\|_{L^{\infty}(B_1(x) \cap \{y_n > \phi(y^{\prime})\})} \le 1$
    and $g_j \in C^{1, \beta_g}(B_1(x) \cap \{y_n = \phi (y^{\prime})\})$ and $\|g_j\|_{_{C^{1, \beta_g}(B_1(x) \cap \{y_n =\phi(y^{\prime})\})}} \le 1$
\end{enumerate}
however,
\begin{equation} \label{Eq0}
	\|u_j-\mathfrak{h}\|_{L^{\infty}(B_1(x) \cap \{y_n > \phi(y^{\prime})\})} > \delta_0 \qquad \,\,\ \text{for every} \,\,\, j \in \mathbb{N}
\end{equation}
for any $\mathfrak{h}$ satisfying \eqref{BD}.

Now, since the operators $(F_j(y, \cdot))$ are $(\lambda,\Lambda)$-elliptic, from Lemma \ref{lem.stab} and condition (ii) they converge uniformly to some operator $\mathfrak{F}(\cdot)$ (with frozen coefficients).

In the sequel, we will show that $u_{\infty}$ is a viscosity solution of
\begin{equation}\label{Eq1L*}
\left\{
\begin{array}{rclcl}
\mathfrak{F} ( D^2 w) & = & 0 & \text{in} & B_1(x) \cap \{y_n > \phi(y^{\prime})\}  \\
  w & = & g_{\infty} & \text{on} & B_1(x) \cap \{y_n = \phi(y^{\prime})\},
\end{array}
\right.
\end{equation}
where $\quad g_j \to g_{\infty}$ (uniformly).

Indeed, let us suppose first that $(\zeta_j)_j$ is a bounded sequence. Then, up to subsequence, it converges to $\zeta_{\infty}$. In the sequel, by condition (iv) and the H\"{o}lder continuity of $u_j$ (namely, Theorem \ref{Xizao} and Remark \ref{RemarkSeqTransl}), we get
$$
\displaystyle	\|u_j\|_{C^{0,\gamma}( \overline{B_1(x) \cap \{y_n > \phi(y^{\prime})\}})}  \le \mathrm{C} \left(n,\lambda, \Lambda, p_{\text{max}}, \sup_{j} \zeta_j \right) \quad \text{for some} \quad \gamma \in (0,1).
$$

Therefore, by applying Arzel\`{a}-Ascoli's compactness criterium we obtain the existence of functions
$$
   u_{\infty} \in C^0(\overline{B_1(x) \cap \{y_n > \phi(y^{\prime})\}}) \quad \text{and} \quad g_{\infty} \in C^{1, \beta_g}(B_1(x) \cap \{y_n = \phi(y^{\prime})\})
$$
and subsequence such that
$$
   u_j \to u_{\infty} \quad  \text{uniformly on} \quad  B_1(x) \cap \{y_n > \phi(y^{\prime})\} \quad
$$
with $u_{\infty} = g_{\infty}$ on $B_1(x) \cap \{y_n = \phi(y^{\prime})\}$.

Hence, by Lemma \ref{P1} we conclude that $u_{\infty}$ is a viscosity solution to \eqref{Eq1L*}.

On the other hand, if $(\zeta_j)_j$ is a unbounded sequence. Then, we may assume $|\zeta_j| \to \infty$ (up to a subsequence). In this case, by invoking \eqref{1.2} and \eqref{N-HDeg}, $u_j$ is a viscosity solution of
$$
\mathcal{H}_j\left(y, \frac{Du_j}{|\zeta_j|} + \frac{\zeta_j}{|\zeta_j|}\right)F_j(x, D^2 u_j)  =  \frac{f_j(y)}{|\zeta_j|^{p_j(x)}} \quad  \text{in} \quad B_1(x) \cap \{y_n >\phi(y^{\prime})\}.
$$

Finally, since $(u_j)_j$ is equi-continuous (see Theorem \ref{pgrande}), by arguing as in \cite[Lemma 4.1 - Case 1]{DeF20} and \cite[Lemma 3.2]{FRZ21} we conclude (after pass a subsequence) that $u_{\infty}$ is a viscosity solution to \eqref{Eq1L*}. This completes the proof.

\end{proof}

\begin{remark}[{\bf Smallness regime - Boundary case}]\label{SmallRegime2}
As before, we may perform a scaling and normalization $v(x) = \frac{u(\tau x+ x_0)}{\kappa}$
in such a way
$$
	\|v\|_{L^{\infty}(\Omega)}, \|\hat{g}\|_{C^{1, \beta_g}(\partial \Omega)} \le 1 \quad \textrm{and} \quad \max\left\{\Theta_{\hat{F}}(x), \,\|\hat{f}\|_{L^{\infty}(\Omega)} \right\}\le \epsilon,
$$
where $v$ solves a PDE possessing the same structural conditions as \eqref{1.1} and then smallness and normalization regime in Lemma \ref{L0} are verified.

\end{remark}

In the sequel, we establish a first step in a geometric regularity approach via a recursive iterative process.

\begin{lemma}[{\bf First step}] \label{L2}
Let $\alpha$ be fixed as in \eqref{SharpExp}. For any $\phi \in C^2(\overline{\Omega})$ with \eqref{Condphi} in force, there exist $\epsilon_0 \in (0,1)$ and $\rho \in \left(0, \frac{1}{2}\right) $ depending on $\alpha, n, \lambda, \Lambda, \|D^2 \phi\|_{L^\infty(\Omega)}, \|g\|_{C^{1,\beta_g}(\partial\Omega)}$ and  $p_{\text{max}}$ such that for any $\zeta \in \mathbb{R}^n$ and $u$ a viscosity solution of
$$
\left\{
\begin{array}{rcrcl}
\mathcal{H}(y, \zeta + Du)F(y, D^2 u) & = & f & \text{in} & \mathrm{B} \cap \{y_n > \phi(y^{\prime})\}  \\
  u  & = & g & \text{on} & \mathrm{B} \cap \{y_n = \phi(y^{\prime})\},
\end{array}
\right.
$$
the following holds: for all $x \in \mathrm{B}$ such that $B_{1}(x) \subset \mathrm{B}$, if
{\small $$
	\|u\|_{L^{\infty}( B_{1}(x) \cap \{y_n >\phi(y^{\prime})\}) )} \le 1 \quad \textrm{and} \quad    \max\left\{\Theta_F(x), \,\,\|f\|_{L^{\infty}(B_1(x) \cap \{y_n > \phi(y^{\prime})\})}\right\} \le \epsilon_0,
$$}
then there exist an affine function $\ell(y) = \mathrm{a} + \mathrm{b} \cdot (y-x)$ ($\mathrm{a} \in \mathbb{R}$ and $\mathrm{b} \in \mathbb{R}^n$) such that
$$
	\|u - \ell\|_{L^{\infty}( B_{\rho}(x) \cap \{y_n >\phi(y^{\prime})\})} \le \rho^{1+\alpha}
$$
and
$$
	|\mathrm{a}| + |\mathrm{b}| \le \mathrm{C}(n,\lambda,\Lambda).
$$
\end{lemma}

\begin{proof}
 Firstly, note that if $B_1(x) \cap \{y_n=\phi(y^{\prime})\} = \emptyset$, then it is sufficient to use the interior estimates due to \cite[Theorem 1.1]{FRZ21}. For this reason, from now on, we are going to assume that $B_1(x) \cap \{y_n =\phi(y^{\prime})\} \not= \emptyset$.

 Now, for a $\delta >0$ to be chosen \textit{a posteriori}, let $\mathfrak{h}$ be the viscosity solution to
$$
	\left\{
\begin{array}{rclcl}
 F( D^2 \mathfrak{h}) & = & 0 & \text{in} & B_{1}(x) \cap \{y_n > \phi(y^{\prime})\}  \\
  \mathfrak{h} & = & g & \text{on} & B_{1}(x) \cap \{y_n = \phi(y^{\prime})\},
\end{array}
\right.
$$
which is $\delta$-close to $u$ in the $L^{\infty}(B_1(x) \cap \{y_n > \phi(y^{\prime})\})$ topology.

Remember that Lemma \ref{L0} ensures the existence of such an $\mathfrak{h}$, provided that $\epsilon_0 >0$ is small enough. Moreover, by virtue of the global $C^{1,\alpha_F}$-regularity of $\mathfrak{h}$ (cf. \cite[Theorem 1.1]{BD2} and \cite[Theorem 1.1]{BD16}), there exists a constant $\mathrm{C}_0>1$ depending only on $n,\lambda,\Lambda, \|g\|_{C^{1, \beta_g}(\partial \Omega)}$ and $\|\phi\|_{C^2}$ such that
{\scriptsize{
\begin{equation}\label{EqEstHomProb}
   	\sup_{y \in B_{r}(x) \cap \{y_n > \phi(y^{\prime})\}} \frac{\left| \mathfrak{h}(y) - (\mathfrak{h}(x) + D\mathfrak{h}(x) \cdot (y-x)) \right|}{r^{1+\alpha_F}} \le \mathrm{C}_0 \quad \text{for all}\quad r \in \left(0, \frac{1}{2}\right)
\end{equation}}}
and
 $$
 	|\mathfrak{h}(x)| + |D\mathfrak{h}(x)| \le \mathrm{C}_0.
 $$
 In this point, let us set
\begin{equation}\label{EqL-funct}
   	\ell(y) \defeq  \mathfrak{h}(x) + D \mathfrak{h}(x) \cdot (y-x).
\end{equation}

 Thus, \eqref{EqEstHomProb},  \eqref{EqL-funct} and Approximation \ref{L0} Lemma  yield that
 \begin{equation}\label{EqPFirstIter}
    \begin{array}{rcl}
     \displaystyle \sup_{B_{\rho}(x) \cap \{y_n > \phi(y^{\prime})\}} |u(y) - \ell(y)| & \le & \displaystyle \sup_{B_{\rho}(x) \cap \{y_n > \phi(y^{\prime})\}} |u(y) - \mathfrak{h}(y)| \\
       & + & \displaystyle \sup_{B_{\rho}(x) \cap \{y_n > \phi(y^{\prime})\}} |\mathfrak{h}(y) - \ell(y) | \\
       & < & \delta + \mathrm{C}_0 \rho^{1+\alpha_F}.
    \end{array}
 \end{equation}
 Since $0 < \alpha < \alpha_F$, we can choose $0 < \rho \ll 1$ so small that
 \begin{equation}\label{Eq5.2}
    	\mathrm{C}_0 \rho^{\alpha_F - \alpha} \le \frac{1}{2} \quad \Leftrightarrow \quad \rho \le \left(\frac{1}{2\mathrm{C}_0}\right)^{\frac{1}{\alpha_F-\alpha}}.
 \end{equation}
 Also, for the previous selected radius, we fix
\begin{equation}\label{Eq5.3}
 	\delta = \frac{1}{2} \rho^{1+\alpha} \le \frac{1}{2} \left(\frac{1}{2\mathrm{C}_0}\right)^{ \frac{1+\alpha}{\alpha_F-\alpha}}.
\end{equation}
 Finally, by combining \eqref{Eq5.2} and \eqref{Eq5.3} and plugging them in \eqref{EqPFirstIter} we obtain
$$
	\|u-\ell\|_{L^{\infty}(B_{\rho}(x) \cap \{y_n >\phi(y^{\prime})\}) }  \le \rho^{1+\alpha}.
$$
\end{proof}

Next, the H\"{o}lder gradient regularity can be derived from following iteration process.

\begin{lemma}[{\bf Iterative procedure - $k^{\text{th}}-$step}] \label{L3}
Suppose that $\rho, \epsilon_0>0$ and $\mathrm{B}$ are as in Lemma \ref{L2}. Suppose that $u$ is a viscosity solution of
\begin{equation} \label{3}
\left\{
\begin{array}{rcrcl}
  \mathcal{H}(y,Du)F( y, D^2 u) & = & f & \text{in} & \mathrm{B} \cap \{y_n > \phi(y^{\prime})\}  \\
  u & = & g & \text{on} & \mathrm{B} \cap \{y_n = \phi(y^{\prime})\},
\end{array}
\right.
\end{equation}
with
{\small$$
	\|u\|_{L^{\infty}( B_{1}(x) \cap \{y_n >\phi(y^{\prime})\}) )} \le 1 \quad \textrm{and} \quad    \max\left\{\Theta_F(x), \,\,\|f\|_{L^{\infty}(B_1(x) \cap \{y_n > \phi(y^{\prime})\})}\right\} \le \epsilon_0,
$$}
then, for any $\alpha$ as in \eqref{SharpExp}, $j \in \{0,1,2,\ldots\}$ and $x \in B$ such that $B_1(x) \subset B$,  there is a sequence $(\ell_j)_j$, where $\ell_j(y) \defeq \mathrm{a}_j + \mathrm{b}_j \cdot (y-x)$ ($\mathrm{a}_j \in \mathbb{R}$, $\mathrm{b}_j \in \mathbb{R}^n$ and $|\mathrm{a}_j|+|\mathrm{b}_j|\leq \mathrm{C}_0$) satisfying
\begin{equation} \label{4}
	\|u - \ell_j\|_{L^{\infty} ( B_{\rho^j}(x) \cap \{ y_n > \phi(y^{\prime})\} )}  \le \rho^{j(1+\alpha)}
\end{equation}
and
$$
	|\mathrm{a}_j - \mathrm{a}_{j-1}| \le \mathrm{C}_0 \rho^{j(1+\alpha)} \quad \text{and} \quad |\mathrm{b}_j - \mathrm{b}_{j-1}| \le\mathrm{C}_0\rho^{(j-1) \alpha}
$$
for the constant $\mathrm{C}_0 = \mathrm{C}_0(n,\lambda,\Lambda, \|g\|_{C^{1, \beta_g}(\partial \Omega)}, \|\phi\|_{C^2})$.
\end{lemma}

\begin{proof}
We make use of a recursive argument inspired in \cite[Theorem 3.1]{ART15} (see also \cite[Corollary 2.12]{BD2}).  Let us argue by induction. Firstly, such a claim holds for $k=1$ by Lemma \ref{L2}. Suppose that this statement holds true for $j=1,2,\ldots,k$. Now, we are going to check it for $j=k+1$.

Initially, as in \cite{BD16}, one can assume that $g(x^{\prime}) = D_i g(x^{\prime})=0$ for $i=1,\ldots, n-1$.  This will be used when we will later evaluate the $C^{1,\beta_g}$ norm of $g_k$.

  Now, let us to verify the $j=k+1$ step of induction. We define
  $$
  	u_k: B_1(x) \cap \{y_n > \phi_k(y^{\prime})\} \rightarrow \mathbb{R}
$$	
	 given by
 $$
 	u_k(y) = \frac{u(\rho^{k}(y-x) + x) - \ell_k  (\rho^k (y-x) +x)}{\rho^{k(1+\alpha)}}.
 $$
 Then, $u_k$ satisfies in the viscosity sense
$$
\left\{
\begin{array}{rcrcl}
 \mathcal{H}_k(y, \zeta_k + Du_k)F_k(y, D^2 u_k) & = & f_k(y) & \text{in} & B \cap \{y_n > \phi_k(y^{\prime})\}  \\
  u_k(y) & = & g_k(y) & \text{on} & B \cap \{y_n = \phi_k(y^{\prime})\}
\end{array}
\right.
$$
where
{\scriptsize{
\begin{eqnarray*}
	F_k(y, \mathrm{X}) &\defeq & \rho^{k(1-\alpha)} F(\rho^k (y-x) + x, \rho^{k(\alpha-1)} \mathrm{X}) \\
	f_k(y) &\defeq & \rho^{ k(1-\alpha (1+p_k(y)))   } f(\rho^k (y-x) + x)\\
	g_k(y^{\prime}) & \defeq & \rho^{-k(1+\alpha)}  \left \{g(\rho^k (y-x)^{\prime} + x^{\prime}) - b^{\prime}_k \cdot (\rho^k (y-x) +x)^{\prime} -b^n_k \phi(\rho^k (y^{\prime}-x^{\prime}) + x)\right\}\\
\zeta_k & \defeq & \rho^{-k \alpha} \mathrm{b}_k\\
\phi_k(y^{\prime}) & \defeq & x_n \left( 1-\frac{1}{\rho^k}\right) + \frac{1}{\rho^k} \phi \bigg(\rho^k(y^{\prime}-x^{\prime}) + x^{\prime} \bigg)\\
	 \mathcal{H}_k(x, \xi) &\defeq & \rho^{- k \alpha p_k(y)} \mathcal{H} \left( \rho^k(y-x) +x,\rho^{k \alpha} \xi\right) \quad \textrm{satisfies \eqref{1.2} with }
\end{eqnarray*}}}
$$
	p_k(y) \defeq p(\rho^k (y-x) +x), \quad q_k(y) \defeq q(\rho^k (y-x) +x)
$$
and
$$
	\mathfrak{a}_k(y) \defeq \rho^{q_k(y) - p_k(y)} \mathfrak{a} (\rho^k (y-x) +x).
$$

It is clear that $F_k(y, \cdot)$ is also a uniformly $(\lambda,\Lambda)$-elliptic operator. By induction assumption, we can see that
$$
	\|u_k\|_{L^{\infty} ( B_1(x) \cap \{y_n > \phi_k(y^{\prime})\})}   \le 1.
$$
Moreover,
\begin{eqnarray*}
	\|p_k\|_{L^{\infty}(B_1(x) \cap \{y_n > \phi_k(y^{\prime})\})} &\le& \|p\|_{L^{\infty} (  {B_{\rho^k}(x) \cap \{y_n > \phi(y^{\prime})\}}) } \le p_{max}\\
	\|q_k\|_{L^{\infty}(B_1(x) \cap \{y_n > \phi_k(y^{\prime})\})} &\le& \|q\|_{L^{\infty} (  {B_{\rho^k}(x) \cap \{y_n > \phi(y^{\prime})\}}) } \le q_{max}\\
	\|a_k\|_{L^{\infty}(B_1(x) \cap \{y_n > \phi_k(y^{\prime})\})} &\le& \|a\|_{L^{\infty} (  {B_{\rho^k}(x) \cap \{y_n > \phi(y^{\prime})\}}) }
	\end{eqnarray*}
	and
	$$
	0 < p_{\text{min}} \le p_k(y) \le p_{\text{max}} \le q_k(y) \le q_{\text{max}} < +\infty.
	$$

	In view of the choice of $\alpha$ in \eqref{SharpExp} we easily estimate
	\begin{eqnarray*}
		\|f_k\|_{L^{\infty}(B_1(x) \cap \{y_n > \phi_k(y^{\prime})\})} &=& \|\rho^{k (1-\alpha (1+p_k(y)))} f(\rho^k (y-x) +x)\|_{L^{\infty}(B_1(x) \cap \{y_n > \phi_k(y^{\prime})\})}\\
		&\le& \epsilon_0 \rho^{k(1-\alpha (1+p_{\text{max}}))}  \\
		&\le& \epsilon_0.
	\end{eqnarray*}
	Furthermore,
$$
   |D \phi_k (y^{\prime})| = |D \phi (\rho^k (y^{\prime}-x^{\prime}) +x^{\prime})| \quad \text{and} \quad
	 	D^2 \phi_k(y^{\prime}) = \rho^k D^2 \phi(\rho^k (y^{\prime}-x^{\prime}) + x^{\prime}),
$$
Hence,
$$
\begin{array}{rcl}
  \|D^2 \phi_k\|_{L^\infty(B_1(x) \cap \{y_n > \phi(y^{\prime})\})} & = & \displaystyle \rho^k \|D^2 \phi\|_{L^{\infty}(B_1(x) \cap \{y_n > \phi_k(y^{\prime})\})} \\
   & \le & \displaystyle \|D^2 \phi\|_{L^\infty(B_1(x) \cap \{y_n > \phi(y^{\prime})\})}.
\end{array}
$$

	Finally, as remarked at the beginning of the proof, we need to check that $g_k$ is uniformly bounded in $C^{1, \beta_g}$, only when $x$ is on the boundary. Let us recall that, for all $y^{\prime}$ and $z$,
	\begin{equation} \label{E1}
		|D g(y^{\prime}) - Dg (z^{\prime})| \le |y^{\prime}-z^{\prime}|^{\beta_g}.
	\end{equation}
	Next, remember that, since
	$$
		\|u - \ell_k\|_{L^{\infty} \big(B_{\rho^k}(x) \cap \{y_n > \phi(y^{\prime})\}\big)} \le \rho^{k(1+\alpha)}
	$$
	then
	\begin{equation} \label{E2}
		\|g(y^{\prime}) - b^{\prime}_k \cdot (y^{\prime}-x^{\prime}) - b^n_k \phi(y^{\prime})\|_{L^{\infty} \big( B_{\rho^k}(x) \cap \{y_n =\phi(y^{\prime})\} \big)}  \le \rho^{k(1+\alpha)}.
	\end{equation}

	Now, using $g(x^{\prime}) = u(x) =0$, since $x$ is on the boundary, using \eqref{E1} together with the fact that $D^2 \phi$ is bounded, one obtains the required uniform bound on the $C^{1,\beta_g}$ norm of $g_k$.
	
		Therefore, the assumptions in Lemma \ref{L2} are satisfied. Thus, there exists an affine function
		$$
			\tilde{\ell}(y) = \tilde{\mathrm{a}} + \tilde{\mathrm{b}} \cdot (y-x)
		$$
		with
		$$
			|\tilde{\mathrm{a}}| + |\tilde{\mathrm{b}}| \le C_0(n,\lambda,\Lambda, \|g\|_{C^{1, \beta_g}(\partial \Omega)}, \|\phi\|_{C^2(\Omega)})
		$$
		such that
		\begin{equation} \label{EST}
			\|u_k(y) - \tilde{l}(y)\|_{L^{\infty} \big( B_{\rho}(x) \cap \{y_n > \phi(y^{\prime})\} \big)} \le \rho^{1+\alpha}.
		\end{equation}
		
		Next, we label
	$$
		\ell_{k+1}(y) \defeq \mathrm{a}_{k+1} + \mathrm{b}_{k+1} \cdot (y-x),
	$$
	where
	$$
		\mathrm{a}_{k+1} = \mathrm{a}_k + \rho^{k(1+\alpha)} \tilde{\mathrm{a}} \quad \textrm{and} \quad \mathrm{b}_{k+1} = \mathrm{b}_k + \rho^{k \alpha} \tilde{\mathrm{b}}.
	$$
	Thus, scaling \eqref{E2} back to unit domain, we reach that
	$$
	\|u - \ell_{k+1}\|_{ L^{\infty} \big ( B_{ \rho^{k+1}} (x) \cap \{y_n > \phi(y^{\prime})\} \big)  } \le \rho^{(k+1)(1+\alpha)}
	$$
	with
	$$
		|\mathrm{a}_{k+1} -\mathrm{a}_k| \le \mathrm{C}_0 \rho^{k(1+\alpha)} \quad \textrm{and} \quad |\mathrm{b}_{k+1} - \mathrm{b}_k| \le \mathrm{C}_0 \rho^{k \alpha},
	$$
which are the desired conclusion.
\end{proof}

\begin{corollary} \label{Cor1}
Suppose that hypotheses of Lemma \ref{L3} are in force. Then, there exists an affine function $\tilde{\ell}(y) = \tilde{\mathrm{a}} + \tilde{\mathrm{b}} \cdot (y-x)$ with
$$
	|\tilde{\mathrm{a}}| + |\tilde{\mathrm{b}}| \le \mathrm{C}_0
$$
such that for each $0 < r \le \rho$ with $\rho$ being the one in Lemma \ref{L3}
\begin{equation} \label{Cor2}
	\|u - \tilde{\ell}\|_{L^{\infty} \big(B_r(x) \cap \{y_n > \phi(y^{\prime})\}\big)} \le \mathrm{C}(\mathrm{C}_0, \alpha).r^{1+\alpha}.
\end{equation}
\end{corollary}

\begin{proof}
From Lemma \ref{L3} we know that $(\mathrm{a}_j)_j \subset \mathbb{R}$ and $(b_j)_j \subset \mathbb{R}^n$ are Cauchy sequences, hence they converge. Now, we denote
$$
\displaystyle	\tilde{\mathrm{a}} \defeq \lim_{j \to \infty} \mathrm{a}_j \quad \text{and} \quad \tilde{\mathrm{b}} = \lim_{j \to \infty} \mathrm{b}_j.
$$
In the sequel, for any $m \ge j$, we have
\begin{eqnarray*}
	|\mathrm{a}_j - \mathrm{a}_m| &\le& |\mathrm{a}_j - \mathrm{a}_{j+1}| + |\mathrm{a}_{j+1} -\mathrm{a}_{j+2}| + \ldots + |\mathrm{a}_{m-1} -\mathrm{a}_m|\\
	&\le& \mathrm{C}_0 \rho^{j(1+\alpha)} + \mathrm{C}_0 \rho^{(j+1)(1+\alpha)} + \ldots + \mathrm{C}_0 \rho^{(m-1)(1+\alpha)}\\
	&=& \mathrm{C}_0 \rho^{j(1+\alpha)} \frac{1-\rho^{(m-j)(1+\alpha)   }}{1-\rho^{1+\alpha}}.
\end{eqnarray*}
Then, letting $m \to \infty$, we get
\begin{equation}\label{EqSeqa_j}
  	|\mathrm{a}_j -\tilde{\mathrm{a}}| \le \mathrm{C}_0 \frac{\rho^{j(1+\alpha)}}{1-\rho^{1+\alpha}}.
\end{equation}

In a similar way, we may show that
\begin{equation}\label{EqSeqb_j}
	|\mathrm{b}_j - \tilde{\mathrm{b}}| \le \mathrm{C}_0 \frac{\rho^{j \alpha}}{1-\rho^{\alpha}}.
\end{equation}

Now, by fixing $0 < r \le \rho$, we can take $j \in \mathbb{N}$ such that
$$
	\rho^{j+1} < r \le \rho^j.
$$

In conclusion, from sentences \eqref{4}, \eqref{EqSeqa_j} and \eqref{EqSeqb_j} we obtain
\begin{eqnarray*}
\|u-\tilde{\ell}\|_{L^{\infty}\big(B_r(x) \cap \{y_n > \phi(y^{\prime})\}\big)} &\le& \|u-\tilde{\ell}\|_{L^{\infty}\big(B_{\rho^{j}}(x) \cap \{y_n > \phi(y^{\prime})\}\big)}\\
&\le&  \|u-\ell_j\|_{L^{\infty}\big(B_{\rho^{j}}(x) \cap \{y_n > \phi(y^{\prime})\}\big)} \\
&+&  \|\ell_j-\tilde{\ell}\|_{L^{\infty}\big(B_{\rho^{j}}(x) \cap \{y_n > \phi(y^{\prime})\}\big)}\\
&\le& \rho^{j(1+\alpha)} + |\mathrm{a}_j - \tilde{\mathrm{a}}| + \rho^j |\mathrm{b}_j - \tilde{\mathrm{b}}| \\
&\le& \rho^{j(1+\alpha)} + \frac{\mathrm{C}_0}{1-\rho^{\alpha}} \rho^{j(1+\alpha)}\\
&\le& \frac{1}{\rho^{1+\alpha}} \left( 1 + \frac{\mathrm{C}_0}{1-\rho^{\alpha}}\right). r^{1+\alpha}.
\end{eqnarray*}
\end{proof}

Finally, we are in a position of completing the proof of Theorem \ref{main1}.

\begin{proof}[{\bf Proof of Theorem \ref{main1}}]
Suppose now that $u$ is a bounded viscosity solution of \eqref{1.1}, for a general $f$ bounded in $L^{\infty}$. By smallest regime (Remark \ref{SmallRegime2}), we may assume that $u$ and $f$ and $g$ are under the condition of Corollary \ref{Cor1}, in any ball that intersects $\partial \Omega$. So for any $x \in \overline{\Omega}$, \eqref{Cor2} is satisfied. Finally, a standard procedure, illustrated in \cite[Lemma 2.6]{BD3}, implies that $u$ is in $C^{1,\alpha} \big( B_{\rho}(x) \cap \overline{\Omega} \big)$, thereby ending the proof of Theorem \ref{main1}.
\end{proof}

\begin{proof}[{\bf Proof of Corollary \ref{Cor2}}]
Since solutions to the homogeneous problem (with ``frozen coefficients'') for such classes of operators are $C_{\text{loc}}^{1,1}$ close the boundary, see \cite{SS14}. Hence, we may select $\alpha_{\mathrm{F}}=1$. Therefore, from Theorem \ref{main1}, we are able to choose $\alpha = \frac{1}{p_{\text{max}}+1} \in (0, 1)$.
\end{proof}

\section{Finer estimates and further applications}\label{sec.Applic}

\subsection{Examples and improvements}

Let us, by way of illustration, explore the assumptions and implications behind the main Theorem \ref{main1*}.

\begin{example} Firstly, we consider the highly oscillating function at origin $p: B_1 \to \R_{+}$
$$
p(x) \defeq \left\{
\begin{array}{lcr}
  e^{-|x|}\cos\left(\frac{1}{|x|}\right) + \frac{3}{2} & \text{if} & |x| \neq 0 \\
  \frac{1}{2} & \text{if} & x=0.
\end{array}
\right.
$$
Thus, $\displaystyle \inf_{B_1} p(x) = \frac{1}{2}$ and $\displaystyle \sup_{B_1} p(x) = \frac{5}{2}$.

Therefore, according to Corollary \ref{Cormain1} any viscosity solution of
$$
\left[|Du|^{p(x)}+\mathfrak{a}(x)|Du|^{q(x)}\right] \mathscr{M}_{\lambda, \Lambda}^{-}(D^2 u) =  f(x)  \quad \text{in}  \quad B_1
$$
belongs to $C_{\text{loc}}^{1, \frac{2}{7}}(B_1)$, which agrees with constant case $p(x) \equiv \frac{5}{2}$.
\end{example}

\begin{example} For a fixed $\kappa>\frac{1}{7}$, let us consider the dramatically discontinuous function $p: B_1 \to \R_{+}$
$$
p(x) \defeq \left\{
\begin{array}{lcr}
  \kappa & \text{if} & |x| \in \mathbb{Q} \\
  \frac{1}{7} & \text{if} & |x| \notin \mathbb{Q}.
\end{array}
\right.
$$
Hence, $\displaystyle \inf_{B_1} p(x) = \frac{1}{7}$ and $\displaystyle \sup_{B_1} p(x) = \kappa$.

Consequently, from Corollary \ref{Cormain1} any viscosity solution of
$$
\left[|Du|^{p(x)}+\mathfrak{a}(x)|Du|^{q(x)}\right] \mathscr{M}_{\lambda, \Lambda}^{+}(D^2 u) =  f(x)  \quad \text{in}  \quad B_1
$$
belongs to $C_{\text{loc}}^{1, \frac{1}{\kappa+1}}(B_1)$. As before, such a regularity estimate agrees with constant case $p(x) \equiv \kappa$
\end{example}

The above examples states that sharp and optimal local regularity, for variable exponents (bounded away from zero at infinity) and a convex/concave operator, depends only on the upper bound of function $p(\cdot)$, even in the case of very irregular exponents.

In such a way, the following pivotal question arises: what we must expect (in a regularity point of view) when $p(\cdot)$ (resp. $q(\cdot)$) enjoys a universal modulus of continuity? Can we obtain finer regularity estimates in such scenario?

For that end, we need some additional assumptions to obtain such finer estimates. Suppose that $p(\cdot)$ and $q(\cdot)$ have a universal modulus of continuity, i.e., there is a non-decreasing function $\hat{\omega}: [0, +\infty) \to [0, +\infty)$ such that
$$
|p(x)-p(y)|+|q(x)-q(y)| \leq \mathfrak{L}_0\hat{\omega}(|x-y|) \quad \text{with} \quad \hat{\omega}(0)=0,
$$
and which satisfies the balancing condition
\begin{equation}\label{EqBalancCondit}
  \displaystyle \limsup_{s \to 0^{+}} \,\,\hat{\omega}(s)\ln(s^{-1}) \leq \mathfrak{L}< \infty.
\end{equation}

Finally, we can present our finer regularity result (cf. \cite[Theorem 7.2]{BPRT20}).

\begin{theorem}[{\bf Improved point-wise estimates}]\label{ImprovedPoint} Assume that assumptions of Theorem \ref{main1*} are in force. Suppose $F$ to be a concave/convex operator and assumption \eqref{EqBalancCondit} holds. Then, for any $x_0 \in B_{\frac{1}{2}}$ we have that $u$ is $C_{\text{loc}}^{1, \frac{1}{p(x_0)+1}}$ at $x_0$. Moreover, there holds
{\scriptsize{
$$
       \displaystyle  \displaystyle \sup_{B_r(x_0)} \,\,\frac{|u(x)-u(x_0) - Du(x_0)\cdot (x-x_0)|}{r^{1+ \frac{1}{1+p(x_0)}}} \leq \mathrm{C}(n, \lambda, \Lambda, \mathfrak{L}_0, \mathfrak{L})\cdot\left(\|u\|_{L^{\infty}(\Omega)}  +\|f\|_{L^{\infty}(\Omega)}^{\frac{1}{p_{\text{min}}+1}}+1\right).
$$}}
\end{theorem}

\begin{remark} In contrast with Theorem \ref{main1*}, in the above result, we obtain a finer point-wise estimate, because
$$
   1+ \frac{1}{1+p(x_0)} \ge 1+ \frac{1}{1+p_{\text{max}}} \quad \forall\,\, x_0 \in B_{1/2}.
$$
Furthermore, different from \cite[Theorem 7.2]{BPRT20}, in our approach we withdraw the restriction of analyzing $C^{1, \alpha}$ estimates along an \textit{a priori} unknown set of singular points of solutions, i.e. $\mathcal{S}_0(u, \Omega^{\prime})$, where the uniform ellipticity character of the operator is lost.
\end{remark}

\begin{example} Let us consider variable exponents fulfilling the Log-condition:
$$
|p(x)-p(y)|+|q(x)-q(y)| \leq \frac{\omega^{\ast}(|x-y|)}{|\ln(|x-y|^{-1})|} \quad \forall\,\,\, x, y \in \Omega, ,\,\, x \neq y,
$$
for a universal modulus of continuity $\omega^{\ast}: [0, +\infty) \to [0, +\infty)$. Notice that the function $s \mapsto \frac{\omega^{\ast}(s)}{|\ln(s^{-1})|}$ is non-decreasing on $(0, s^{\ast})$ with $\displaystyle \lim_{s \to 0}  \frac{\omega^{\ast}(s)}{|\ln(s^{-1})|}=0$.

Hence, assumption \eqref{EqBalancCondit} does hold. Particularly, such a class embraces the so-called \textit{Log-H\"{o}lder condition}, i.e.
$$
|p(x)-p(y)|+|q(x)-q(y)| \leq \frac{\mathfrak{c}_{p, q}}{-\ln(|x-y|)} \quad \forall\, x, y \in \Omega, \,\, \text{with} \,\,|x - y|\le \frac{1}{2},
$$
which plays a decisive role in proving regularity estimates in many context of elliptic PDEs with non-standard growth (see \cite[Section 4.1]{DHHR17}, \cite[Section 4]{MingRad21} and \cite[Part 1]{RaRe15}).
\end{example}

\begin{proof}[{\bf Proof of Theorem \ref{ImprovedPoint}}] We will revisit the proof of Theorem \ref{main1*}, making use of assumption \eqref{EqBalancCondit} appropriately. First, WLOG we may suppose that $x_0 = 0$. Now, by following \cite[Theorem 7.2]{BPRT20} we can estimate
$$
  \rho^{1-\frac{1+p(\rho x)}{1+p(0)}} \le \left[\left(\rho^{-1}\right)^{\mathfrak{L}_0 \omega(\rho)}\right]^{\frac{1}{1+p(0)}} \quad \text{for all} \quad x \in B_1.
$$
Moreover, assumption \eqref{EqBalancCondit} yields
$$
\left(\rho^{-1}\right)^{\omega(\rho)} \le \left(\rho^{-1}\right)^{\sqrt{2}\mathfrak{L} \ln(\rho^{-1})} \le e^{\sqrt{2}\mathfrak{L}}.
$$
As a result, for some $\rho_0 \ll 1$ we get
\begin{equation}\label{EqRHO}
\rho^{1-\frac{1+p(\rho x)}{1+p(0)}} \le \Phi(\mathfrak{L}_0, \mathfrak{L}) \quad \text{for all} \quad \rho\ll \rho_0.
\end{equation}
At this point, arguing as in the proof of Lemmas \ref{c3.1} and \ref{lem.dy}, we can establish
\begin{equation}\label{ImpEstEquationIter}
  \displaystyle\sup_{B_{\rho^k}(0)}\left|u(x)-u(0)\right|\leq\rho^{k\left(1+\frac{1}{1+p(0)}\right)}+|D u (0)|\sum_{j=0}^{k-1}\rho^{k+\frac{j}{1+p(0)}} \quad \forall \, k \in \mathbb{N},
\end{equation}
provided the smallness regime (see Remark \ref{SmallRegime}) and \eqref{EqRHO} are in force
$$
    \max\left\{\Theta_{\mathrm{F}}(x), \left\|f\right\|_{L^{\infty}(B_{1})}\right\} \leq \frac{\delta_{\iota}}{\Phi(\mathfrak{L}_0, \mathfrak{L})+1}.
$$

The previous estimates in \eqref{ImpEstEquationIter} imply the statement of Lemma \ref{l3.3}:
$$
\displaystyle \sup_{B_{r}(0)} \frac{|u(x)-u(0)|}{r^{1+\frac{1}{1+p(0)}}}\leq \mathrm{M}_0 \cdot \left(1+|D u(0)|r^{-\frac{1}{1+p(0)}}\right),\,\,\forall r\in(0,\rho),
$$
which help us to complete the desired proof.
\end{proof}

\bigskip

In the sequel, we will present scenarios where our results also take place.

\subsection{Sharp estimates to Strong $p(x)-$Laplace under constraints}

Historically, the strong $p(x)-$Laplace operator dates back to Adamowicz-H\"{a}st\"{o}'s fundamental works
in \cite{AH10} and \cite{AH11}, which were directly inspired on the strong form of the $p-$Laplace operator, i.e.,
$$
\hat{\Delta}_p u \defeq |\nabla u|^{p-4}(|\nabla u|^2 \Delta u + (p-2)\Delta_{\infty} u),
$$
where
$$
\displaystyle \Delta_{\infty} u \defeq \sum_{i,j=1}^{n} u_{x_i}u_{x_j}u_{x_i x_j} = \nabla u \cdot D^2 u  \nabla u^{t}
$$
is the nowadays well-known \textit{$\infty-$Laplacian operator} (see \cite{ACJ04} for a complete historical description on this theme).

Now, if we replace $p$ by $p(x)$ in the above definition, it yields a kind
of generalization of the $p-$Laplace operator, namely the strong $p(x)-$Laplace operator
$$
  -\Delta^S_{p(x)} u = - \Div(|\nabla u|^{p(x)-2}\nabla u)+ |\nabla u|^{p(x)-2} \log(|\nabla u|) \nabla u \cdot \nabla p.
$$

Recently, Siltakoski in \cite{Silt18} has considered viscosity to the so-called normalized $p(x)-$Laplace equation given by
\begin{equation}\label{Normp(x)-Lapla}
   -\Delta^N_{p(x)} u \defeq -\Delta u - \frac{p(x)-2}{|Du|^2} \Delta_{\infty} u = 0,
\end{equation}
whose interest in such a class of equation stems from its deep connect between PDEs and Probability Theory. Precisely, stochastic Tug-of-war games with spatially varying probabilities (see \cite{AHP17} for related topic).

Formally, one has
$$
-|\nabla u|^{p(x)-2}\Delta^N_{p(x)} u =  - \Div(|\nabla u|^{p(x)-2}\nabla u)+ |\nabla u|^{p(x)-2} \log(|\nabla u|) \nabla u \cdot \nabla p.
$$

As a consequence, a notion of weak and viscosity solutions can be based on the strong $p(x)-$Laplacian equation
\begin{equation}\label{Stronp(x)-Lapla}
 -\Delta^S_{p(x)} u = 0.
\end{equation}
In effect, Siltakoski proved (see \cite[Theorem 4.1]{Silt18}) that viscosity solutions of \eqref{Normp(x)-Lapla} are equivalent to viscosity solutions of \eqref{Stronp(x)-Lapla}, provided $p \in C^{0, 1}(\Omega)$ and $\displaystyle \inf_{\Omega} p(x)>1$. Particularly, viscosity solutions to \eqref{Normp(x)-Lapla} are $C_{\text{loc}}^{1, \alpha^{\prime}}$ for some $\alpha^{\prime} \in (0, 1)$. Additionally, Siltakoski addressed the following result;

\begin{lemma}[{\cite[Lemma 6.2]{Silt18}}]\label{LemmaEquiv} A function $u$ is a viscosity solution to \ref{Normp(x)-Lapla} if and only if it is a viscosity solution to
{\scriptsize{
\begin{equation}\label{EqAplic}
 -|\nabla u|^2\Delta u-(p(x)-2)\Delta_{\infty} u =   -\tr\left((|\nabla u|^2 \mathrm{Id}_n + (p(x)-2) \nabla u \otimes \nabla u)D^2u\right) = 0,
\end{equation}}}
\end{lemma}

Particularly, as a consequence of Lemma \ref{LemmaEquiv}, \cite[Theorem 4.1]{Silt18} and our findings (i.e. Theorem \ref{main1*}) we are able to prove the sharp result estimates:

\begin{theorem}\label{Thm5.1} Let $u$ be a bounded weak solution to \eqref{Stronp(x)-Lapla}. Assume further the assumptions on \cite[Theorem 4.1]{Silt18} and Lemma \ref{LemmaEquiv} are in force, and $\mathcal{G}[u]  = (p(x)-2)\Delta_{\infty} u \in L^{\infty}(B_1)$. Then, $u \in C_{\text{loc}}^{1, \frac{1}{3}}(B_1)$. Moreover, there holds
	$$
	\displaystyle [u]_{C^{1, \frac{1}{3}}\left(B_{\frac{1}{2}}\right)}\leq  \mathrm{C}(\verb"universal")\cdot \left[\|u\|_{L^{\infty}(B_1)} + \sqrt[3]{\|\mathcal{G}[u]\|_{L^{\infty}(B_1)}}\right].
	$$
\end{theorem}

It is worth to stress that $\mathcal{G}$ represents the ``bad part'' of the operator, since it does not satisfy the structural assumptions (A0)-(A2). However, if it is possible to obtain a sort of uniform bound for such a term, we can obtain a sharp regularity estimate (cf. \cite[Proposition 4.6]{ART15}).

Theorem \ref{Thm5.1} might be interpreted as a sort of weaker form of the longstanding $C^{1, \frac{1}{3}}$ conjecture for inhomogeneous $\infty-$Laplacian equation, which states that if $v$ is a bounded viscosity solution to
$$
\Delta_{\infty} v = f \in L^{\infty}(B_1) \qquad \Rightarrow \qquad v \in C_{\text{loc}}^{1, \frac{1}{3}}(B_1).
$$

In this direction, \cite{ES08}, \cite{LMZ19} and \cite{LW08} are the best known regularity estimates to infinity-Laplacian equation. The full regularity's description in any dimension is still a challenging, open problem.

\subsection{Regularity for problems with $(p(x), q(x))-$growth}

Our methods also allow us to improve (to some extent) the known regularity results for solutions to non-linear elliptic problems with $(p(x), q(x))-$growth as follows (cf. \cite{RaTach20}):
$$
-\mathcal{L}w(x) \defeq \Div\left(|\nabla w|^{p(x)-2}\nabla w + \mathfrak{a}(x)|\nabla w|^{q(x)-2} \nabla w\right)=0.
$$

Let us remember that such solutions might be obtained as minimizers for functionals of double phase with variable exponents given by
{\small{
$$
\displaystyle  w \mapsto \min_{u_0+W_0^{1,1}(\Omega)} \int_{\Omega} \left( \frac{1}{p(x)}|\nabla w|^{p(x)} + \frac{\mathfrak{a}(x)}{q(x)}|\nabla w|^{q(x)}\right)dx.
$$}}

Recently, in \cite[Theorem 1.2]{RaTach20} a $C^{1, \gamma}_{\text{loc}}(\Omega)$-regularity result was addressed for such minimizers provided that $p(\cdot), q(\cdot)$ and $a(\cdot)$ are real-valued functions on $\Omega$ fulfilling:
\begin{enumerate}
  \item $q(x) \ge p(x) \ge p_0>1$  and $\mathfrak{a}(x)\ge 0$;
  \item $p(\cdot), q(\cdot) \in C^{0, \sigma}(\Omega)$, $a(\cdot) \in C^{0, \alpha^{\prime}}(\Omega)$  ($\alpha^{\prime}, \sigma \in (0,1]$);
  \item If $\beta^{\prime} \defeq \min\left\{\alpha^{\prime}, \sigma\right\}$, then
  $$
  \displaystyle \sup_{x \in \Omega} \frac{q(x)}{p(x)}<1+\frac{\beta^{\prime}}{n}.
  $$
\end{enumerate}

Now, due to equivalence of notion solutions in \cite{FZ20} and \cite{Silt18}, and by formal computation such an operator can be written (in non-divergence form) as:
\begin{equation}\label{Eqp&qLaplac3}
  -\mathcal{L}w(x) \defeq \mathcal{J}_{1}(\nabla u, D^2 w) + \mathcal{J}_{2}(\nabla u, D^2 w)= 0,
\end{equation}
where
{\scriptsize{
$$
\left\{
\begin{array}{rcl}
  \mathcal{J}_1(\nabla w, D^2 w) & \defeq & \left[|\nabla w|^{p(x)-2}+\mathfrak{a}(x)|\nabla w|^{q(x)-2}\right]\Delta w(x) \\
   \mathcal{J}_2(\nabla w, D^2 w) & \defeq &  \left[\left(p(x)-2\right)|\nabla w|^{p(x)-2}+\left(q(x)-2\right)\mathfrak{a}(x)|\nabla w|^{q(x)-2}\right]\Delta^N_{\infty} w(x)\\
    &+& \left(|\nabla u|^{p(x)-2}  + |\nabla u|^{q(x)-2} \right)\log(|\nabla u|) \nabla u \cdot \left(\nabla p + \nabla q\right)\\
    &+& |\nabla u|^{q(x)-2} \nabla u \cdot \nabla \mathfrak{a},
\end{array}
\right.
$$}}
and
$$
\displaystyle  \Delta^N_{\infty} w(x) \defeq \frac{1}{|\nabla u|^2} \sum_{i, j} u_{x_i} u_{x_j} u_{x_i x_j}
$$
is the \textit{Normalized $\infty-$Laplacian operator} (see, \cite{LW08II}).

Therefore, under some constraints and invoking our sharp estimates in Theorem \ref{main1*} we are able to prove the next optimal regularity result in contrast with \cite[Theorem 1.2]{RaTach20} (cf. \cite{ALS15},  \cite{ATU17} and \cite{DSVI}).

\begin{theorem}\label{Thm5.2} Let $w \in C^{0, 1}(B_1)$ be a bounded viscosity solution to \eqref{Eqp&qLaplac3}. Assume further $\Delta_{\infty}^N w, \nabla \mathfrak{a}, \nabla p, \nabla q  \in L^{\infty}(B_1)$. Then, $w \in C_{\text{loc}}^{1, \frac{1}{p_{\text{max}}-1}}(B_1)$. Additionally, the following estimate holds
	$$
	\displaystyle [w]_{C^{1, \frac{1}{p_{\text{max}}-1}}\left(B_{\frac{1}{2}}\right)}\leq  \mathrm{C}(\verb"universal")\cdot \left[\|w\|_{L^{\infty}(B_1)} + \|\mathcal{J}_2\|_{L^{\infty}(B_1)}^{\frac{1}{p_{\text{min}}-1}}\right].
	$$
\end{theorem}

As before, Theorem \ref{Thm5.2} is a sort of non-variational weak formulation of the $C^{p^{\prime}}$ conjecture for inhomogeneous $p-$Laplacian equation (cf. \cite{ATU17} for $p>2$), which reads the following:
$$
\text{a bounded solution to} \quad -\Delta_{p} v = f \in L^{\infty}(B_1) \qquad \Rightarrow \qquad v \in C_{\text{loc}}^{1, \frac{1}{p-1}}(B_1).
$$
Finally, the full description for the $C^{p^{\prime}}$ regularity conjecture in any dimension is still an open issue.

\subsection{Sharp regularity in geometric free boundary problems}

We stress that our geometric approach is particularly refined and quite far-reaching in order to be employed in other classes of problems. Particularly, we may to study dead-core problems for fully nonlinear models with unbalanced variable degeneracy as follows
\begin{equation}\label{Maineq}
	\left[|\nabla u|^{p(x)}+\mathfrak{a}(x)|\nabla u|^{q(x)}\right] \Delta u = f_0(x)\cdot u^{\varsigma}\chi_{\{u>0\}} \quad \textrm{in} \quad \Omega,
\end{equation}
where $\varsigma \in [0,  p_{\text{max}}+1)$ is the order of reaction, $f_0$ is bounded away from zero and infinity (the Thiele modulus) and assumptions (A0)-(A5) are in force. We suggest the readers to refer to first's author works \cite{daSRS19I} and \cite{daSS18} (and therein references) for an enlightening mathematical description of such dead-core problems ruled by quasi-linear operators.

In contrast with Theorem \ref{main1} we are able to establish an improved regularity estimate for non-negative solutions of \eqref{Maineq} along their touching ground boundary $\partial\{u>0\}$ (cf. \cite{daSLR20}, \cite{daSRS19I}, \cite{daSS18} and \cite{Tei16} for similar results). The proof for such improved estimate follows the same lines as \cite[Theorem 1.2]{daSLR20} and \cite[Theorem 1.2]{daSS18}. For this reason, we will omit it here.

\begin{theorem}[{\bf Improved regularity along free boundary}]\label{IRThm}

Let $u$ be a nonnegative and bounded viscosity solution to \eqref{Maineq}. Then, given $r_0> 0$ there exists
a constant $\displaystyle \mathrm{C}_0 = \mathrm{C}_0(n, p_{\text{min}}, q_{\text{max}}, r_0, \inf_{\Omega} f(x))>0$ such that for any $x_0 \in \Omega$ such that $B_{r_0}(x_0) \subset \Omega$ and any $r \le \frac{r_0}{2}$, the following estimate holds
$$
\displaystyle  \sup_{B_r (x_0)} u(x) \leq \mathrm{C}_0\cdot \min\left\{  \inf_{B_r (x_0)} u(x), \,\, r^{\frac{p_{\text{max}}+2}{p_{\text{max}}+1-\varsigma}}\right\}.
$$

In particular, if $x_0 \in \partial \{u>0\} \cap \Omega$ (i.e. a free boundary point), then
$$
\displaystyle  \sup_{B_r (x_0)} u(x) \leq \mathrm{C}_0 \cdot r^{\frac{p_{\text{max}}+2}{p_{\text{max}}+1-\varsigma}}.
$$
\end{theorem}

Similarly to Theorem \ref{main3} we are able to prove the following Non-degeneracy property of solutions. The proof is quite similar to the one in \cite[Theorem 4.3]{daSR20}. Thereby, we leave the details of the proof to the reader.

\begin{theorem}[{\bf Non-degeneracy}]\label{LGR} Let $u$ be a nonnegative, bounded viscosity solution to \eqref{Maineq} in $B_1(0)$ with $f(x) \geq \mathfrak{m}>0$ and let $x_0 \in \overline{\{u >0\}} \cap B_{\frac{1}{2}}(0)$ be a fixed point. Then, for any $0<r<\frac{1}{2}$, there holds
$$
   \displaystyle \sup_{\partial B_r(x_0)} \,\frac{u(x)}{r^{\frac{p_{\text{min}}+2}{p_{\text{min}}+1-\varsigma}}} \geq \mathrm{C}(\mathfrak{m},\|\mathfrak{a}\|_{L^{\infty}(\Omega)}, \mathrm{L}_1, n, p_{\text{min}}, q_{\text{max}}, \varsigma, \Omega).
$$
\end{theorem}

\hspace{2cm}

Now, let us move towards on sharp regularity for solutions of obstacle type problems governed by our class of operators. Precisely, we deal with viscosity solutions of
\begin{equation} \label{Eq1}
\left\{
\begin{array}{rclcl}
 \left[|\nabla u|^{p(x)}+\mathfrak{a}(x)|\nabla u|^{q(x)}\right]\Delta u & = & f(x)\cdot \chi_{\{u>\phi\}} & \text{in} & B_{1}(0) \\
  u(x) & \ge & \phi(x) & \text{on} &  B_{1}(0)\\
  u(x) &=& g(x) & \text{on} & \partial B_1(0),
\end{array}
\right.
\end{equation}
with $\phi \in C^{1,1}(B_1(0))$, $f \in L^{\infty}(B_1) \cap C^0(B_1(0))$ and $g$ is a continuous boundary datum.  We are able to prove $C^{1, \frac{1}{p_{\text{max}}+1}}(B_{1/2}(0))$ estimates.

Recently, nonlinear elliptic obstacle problems have been addressed in first's author works\cite{DSVI} and \cite{DSVII}, where were proved $C^{1, \alpha}$ estimates for a class of fully nonlinear operators of degenerate type $\mathcal{G}_0[u] \defeq |\nabla u|^{p}F(D^2 u)$.

Finally, making use of same ideas as in \cite[Theorem 1.3]{DSVII}, we are in a position to state the following result:

\begin{theorem}[{\bf Regularity along free boundary points}]\label{T1}
Suppose that the assumption (A0)-(A5)  are in force for a convex or concave operator $F$. Let $u$ be a bounded viscosity solution to \eqref{Eq1} with obstacle $\phi \in C^{1,1}(B_1(0))$. Then, $u \in C^{1,\frac{1}{p_{\text{max}}+1} }(B_{1/2})$, along free boundary points. More precisely, for any point $x_0 \in \partial \{u > \phi\} \cap B_{1/2}$ and for $r \in \left(0, \frac{1}{2}\right)$ there holds
{\scriptsize{
$$
	[u]_{C^{1, \frac{1}{p_{\text{max}}+1}}(B_r)} \le \mathrm{C}(\verb"universal")\cdot \left[\|u\|_{L^{\infty}(B_1)} + \left(\|\phi\|_{C^{1, 1}(B_1)}^{p_{\text{max}}+1}+\|f\|_{L^{\infty}(B_1)}\right)^{\frac{1}{p_{\text{max}}+1}}\right]	,
$$}}
\end{theorem}

In contrast with Theorems \ref{main3} and \ref{LGR}, we also may address the following non-degeneracy result. The proof makes use of the same ideas as in \cite[Theorem 1.7]{DSVI} and  \cite[Theorem 1.7]{DSVII}. For this reason, we leave the details of the proof to the reader.

\begin{theorem}[{\bf Non-degeneracy property}]\label{ThmNon-Deg} Suppose that the assumptions of Theorem \ref{T1} are in force. Let $u$ be a bounded a viscosity solution to the obstacle problem \eqref{Eq1} with source term satisfying $\displaystyle \inf_{B_1(0)}f(x) \defeq \mathfrak{m}>0$. Given $x_0 \in \{u>\phi\} \cap \Omega^{\prime}$, then for all $0<r \ll 1$ there holds
$$
  \displaystyle \sup_{\partial B_r(x_0)} \frac{u(x)-\phi(x_0)}{r^{\frac{p_{\text{min}}+2}{p_{\text{min}}+1}}}\geq \mathrm{c}(\mathfrak{m},\|\mathfrak{a}\|_{L^{\infty}(\Omega)}, \mathrm{L}_1, n, p_{\text{min}}, q_{\text{max}}, \Omega^{\prime}).
$$
\end{theorem}

\hspace{2cm}

Finally, we might consider to address regularity estimates to nonlinear free boundary problems (for short FBP) for fully non-linear elliptic problem with unbalanced variable degeneracy as follows
\begin{eqnarray}\label{P 5.1. introduc}
	\left \{
		\begin{array}{rclcl}
			\left[|\nabla u|^{p(x)}+\mathfrak{a}(x)|\nabla u|^{q(x)}\right] \Delta u & = & f(x) & \text{ in } & \Omega_{+}\left( u \right),\\
			|\nabla u| & = & \mathrm{Q}(x) & \text{ on } & \mathfrak{F}(u),
		\end{array}
	\right.
\end{eqnarray}
where the assumptions (A0)-(A5) are satisfied, $\mathrm{Q} \geq 0$ is a continuous function and
$$
u\geq 0\text{ in }\Omega,\quad	\Omega^{+}(u)\defeq \{x \in \Omega : u(x) >0\} \quad \textrm{and} \quad \mathfrak{F}(u) \defeq \partial \Omega^{+}(u) \cap \Omega.
$$

In such a way, by following the approach developed in the authors' work \cite[Theorems 1.3 and 1.4]{daSRRV21} we might address fine properties of the free boundary:
$$
\text{ {\bf Flat/Lipschitz} free boundaries are} \,\,\,C_{\text{loc}}^{1, \gamma}\,\,\, \text{for some} \,\, \gamma(\verb"universal") \in (0, 1).
$$

In turn, the FBP considered in \eqref{P 5.1. introduc} appears as the limit of certain inhomogeneous singularly perturbed problems in the non-variational framework of high energy activation model in combustion and flame propagation theories ( cf. \cite{CK-AM04} and \cite{Weiss03} for the stationary divergence setting and \cite{ART17}, \cite{RS15} and \cite{RT11} for related non-divergence topics), whose simplest model case is given by we seek a non-negative profile $u^{\varepsilon}$ (for each $\varepsilon>0$ fixed) satisfying
\begin{equation}\label{SPP}
  \left\{
\begin{array}{rclcl}
	\left[|\nabla u^{\varepsilon}|^{p_{\varepsilon}(x)} + \mathfrak{a}_{\varepsilon}(x)|\nabla u^{\varepsilon}|^{q_{\varepsilon}(x)} \right] \Delta u^{\varepsilon} & = & \zeta_{\varepsilon}(u^{\varepsilon}) + f_{\varepsilon}(x) & \mbox{in} & \Omega,\\
	u^{\varepsilon}(x) & = & g(x) & \mbox{on} & \partial \Omega,
\end{array}
\right.
\end{equation}
in the viscosity sense, for suitable data $p_{\varepsilon}(\cdot), q_{\varepsilon}(\cdot)$, $\mathfrak{a}_{\varepsilon}(\cdot)$, $g$, where $\zeta_{\varepsilon}(s) \defeq \frac{1}{\varepsilon}\zeta\left(\frac{s}{\varepsilon}\right)$ one behaves singularly of order $\mbox{o} \left(\varepsilon^{-1} \right)$ near $\varepsilon$-layer surfaces. In such a scenario, existing solutions are globally (uniformly) Lipschitz continuous (see \cite[Theorem 1.4]{daSJR21} and \cite[Theorem 1]{RS15} for specific results) such that
{\scriptsize{
$$
\|\nabla u^{\varepsilon}\|_{L^{\infty}(\overline{\Omega})} \leq \mathrm{C}(n, (p_{\varepsilon})_{\text{min}}, (q_{\varepsilon})_{\text{max}}, \|\zeta\|_{L^{\infty}(\Omega)}, \|\mathfrak{a}_{\varepsilon}\|_{L^{\infty}(\Omega)}, \|g\|_{C^{1, \kappa}(\partial \Omega)}, \|f_{\varepsilon}\|_{L^{\infty}(\Omega)}, \Omega).
$$}}

 As a result, up to a subsequence, there exists a function $u_0$, obtained as the uniform limit of $u^{\varepsilon_{k_j}}$, as $\varepsilon_{k_j} \to 0$. Additionally, such a limiting profile fulfills in the viscosity sense
\[
\left[|\nabla u_0|^{p_{0}(x)} + \mathfrak{a}_0(x)|\nabla u_0|^{q_{0}(x)} \right] \Delta u_0 = f_0(x)
\]
for an appropriate RHS $f_0\geq 0$ (see \cite[Theorem 1.7]{daSJR21}).

Therefore, we are able to apply our findings to a family of functions of such inhomogeneous singular perturbation problems \eqref{SPP} as we have developed in \cite{ART17}, \cite{daSJR21} and \cite{daSRRV21} respectively.

\bigskip

\noindent{\bf Acknowledgements.} J.V. da Silva and G.C. Ricarte have been partially supported by Conselho Nacional de Desenvolvimento Cient\'{i}fico e Tecnol\'{o}gico (CNPq-Brazil) under Grants No. 310303/2019-2 and No. 303078/2018-9. E.C. J\'{u}nior and G.C. Rampasso are partially supported by CAPES-Brazil. This research was financed in part by the Coordena\c{c}\~{a}o de Aperfei\c{c}oamento de Pessoal de N\'{i}vel Superior - (CAPES - Brazil) - Finance Code 001.


\begin{thebibliography}{99}


\bibitem{AH10} T. Adamowicz and P. H\"{a}st\"{o},
\textit{Mappings of finite distortion and PDE with nonstandard growth},
Int. Math. Res. Not. IMRN, 10 (2010), 1940-1965.

\bibitem{AH11} T. Adamowicz and P. H\"{a}st\"{o},
\textit{Harnack's inequality and the strong $p(x)-$Laplacian},
J. Differential Equations, 250 (2011), 1631-1649.

\bibitem{AdaSRT19} M.D. Amaral, J.V. da Silva, G.C. Ricarte and R. Teymurazyan,
\textit{Sharp regularity estimates for quasilinear evolution equations}.
Israel J. Math. 231 (2019), no. 1, 25-45.

\bibitem{ALS15} J. Andersson, E. Lindgren and H. Shahgholian,
\textit{Optimal regularity for the obstacle problem for the $p-$Laplacian}.
J. Differential Equations 259 (2015), no. 6, 2167-2179.

\bibitem{ART15} D.J. Ara\'{u}jo, G.C. Ricarte and E.V. Teixeira,
\textit{Geometric gradient estimates for solutions to degenerate elliptic equations}.
Calc. Var. Partial Differential Equations 53 (2015), 605-625.

\bibitem{ART17} D.J. Ara\'{u}jo, G.C. Ricarte and E.V. Teixeira,
\textit{Singularly perturbed equations of degenerate type}.
Ann. Inst. H. Poincar\'{e} Anal. Non Lin\'{e}aire 34 (2017), no. 3, 655-678.

\bibitem{AS21} D.J. Ara\'{u}jo and B. Sirakov,
\textit{Improved boundary and global regularity for degenerate fully nonlinear elliptic equations}.
Arxiv Preprint \url{arXiv:2108.01150}.

\bibitem{ATU17} D.J. Ara\'{u}jo, E.V.  Teixeira and J.M.  Urbano,
\textit{A proof of the $C^{p^{\prime}}$ regularity conjecture in the plane}.
Adv. Math. 316 (2017), 541-553.

\bibitem{ACJ04} G. Aronsson, M.G.  Crandall and P. Juutinen,
\textit{A tour of the theory of absolutely minimizing functions}.
Bull. Amer. Math. Soc. (N.S.) 41 (2004), no. 4, 439–505.

\bibitem{AHP17} \'{A}. Arroyo, J. Heino and M. Parviainen,
\textit{Tug-of-war games with varying probabilities and the normalized $p(x)-$Laplacian}.
 Commun. Pure Appl. Anal. 16(3), 915-944 (2017).

\bibitem{APR} A. Attouchi, M. Parviainen and E. Ruosteenoja,
\textit{$C^{1,\alpha}$ regularity for the normalized $p$-Poisson problem}.
J. Math. Pures Appl. (9) 108 (2017), 553-591.

\bibitem{AR18} A. Attouchi and E. Ruosteenoja,
\textit{Remarks on regularity for $p-$Laplacian type equations in non-divergence form}.
J. Differential Equations 265 (2018), no. 5, 1922-1961.

\bibitem{BRR19} A. Bahrouni, V.D. R\v{a}dulescu and D.D.  Repov\v{s},
\textit{Double phase transonic flow problems with variable growth: nonlinear patterns and stationary waves}. Nonlinearity 32 (2019), no. 7, 2481–2495.

 \bibitem{BV21} A. Benerjee and R. B. Verma
 \textit{$C^{1,\alpha}$ Regularity for Degenerate Fully Nonlinear Elliptic Equations with Neumann Boundary Conditions}.
 To appear in Potential Anal (2021). \url{https://doi.org/10.1007/s11118-021-09918-z}.

\bibitem{BCM15III} P. Baroni, M. Colombo and G. Mingione,
\textit{Regularity for general functionals with double phase}.
Calc. Var. Partial Differential Equations 57 (2018), no. 2, Art. 62, 48 pp.

\bibitem{BD'AFP00} V. Benci, P. D'Avenia, D. Fortunato and L. Pisani,
\textit{Solitons in several space dimensions: Derrick's problem and infinitely many solutions}.
Arch. Ration. Mech. Anal. 154 (2000), no. 4, 297-324.

\bibitem{BD2} I. Birindelli and F. Demengel,
\textit{$C^{1, \beta}$ regularity for Dirichlet problems associated to fully nonlinear degenerate elliptic equations}.
ESAIM Control Optim. Calc. Var. 20 (2014), no. 4, 1009-1024.

\bibitem{BD3} I. Birindelli and F. Demengel,
\textit{H\"{o}lder regularity of the gradient for solutions of fully nonlinear equations with sub linear first order term}.
Geometric methods in PDE's, 257-268, Springer INdAM Ser., 13, Springer, Cham, 2015.

\bibitem{BD16} I. Birindelli and F. Demengel,
\textit{Fully nonlinear operators with Hamiltonian: H\"{o}lder regularity of the gradient}.
NoDEA Nonlinear Differential Equations Appl. 23 (2016), no. 4, Art. 41, 17 pp.

\bibitem{BPRT20} A. Bronzi; E. Pimentel; G.C. Rampasso and E.V. Teixeira,
\textit{Regularity of solutions to a class of variable-exponent fully nonlinear elliptic equations}.
J. Funct. Anal. 279 (2020), no. 12, 108781.

\bibitem{C89} L.A. Caffarelli,
\textit{Interior apriori estimates for solutions of fully nonlinear equations}.
Ann. of Math. (2) 130 (1989), 189-213.

\bibitem{CC95} L.A. Caffarelli and X. Cabr\'{e},
\textit{Fully nonlinear elliptic equations}.
American Mathematical Society Colloquium Publications, 43.
American Mathematical Society, Providence, RI, 1995. vi+104 pp. ISBN: 0-8218-0437-5.

\bibitem{CK-AM04} L.A. Caffarelli, K-A. Lee and A. Mellet,
\textit{Singular limit and homogenization for flame propagation in periodic excitable media}.
Arch. Ration. Mech. Anal. 172 (2004), no. 2, 153–190.

\bibitem{CM15} M. Colombo and G. Mingione,
\textit{Bounded minimisers of double phase variational integrals}.
Arch. Ration. Mech. Anal. 218 (2015), no. 1, 219-273.

\bibitem{CM15II} M. Colombo and G. Mingione,
\textit{Regularity for double phase variational problems}.
Arch. Rational Mech. Anal. 215 (2015) 443-496.

\bibitem{CIL92} M. Crandall; H. Ishii and P-L. Lions,
\textit{User's guide to viscosity solutions of second order partial differential equations.}
Bull. Amer. Math. Soc. (N.S.) {\bf 27} (1992), no. 1, 1--67.

\bibitem{daSJR21} J.V. da Silva; E.C. J\'{u}nior and G.C. Ricarte,
\textit{Fully non-linear singularly perturbed models with non-homogeneous degeneracy}.
Arxiv Preprint \url{arXiv:2101.08664}.

\bibitem{daSLR20} J.V. da Silva; R.A. Leit\~{a}o and  G.C. Ricarte,
\textit{Geometric regularity estimates for fully nonlinear elliptic equations with free boundaries}.
Math. Nachr. 294 (2021), no. 1, 38-55.

\bibitem{daSN21} J.V. da Silva and  G.S. Nornberg,
\textit{Regularity estimates for fully nonlinear elliptic PDEs with general Hamiltonian terms and unbounded ingredients}.
To appear in Calc. Var. Partial Differential Equations, 2021 - \url{10.1007/s00526-021-02082-7}.

\bibitem{daSRRV21} J.V. da Silva; G.C. Rampasso; G.C. Ricarte and H. Vivas,
\textit{Free boundary regularity for a class of one-phase problems with non-homogeneous degeneracy}.
Arxiv Preprint \url{arXiv:2103.11028}.

\bibitem{daSR20} J.V. da Silva and G.C. Ricarte,
\textit{Geometric gradient estimates for fully nonlinear models with non-homogeneous degeneracy and applications}.
Calc. Var. Partial Differential Equations 59, 161 (2020).

\bibitem{daSRS19I} J.V. da Silva, J.D. Rossi and A. Salort,
\textit{Regularity properties for $p-$dead core problems and their asymptotic limit as $p \to \infty$},
J. Lond. Math. Soc. (2) 99 (2019), no. 1, 69-96.

\bibitem{daSS18} J.V. da Silva and A. Salort,
\textit{Sharp regularity estimates for quasi-linear elliptic dead core problems and applications}.
Calc. Var. Partial Differential Equations 57 (2018), no. 3, 57: 83.

\bibitem{DSVI} J.V. da Silva and H. Vivas,
\textit{The obstacle problem for a class of degenerate fully nonlinear operators},
Rev. Mat. Iberoam. 37 (2021), no. 5, 1991-2020.

\bibitem{DSVII} J.V. da Silva and H. Vivas,
\textit{Sharp regularity for degenerate obstacle type problems: a geometric approach},
Discrete Contin. Dyn. Syst. 41 (2021), no. 3, 1359–1385.

\bibitem{DeF20} C. De Filippis,
\textit{Regularity for solutions of fully nonlinear elliptic equations with nonhomogeneous degeneracy}.
Proc. Roy. Soc. Edinburgh Sect. A 151 (2021), no. 1, 110-132.

\bibitem{DeFM19II} C. De Filippis and G. Mingione,
\textit{On the Regularity of Minima of Non-autonomous functionals}.
The Journal of Geometric Analysis, 30 (2) (2020) 1584-1626.

\bibitem{DeFM20} C. De Filippis and G. Mingione,
\textit{Lipschitz bounds and nonautonomous integrals}.
To Appear in  Arch. Ration. Mech. Anal. (2020). \url{arXiv:2007.07469}.

\bibitem{DeFO19} C. De Filippis and J. Oh,
\textit{Regularity for multi-phase variational problems}.
J. Differential Equations 267 (2019), no. 3, 1631–1670.

\bibitem{DHHR17} L. Diening, P. Harjulehto, P. H\"{a}st\"{o} and M. Ru\v{z}i\v{c}ka,
\textit{Lebesgue and Sobolev spaces with variable exponents}.
Lecture Notes in Mathematics, 2017. Springer, Heidelberg, 2011. x+509 pp. ISBN: 978-3-642-18362-1.

\bibitem{ES08} L.C. Evans and O. Savin,
\textit{$C^{1, \alpha}$ regularity for infinity harmonic functions in two dimensions}.
Calc. Var. Partial Differential Equations 32 (2008), no. 3, 325-347.

\bibitem{FRZ21} Y. Fang, V.D. R\v{a}dulescu and C. Zhang,
\textit{Regularity of solutions to degenerate fully nonlinear elliptic equations with variable exponent}.
Arxiv \url{arXiv:2103.12938}.

\bibitem{FZ20} Y. Fang and C. Zhang,
\textit{Equivalence between distributional and viscosity solutions for the double-phase equation}.
Advances in Calculus of Variations, vol. , no. , 2020, pp. 000010151520200059. \url{https://doi.org/10.1515/acv-2020-0059}.

\bibitem{IS} C. Imbert and L. Silvestre,
\textit{$C^{1, \alpha}$ regularity of solutions of degenerate fully non-linear elliptic equations}.
Adv. Math. 233 (2013), 196-206.

\bibitem{Kov99} J. Kovats,
\textit{Dini-Campanato spaces and applications to nonlinear elliptic equations}.
Electron. J. Differential Equations 1999, No. 37, 20 pp.

\bibitem{LU68} O.A. Ladyzhenskaya and N.N. Ural'tseva,
\textit{Linear and quasilinear elliptic equations}.
Translated from the Russian by Scripta Technica, Inc. Translation editor: Leon Ehrenpreis Academic Press, New York-London 1968 xviii+495 pp.

\bibitem{LL17} E. Lindgren, P. Lindqvist,
\textit{Regularity of the $p$-Poisson equation in the plane}.
J. Anal. Math. 132 (2017), 217-228.

\bibitem{LD18} W. Liu, and G. Dai,
\textit{Existence and multiplicity results for double phase problem}.
J. Differential Equations 265 (2018), no. 9, 4311–4334.

\bibitem{LMZ19} G. Lu, Q. Miao and Y. Zhou,
\textit{Viscosity solutions to inhomogeneous Aronsson's equations involving Hamiltonians $ \langle A(x)p,p \rangle$}.
Calc. Var. Partial Differential Equations 58 (2019), no. 1, Paper No. 8, 37 pp.

\bibitem{LW08} G. Lu and P. Wang,
\textit{Inhomogeneous infinity Laplace equation}.
Adv. Math. 217 (2008), no. 4, 1838–1868.

\bibitem{LW08II} G. Lu and P. Wang,
\textit{A PDE perspective of the normalized infinity Laplacian}.
Comm. Partial Differential Equations 33 (2008), no. 10-12, 1788-1817.

\bibitem{Marc89} P. Marcellini,
\textit{Regularity of minimizers of integrals of the calculus of variations with nonstandard growth conditions}.
Arch. Rational Mech. Anal. 105 (1989), no. 3, 267-284.

\bibitem{MingRad21} G. Mingione and V. R\v{a}dulescu,
\textit{Recent developments in problems with nonstandard growth and nonuniform ellipticity}.
J. Math. Anal. Appl. 501 (2021), no. 1, 125197, 41 pp.

\bibitem{PRS20} E.A. Pimentel; G.C. Rampasso and M.S. Santos,
\textit{Improved regularity for the $p$-Poisson equation}.
Nonlinearity 33 (2020), no. 6, 3050-3061.

\bibitem{RaRe15} V. R\v{a}dulescu and D. Repov\v{s},
\textit{Partial differential equations with variable exponents}.
Variational methods and qualitative analysis. Monographs and Research Notes in Mathematics. CRC Press, Boca Raton, FL, 2015. xxi+301 pp. ISBN: 978-1-4987-0341-3.

\bibitem{RaTach20} M.A. Ragusa and A. Tachikawa,
\textit{Regularity for minimizers for functionals of double phase with variable exponents}.
Adv. Nonlinear Anal. 9 (2020), no. 1, 710–728.

\bibitem{RS15} G.C. Ricarte and J.V. Silva,
\textit{Regularity up to the boundary for singularly perturbed fully nonlinear elliptic equations},
Interfaces and Free Bound. 17 (2015), 317-332.

\bibitem{RT11} G. Ricarte and E. Teixeira,
\textit{Fully nonlinear singularly perturbed equations and asymptotic free boundary}.
J. Funct. Anal. 261 (2011) 1624-1673

\bibitem{SS14} L. Silvestre and B. Sirakov,
\textit{Boundary regularity for viscosity solutions of fully nonlinear elliptic equations}.
Comm. Partial Differential Equations 39 (2014), no. 9, 1694-1717.

\bibitem{Silt18} J. Siltakoski,
\textit{Equivalence of viscosity and weak solutions for the normalized $p(x)-$Laplacian}.
Calc. Var. Partial Differential Equations 57 (2018), no. 4, Paper No. 95, 20 pp.

\bibitem{T15} E.V. Teixeira,
\textit{Hessian continuity at degenerate points in nonvariational elliptic problems}.
Int. Math. Res. Not. 2015 (2015), 6893-6906.

\bibitem{Tei16} E.V. Teixeira,
\textit{Regularity for the fully nonlinear dead-core problem}.
Math. Ann. 364 (2016), no. 3-4, 1121-1134.

\bibitem{Tru88} N.S. Trudinger,
\textit{H\"{o}lder gradient estimates for fully nonlinear elliptic equations}.
Proc. Roy. Soc. Edinburgh Sect. A 108 (1988), no. 1-2, 57-65.

\bibitem{UU83} N.N. Ural'tseva and A.B. Urdaletova,
\textit{Boundedness of gradients of generalized solutions of degenerate nonuniformly elliptic quasilinear equations}.
Vestnik Leningrad. Univ. Mat. Mekh. Astronom. 1983, vyp. 4, 50-56.

\bibitem{Weiss03} G.S. Weiss,
\textit{A singular limit arising in combustion theory: fine properties of the free boundary}.
Calc. Var. Partial Differential Equations 17 (2003), no. 3, 311–340.

\bibitem{Wint09} N. Winter,
\textit{$W^{2,p}$ and $W^{1,p}$-estimates at the boundary for solutions of fully nonlinear, uniformly elliptic equations}.
Z. Anal. Anwend. 28 (2009), no. 2, 129-164.

\bibitem{Zhi93} V.V. Zhikov,
\textit{Lavrentiev phenomenon and homogenization for some variational problems}.
C. R. Acad. Sci. Paris Sér. I Math. 316 (1993), no. 5, 435-439.

\bibitem{Zhi95} V.V. Zhikov,
\textit{On Lavrentiev's phenomenon}.
Russian J. Math. Phys. 3 (1995), no. 2, 249–269.

\bibitem{Zhi97} V.V. Zhikov,
\textit{On some variational problems},
 Russ. J. Math. Phys. 5 (1997) 105-116.

\bibitem{Zhi11} V.V. Zhikov,
\textit{On variational problems and nonlinear elliptic equations with nonstandard growth conditions},
J. Math. Sci., 173 (2011), pp. 463-570.

\end{thebibliography}
\end{document}